\documentclass{amsart}
\usepackage{amsxtra,microtype}
\usepackage{stmaryrd }
\usepackage{mathrsfs}
\usepackage{amssymb,amsmath,amsthm,stmaryrd,latexsym,wasysym}
\usepackage[all]{xy}
\usepackage{hyperref}
\usepackage{tikz}
\usepackage{tikz-cd}
\usetikzlibrary{decorations.pathmorphing}
\usetikzlibrary{arrows}
\usepackage{cancel}
\usepackage{appendix}
\theoremstyle{plain}
\usepackage{xcolor}
\usepackage{xspace}
\newtheorem{theorem}{Theorem}[section]
\newtheorem{lemma}[theorem]{Lemma}
\newtheorem{proposition}[theorem]{Proposition}
\newtheorem{corollary}[theorem]{Corollary}
\newtheorem{definition}[theorem]{Definition} \theoremstyle{definition}
\newtheorem{example}[theorem]{Example}
\newtheorem{remark}[theorem]{Remark}
\newcommand{\duer}{^*_}

\newcommand{\R}{\mathbb{R}} 
 
\newcommand{\inv}{^{-1}}
\newcommand{\N}{\mathbb{N}} 

\newcommand{\mx}{\mathfrak{X}} 
\newcommand{\dr}{\mathbf{d}}
\newcommand{\ldr}[1]{{{\pounds}}_{#1}}
\newcommand{\ip}[1]{{\mathbf{i}}_{#1}}
\newcommand{\an}[1]{\arrowvert_{#1}} 
\newcommand{\lb}{\llbracket} 
\newcommand{\rb}{\rrbracket}
\newcommand{\Beta}{\boldsymbol{\beta}}

\DeclareMathOperator{\pr}{pr}

\DeclareMathOperator{\Id}{Id}
\DeclareMathOperator{\Jac}{Jac}

\newcommand{\nsp}[2]  {%
\langle\mspace{-6.8mu}%
\langle\mspace{-6.8mu}%
\langle\mspace{-6.8mu}%
\langle\mspace{-6.8mu}%
\langle\mspace{-6.8mu}%
\langle\mspace{-6.8mu}%
\langle{#1,\,}{#2}%
\rangle%
\mspace{-6.8mu}\rangle%
\mspace{-6.8mu}\rangle%
\mspace{-6.8mu}\rangle%
\mspace{-6.8mu}\rangle%
\mspace{-6.8mu}\rangle%
\mspace{-6.8mu}\rangle}

\allowdisplaybreaks

\begin{document}
%%%%%%%%%%%%%%%%%%%%%%%%%%%%%%%%%%%%%%%%%%%%%%%%%%%%%%%%%%%%%%%%%%%%%%%%%%%
%%%%%%%%%%%%%%%%%%%%%%    Title    %%%%%%%%%%%%%%%%%%%%%%%%%%%%%%%%%%%%%%%%
\title{On LA-Courant algebroids and Poisson Lie 2-algebroids.}

%%% author one information

\author{M. Jotz Lean} \address{Mathematisches Institut, Georg-August-Universit\"at G\"ottingen.}  \email{madeleine.jotz-lean@mathematik.uni-goettingen.de}
\subjclass[2010]{Primary: 53B05, %linear and affine connections
  Secondary:
  53D17. %Poisson manifolds; Poisson groupoids and algebroids
}

\begin{abstract}
  This paper provides an alternative, much simpler, definition for
  Li-Bland's LA-Courant algebroids, or Poisson Lie 2-algebroids, in
  terms of split Lie 2-algebroids and self-dual
  2-representations. This definition generalises in a precise sense
  the characterisation of (decomposed) double Lie algebroids via
  matched pairs of 2-representations. We use the known geometric
  examples of LA-Courant algebroids in order to provide new examples
  of Poisson Lie 2-algebroids, and we explain in this general context
  Roytenberg's equivalence of Courant algebroids with symplectic Lie
  2-algebroids.

  We study further the core of an LA-Courant algebroid and we prove
  that it carries an induced degenerate Courant algebroid
  structure. In the nondegenerate case, this gives a new construction
  of a Courant algebroid from the corresponding symplectic Lie
  2-algebroid.  Finally we completely characterise VB-Dirac and
  LA-Dirac structures via simpler objects, that we compare to
  Li-Bland's pseudo-Dirac structures.
\end{abstract}
\maketitle

%\newpage

\tableofcontents

\section{Introduction}
This paper surveys the author's recent equivalences of graded
manifolds of degree $2$ with metric double vector bundles
\cite{Jotz18b} (see also \cite{delCarpio-Marek15}), of self-dual
2-representations with decomposed metric VB-algebroids \cite{Jotz18b}
(using \cite{GrMe10a}), and of split Lie 2-algebroids with decomposed
VB-Courant algebroids \cite{Jotz17b}, \cite{Li-Bland12}. We combine
all those results to provide an alternative definition for Li-Bland's
\emph{LA-Courant algebroids} \cite{Li-Bland12}, which are equivalent
to \emph{Poisson Lie 2-algebroids}.

We prove that a split Poisson Lie 2-algebroid is equivalent to the
matched pair of a self-dual 2-representation with a split Lie
2-algebroid. Here, the self-dual 2-representation is the one that is
equivalent to the Poisson structure of degree $-2$. Given a Poisson
Lie 2-algebroid, a splitting of the underlying $[2]$-manifold is
equivalent to a decomposition of the corresponding metric double
vector bundle \cite{Jotz18b}. The induced split Lie 2-algebroid is
then equivalent to a VB-Courant algebroid structure in this
decomposition \cite{Jotz17b} and the induced split Poisson structure
of degree $-2$ is equivalent to the decomposition of a (metric)
VB-algebroid structure on the other side of the metric double vector
bundle \cite{Jotz18b}. The compatibility of the graded Poisson
structure and the Lie 2-algebroid structure is equivalent to the
VB-Courant algebroid and the metric VB-algebroid defining together an
LA-Courant algebroid (short for \emph{Lie algebroid Courant
  algebroid}) \cite{Li-Bland12}. Hence, we find that a metric double
vector bundle with a VB-Courant algebroid structure and a metric
VB-algebroid structure define together an LA-Courant algebroid if and
only if, in any decomposition, the induced split Lie 2-algebroid and
the induced self-dual 2-representation build a matched pair.

The original definition of an LA-Courant algebroid \cite{Li-Bland12}
is inspired from and as technical as Mackenzie's first definition of a
double Lie algebroid \cite{Mackenzie11}. In short, a VB-Courant
algebroid with a linear Lie algebroid on its other side is an
LA-Courant algebroid if a relation defined by the Lie algebroid
structure in the tangent prolongation of the VB-Courant algebroid --
the underlying geometric structure is a triple vector bundle -- is a
Dirac structure in this tangent prolongation. In an earlier version of
this work \cite{Jotz15}, we deduced from this definition our equations in the
definition of the matched pair of a split Lie 2-algebroid with a
self-dual 2-representation. Then we found easily that those equations
were also equivalent to the homological vector field defining the Lie
2-algebroid to be a Poisson vector field.

However, working out the equations directly from Li-Bland's LA-Courant
algebroid condition is very long and technical (see the Appendix B of
\cite{Jotz15}). In this paper, we prefer therefore using Li-Bland's
equivalence \cite{Li-Bland12} in order to find the characterisation of
an LA-Courant algebroid via the matched pair.  We provide so a better
(more handy) definition of LA-Courant algebroids.  We explain again along
the way the parallels between the theory of Lie algebroids,
double Lie algebroids and 2-representations on the one hand, and
Courant algebroids, LA-Courant algebroids and Lie 2-algebroids on the
other hand. We find in particular that a matched pair of
$2$-representations \cite{GrJoMaMe18} not only defines a split Lie
2-algebroid \cite{Jotz17b}, but also a split Poisson Lie 2-algebroid
-- this is a different construction. 

We prove further that the core of an LA-Courant algebroid has an induced
degenerate Courant algebroid structure -- just like the core of a
double Lie algebroid has an induced Lie algebroid structure
\cite{Mackenzie11,GrJoMaMe18}. This allows us to explain in a
constructive manner the equivalence of symplectic Lie 2-algebroids
with Courant algebroids \cite{Roytenberg02}.

Finally, we study VB- and LA-Dirac structures in VB- and LA-Courant
algebroids, and in particular, we prove that LA-Dirac structures are
double Lie algebroids. More precisely, any double Lie algebroid can be
understood as an LA-Dirac structure in an appropriate LA-Courant
algebroid; just like any Lie algebroid can be seen as a Dirac
structure in the induced Courant algebroid. This shows that the 7
equations defining a matched pair of 2-representations
\cite{GrJoMaMe18} can in fact be deduced from the 5 equations defining
the matched pair of a split Lie 2-algebroid with a self-dual
2-representation.  For completeness, we explain as well how Li-Bland's
pseudo-Dirac structures \cite{Li-Bland14} fit in our description of
LA-Dirac structures in the tangent prolongation of a Courant
algebroid.

\subsection*{Outline of the paper}
In Section \ref{prereqs}, we recall some general notions and
conventions around dull brackets, Dorfman connections, Courant
algebroids, graded manifolds and double vector bundles.  In Section
\ref{known_eqs} we recall the equivalence of metric double vector
bundles with $[2]$--manifolds, of metric VB-algebroids with Poisson
[2]-manifolds and with self-dual 2-representations when decomposed.
We recall also the equivalence of decompositions of VB-Courant
algebroids with split Lie 2-algebroids.  In Section \ref{sec:LA-Cour}
we give the definition of the matched pair of a split Lie 2-algebroid
with a self-dual 2-representation, and we prove our main theorem. In
Section \ref{core_LA_CA} we study the core of an LA-Courant algebroid and in
Section \ref{sec:dirac} we describe VB- and LA-Dirac structures in
decompositions. The text is illustrated with several examples that
were prepared in \cite{Jotz18a, Jotz17b}.

\section{Prerequisites, notation and conventions}\label{prereqs}
We recall in this section some necessarily background, and we set our
notation convention.
\subsection{General conventions}
We write $p_M\colon TM\to M$, $q_E\colon E\to M$ for vector bundle
maps. For a vector bundle $Q\to M$ we often identify without further
mentioning the vector bundle $(Q^*)^*$ with $Q$ via the canonical
isomorphism. We write $\langle\cdot\,,\cdot\rangle$ for the canonical
pairing of a vector bundle with its dual;
i.e.~$\langle q_m,\tau_m\rangle=\tau_m(q_m)$ for $q_m\in Q$ and
$\tau_m\in Q^*$. We use several different pairings; in general, which
pairing is used is clear from its arguments.  Given a section
$\varepsilon$ of $E^*$, we always write
$\ell_\varepsilon\colon E\to \R$ for the linear function associated to
it, i.e.~the function defined by
$e_m\mapsto \langle \varepsilon(m), e_m\rangle$ for all $e_m\in E$.

Let $M$ be a smooth manifold. We denote by $\mx(M)$ and $\Omega^1(M)$
the sheaves of local smooth sections of the tangent and the cotangent
bundle, respectively. For an arbitrary vector bundle $E\to M$, the
sheaf of local sections of $E$ will be written $\Gamma(E)$.  

\subsection{Dull brackets, Dorfman connections, Courant algebroids}
Let $q_Q\colon Q\to M$ be an anchored vector bundle, with anchor
$\rho_Q\colon Q\to TM$. A \textbf{dull bracket} on sections of $Q$ is
here an $\R$-bilinear, skew-symmetric bracket
$\lb\cdot\,,\cdot\rb\colon\Gamma(Q)\times\Gamma(Q)\to\Gamma(Q)$
that is compatible with the anchor, $\rho_Q\lb q_1, q_2\rb=[\rho_Q(q_1), \rho_Q(q_2)]$
and that satisfies the Leibniz identity
\begin{equation}\label{leibniz}
\lb q_1, f q_2\rb =\rho_Q(q_1)(f)q_2+f\lb q_1, q_2\rb
\end{equation}
for all $f\in C^\infty(M)$ and $q_1,q_2\in\Gamma(Q)$.

Dualising the dull bracket in the sense of derivations, 
we get a \textbf{Dorfman connection} $\Delta\colon
\Gamma(Q)\times\Gamma(Q^*)\to\Gamma(Q^*)$:
\begin{equation}\label{dualdd} \rho_Q(q)\langle q',\tau\rangle=\langle\lb q,q'\rb,
\tau\rangle+\langle q', \Delta_q\tau\rangle
\end{equation}
for all $q,q'\in\Gamma(Q)$ and $\tau\in\Gamma(Q^*)$. The Leibniz
identities in the two entries of the dull bracket give 
$\Delta_q(f\tau)=f\Delta_q\tau+\rho_Q(q)(f)\tau$ and 
$\Delta_{fq}\tau=f\Delta_q\tau+\langle q, \tau\rangle\rho_Q^*\dr f$
for $q\in\Gamma(Q)$, $\tau\in\Gamma(Q^*)$ and $f\in C^\infty(M)$. The
compatibility of the bracket with the anchor reads 
\begin{equation}\label{dorfman_on_exact}
\Delta_q(\rho_Q^*\dr f)=\rho_Q^*\dr(\rho_Q(q)(f))
\end{equation}
 $q\in\Gamma(Q)$ and $f\in C^\infty(M)$. 

Given an $\R$-bilinear skew-symmetric bracket $[\cdot\,,\cdot]$ on sections of a
vector bundle $E\to M$, its \textbf{Jacobiator} is the map
$\operatorname{Jac}_{[\cdot\,,\cdot]}\colon\Gamma(E)\times\Gamma(E)\times\Gamma(E)\to\Gamma(E)$,
\[\operatorname{Jac}_{[\cdot\,,\cdot]}(e_1,e_2,e_3)=[[e_1,e_2],e_3]+[[e_2,e_3],e_1]+[[e_3,e_1],e_2].
\]
The Jacobiator of a dull bracket satisfies
\begin{equation}\label{Jac_Dorfman}
\operatorname{Jac}_{\lb\cdot\,,\cdot\rb}(q_1,q_2,q_3)=R_\Delta(q_1,q_2)^*q_3,
\end{equation}
where $R_\Delta\in\Omega^2(Q,\operatorname{End}(Q^*))$ is the
curvature of the dual Dorfman connection:
$R_\Delta(q_1,q_2)\tau=\Delta_{q_1}\Delta_{q_2}\tau-\Delta_{q_2}\Delta_{q_1}\tau-\Delta_{\lb q_1,q_2\rb}\tau$.

A dull algebroid $(Q,\rho_Q,\lb\cdot\,,\cdot\rb)$ as above defines as
usual a Koszul
differential operator
$d_Q\colon \Omega^\bullet(Q)\to\Omega^{\bullet+1}(Q)$ with
$\dr_Q(\omega\wedge\eta)=\dr_Q\omega\wedge
\eta+(-1)^{|\eta|}\omega\wedge\dr_Q\eta$
for all $\omega,\eta\in \Omega^\bullet(Q)$.

Let $E$ be a vector bundle over the same base manifold $M$.  A
\textbf{linear $Q$-connection on $E$} is a linear connection
$\nabla\colon\Gamma(Q)\times\Gamma(E)\to\Gamma(E)$. The curvature
$R_\nabla\in\Omega^2(Q,\operatorname{End}(E))$ of $\nabla$ is defined
as usual. The dull bracket and the connection define together a
differential operator
$\dr_\nabla \colon \Omega^\bullet(Q,E)\to\Omega^{\bullet+1}(Q,E)$:
$\dr_\nabla e=\nabla_\cdot e$ for $e\in\Gamma(E)$, and 
$\dr_\nabla(\omega\wedge\eta)=\dr_Q\omega\wedge
\eta+(-1)^{|\eta|}\omega\wedge\dr_\nabla\eta$ for $\omega\in
\Omega^\bullet(Q)$
and $\eta\in\Omega^\bullet(Q,E)$.

\medskip
Consider an anchored vector bundle $(\mathsf E\to M, \rho)$.  Assume that $\mathsf E$ is paired with
itself via a pairing $\langle\cdot\,,\cdot\rangle\colon \mathsf
E\times_M \mathsf E\to \R$ and that there exists a
map $\mathcal D\colon C^\infty(M)\to \Gamma(\mathsf E)$ such that
$\langle \mathcal Df, e\rangle=\rho(e)(f)$ for all
$f\in C^\infty(M)$.
Then $\mathsf E\to M$ is a \textbf{degenerate Courant algebroid} over the manifold $M$ if $\mathsf E$ is in addition
equipped with an $\R$-bilinear bracket $\lb\cdot\,,\cdot\rb$ on the
smooth sections $\Gamma(\mathsf E)$ such that the following conditions
are satisfied:
\begin{enumerate}\label{CA_def}
\item[(CA1)] $\lb e_1, \lb e_2, e_3\rb\rb= \lb \lb e_1, e_2\rb, e_3\rb+ \lb
  e_2, \lb e_1, e_3\rb\rb$,
\item[(CA2)] $\rho(e_1 )\langle e_2, e_3\rangle= \langle\lb e_1,
  e_2\rb, e_3\rangle + \langle e_2, \lb e_1 , e_3\rb\rangle$,
\item[(CA3)] $\lb e_1, e_2\rb+\lb e_2, e_1\rb =\mathcal D\langle e_1 ,
  e_2\rangle$,
\item[(CA4)]     $\rho\lb e_1, e_2\rb = [\rho(e_1), \rho(e_2)]$,
\item[(CA5)]      $\lb e_1, f e_2\rb= f \lb e_1 , e_2\rb+ (\rho(e_1 )f )e_2$
\end{enumerate}
for all $e_1, e_2, e_3\in\Gamma(\mathsf E)$ and $f\in C^\infty(M)$.
If the pairing $\langle\cdot\,,\cdot\rangle$ is nondegenerate, then
$(\mathsf E\to M, \rho, \langle\cdot\,,\cdot\rangle,
\lb\cdot\,,\cdot\rb)$ is a \textbf{Courant algebroid}
\cite{LiWeXu97,Roytenberg99} and Conditions (CA4) and (CA5) follow
then from (CA1), (CA2) and (CA3) (see \cite{Uchino02} and also
\cite{Jotz18a} for a quicker proof).

\subsection{$\N$-Manifolds of degree 2}
\medskip An \textbf{ $\N$-manifold} $\mathcal M$ of degree $2$ and
dimension $(p; r_1,r_2)$ is a smooth manifold $M$ of dimension $p$
together a sheaf $C^\infty(\mathcal M)$ of $\N$-graded, graded
commutative, associative, unital $C^\infty(M)$-algebras over $M$, that
is locally freely generated by $r_1+r_2$ elements
$\xi_1^{1},\ldots,\xi_1^{r_1}$, $\xi_2^1,\ldots,\xi_2^{r_2}$ with
$\xi_i^j$ of degree $i$ for $i=1,2$ and $j\in\{1,\ldots,r_i\}$.  We
write ``$[2]$-manifold'' for ``$\N$-manifold of degree $2$''. A
morphism of $\N$-manifolds $\mu\colon\mathcal N\to \mathcal M$ over a
smooth map $\mu_0\colon N\to M$ of the underlying smooth manifolds is
a morphism
$\mu^\star\colon C^\infty(\mathcal M)\to C^\infty(\mathcal N)$ of
sheaves of graded algebras over
$\mu_0^*\colon C^\infty(M)\to C^\infty(N)$.

Let $E_1$ and $E_{2}$ be smooth vector bundles of finite ranks
$r_1,r_2$ over $M$ and assign the degree $i$ to the fiber coordinates
of $E_{i}$, for each $i=1,\ldots,n$.  The direct sum
$E=E_{1}\oplus E_{2}$ is a graded vector bundle with grading
concentrated in degrees $1$ and $2$.  The $[2]$-manifold
$E_{1}[-1]\oplus E_{2}[-2]$ has the elements of local frames of
${E_{i}^*}$ as local generators of degree $i$, for $i=1,2$, and so
dimension $(p;r_1,r_2)$.  A $[2]$-manifold
$\mathcal M=E_{1}[-1]\oplus E_{2}[-2]$ defined in this manner by a
graded vector bundle is called a \textbf{split $[2]$-manifold}. In
other words, we have $C^\infty(\mathcal M)^0=C^\infty(M)$,
$C^\infty(\mathcal M)^1=\Gamma(E_{1}^*)$ and
$C^\infty(\mathcal M)^2=\Gamma(E_{2}^*\oplus \wedge^2E_{1}^*)$. A
morphism
$\mu\colon F_{1}[-1]\oplus F_{2}[-2]\to E_{1}[-1]\oplus E_{2}[-2]$ of
split $[2]$-manifolds over the bases $M$ and $N$, respectively,
consists of a smooth map $\mu_0\colon N\to M$, three vector bundle
morphisms $\mu_1\colon F_{1}\to E_{1}$, $\mu_2\colon F_{2}\to E_{2}$
and $\mu_{12}\colon \wedge^2F_{1}\to E_{2}$ over $\mu_0$. The map
$\mu^\star$ sends a degree $1$ function $\xi\in\Gamma(E_{1}^*)$
to\footnote{$\mu_1^\star\xi\in\Gamma(F_1^*)$ is defined by
  $\langle
  (\mu_1^\star\xi)_p,e_p\rangle=\langle\xi(\mu_0(p)),\mu_1(e_p)\rangle$
  for all $e_p\in F_1$.}  ${\mu_1}^\star\xi\in\Gamma(F_{1}^*)$ and a
degree $2$-function $\xi\in\Gamma(E_{2}^*)$ to
${\mu_2}^\star\xi+\mu_{12}^\star\xi\in\Gamma(F_{2}^*\oplus \wedge^2
F_{1}^*)$.

\subsection{Double Lie algebroids and matched pairs of $2$-representations}
\label{matched_pair_2_rep_sec}

We refer to Section 2.2 of \cite{Jotz18b} for the definition of a
double vector bundle, and for the necessary background on their linear
and core sections, and on their linear splittings and dualisations.
Sections 2.3--2.5 of \cite{Jotz18b} recall the definition of a
VB-algebroid, and also the equivalence of $2$-term representations up
to homotopy (called here ``$2$-representations'' for short) with
linear decompositions of VB-algebroids \cite{GrMe10a}. The notation
that we use here is the same as in \cite{Jotz18b}. Of course, we also
refer to \cite{Pradines77,Mackenzie05,GrMe10a} for more details on
double vector bundles and VB-algebroids.

In this section we only recall the correspondence of decompositions of
double Lie algebroids with matched pairs of $2$-representations.

\medskip

If $(D,A,B,M)$ is a VB-algebroid with Lie algebroid structures on
$D\to B$ and $A\to M$, then the dual vector bundle $D\duer B\to B$ has
a \emph{Lie-Poisson structure} (a linear Poisson structure), and the
structure on $D\duer B$ is also Lie-Poisson with respect to
$D\duer B\to C^*$ \cite[3.4]{Mackenzie11}. Dualising this bundle gives
a Lie algebroid structure on $(D\duer B)\duer{C^*}\to C^*$. This
equips the double vector bundle $((D\duer B)\duer{C^*}; C^*,A;M)$ with
a VB-algebroid structure. Using the isomorphism defined by the
non-degenerate pairing
$-\nsp{\cdot}{\cdot}\colon D\duer A\times_{C^*}D\duer B\to \mathbb R$,
(see \cite{Mackenzie05} and \cite[\S 2.2.4]{Jotz18b} for a summary and
our sign convention), the double vector bundle
$(D\duer A\to C^*;A\to M)$ is also a VB-algebroid. In the same manner,
if $(D\to A, B\to M)$ is a VB-algebroid then we use the non-degenerate pairing
$\nsp{\cdot}{\cdot}\colon D\duer A\times_{C^*}D\duer B\to \mathbb R$ to get a VB-algebroid structure on
$(D\duer B\to C^*;B\to M)$.

Let $\Sigma\colon A\times_MB\to D$ be a linear splitting of $D$ and
denote by $(\nabla^B,\nabla^C,R_A)$ the induced 2-representation of
the Lie algebroid $A$ on $\partial_B\colon C\to B$ (see
\cite{GrMe10a}; this is also recalled in Section 2.5 of
\cite{Jotz18b}). The linear splitting $\Sigma$ induces a linear
splitting $\Sigma^\star\colon A\times_M C^*\to D\duer A$ of $D\duer
A$.  The 2-representation of $A$ that is associated to this splitting is
then $({\nabla^C}^*,{\nabla^B}^*,-R_A^*) $ on the complex
$\partial_B^*\colon B^*\to C^*$. This is proved in the appendix of
\cite{DrJoOr15}.

\medskip
A \textbf{double Lie algebroid} \cite{Mackenzie11} is a double vector
bundle $(D,A,B,M)$ with core $C$, and with Lie algebroid
structures on each of $A\to M$, $B\to M$, $D\to A$ and $D\to B$ such
that each pair of parallel Lie algebroids gives $D$ the structure of a
VB-algebroid, and such that
the pair $(D\duer A, D\duer B)$ with the induced Lie algebroid
structures on base $C^*$ and the pairing $\nsp{\cdot}{\cdot}$, is a
Lie bialgebroid.

Consider a double vector bundle $(D;A,B;M)$ with core $C$ and a VB-Lie
algebroid structure on each of its sides.  After the choice of a
splitting $\Sigma\colon A\times_M B\to D$, the Lie algebroid
structures on the two sides of $D$ are described by two
$2$-representations \cite{GrMe10a}.  We prove in \cite{GrJoMaMe18} that
$(D\duer A, D\duer B)$ is a Lie bialgebroid over $C^*$ if and only if,
for any splitting of $D$, the two induced 2-representations form a
matched pair as in the following definition \cite{GrJoMaMe18}.

\begin{definition}\label{matched_pair_2_rep}
  Let $(A\to M, \rho_A, [\cdot\,,\cdot])$ and $(B\to M, \rho_B,
  [\cdot\,,\cdot])$ be two Lie algebroids and assume that $A$ acts on
  $\partial_B\colon C\to B$ up to homotopy via
  $(\nabla^{B},\nabla^{C}, R_{A})$ and $B$ acts on
  $\partial_A\colon C\to A$ up to homotopy via
  $(\nabla^{A},\nabla^{C}, R_{B})$\footnote{For the sake of
    simplicity, we write in this definition $\nabla$ for all the four
    connections. It will always be clear from the indexes which
    connection is meant. We write $\nabla^A$ for the $A$-connection
    induced by $\nabla^{AB}$ and $\nabla^{AC}$ on $\wedge^2 B^*\otimes
    C$ and $\nabla^B$ for the $B$-connection induced on $\wedge^2
    A^*\otimes C$.  }.  Then we say that the two representations up to
  homotopy form a matched pair if
\begin{enumerate}
\item[(m1)]
  $\nabla_{\partial_Ac_1}c_2-\nabla_{\partial_Bc_2}c_1=-(\nabla_{\partial_Ac_2}c_1-\nabla_{\partial_Bc_1}c_2)$,
\item[(m2)] $[a,\partial_Ac]=\partial_A(\nabla_ac)-\nabla_{\partial_Bc}a$,
\item[(m3)] $[b,\partial_Bc]=\partial_B(\nabla_bc)-\nabla_{\partial_Ac}b$,
\item[(m4)]
$\nabla_b\nabla_ac-\nabla_a\nabla_bc-\nabla_{\nabla_ba}c+\nabla_{\nabla_ab}c=
R_{B}(b,\partial_Bc)a-R_{A}(a,\partial_Ac)b$,
\item[(m5)]
  $\partial_A(R_{A}(a_1,a_2)b)=-\nabla_b[a_1,a_2]+[\nabla_ba_1,a_2]+[a_1,\nabla_ba_2]+\nabla_{\nabla_{a_2}b}a_1-\nabla_{\nabla_{a_1}b}a_2$,
\item[(m6)]
  $\partial_B(R_{B}(b_1,b_2)a)=-\nabla_a[b_1,b_2]+[\nabla_ab_1,b_2]+[b_1,\nabla_ab_2]+\nabla_{\nabla_{b_2}a}b_1-\nabla_{\nabla_{b_1}a}b_2$,
\end{enumerate}
for all $a,a_1,a_2\in\Gamma(A)$, $b,b_1,b_2\in\Gamma(B)$ and
$c,c_1,c_2\in\Gamma(C)$, and
\begin{enumerate}\setcounter{enumi}{6}
\item[(m7)] $\dr_{\nabla^A}R_{B}=\dr_{\nabla^B}R_{A}\in \Omega^2(A,
  \wedge^2B^*\otimes C)=\Omega^2(B,\wedge^2 A^*\otimes C)$, where
  $R_{B}$ is seen as an element of $\Omega^1(A, \wedge^2B^*\otimes
  C)$ and $R_{A}$ as an element of $\Omega^1(B, \wedge^2A^*\otimes
  C)$.
\end{enumerate}
\end{definition}

\subsection{The equivalence of $[2]$-manifolds with metric double vector bundles}\label{recall_n}\label{known_eqs} 
We begin by summarising the correspondence found in \cite{Jotz18b}
between double vector bundles endowed with a linear metric, and
$\N$-manifolds of degree $2$. \medskip

A \textbf{metric double vector bundle} is a double vector bundle
$(\mathbb E, Q, B, M)$ with core $Q^*$, equipped with a \textbf{linear symmetric
  non-degenerate pairing $\mathbb E\times_B\mathbb E\to \R$},
i.e.~such that 
\begin{enumerate}
\item $\langle \tau_1^\dagger, \tau_2^\dagger\rangle=0$ for
  $\tau_1,\tau_2\in\Gamma(Q^*)$,
\item $\langle \chi, \tau^\dagger\rangle=q_B^*\langle q,\tau\rangle$
  for $\chi\in\Gamma_B^l(\mathbb E)$ linear over $q\in\Gamma(Q)$ and
  $\tau\in\Gamma(Q^*)$ and
\item $\langle\chi_1,\chi_2\rangle$ is a linear function on $B$ for
  $\chi_1,\chi_2\in \Gamma_B^l(\mathbb E)$.
\end{enumerate}
Note that the \textbf{opposite} $(\overline{\mathbb E},Q,B,M)$ of
a metric double vector bundle $(\mathbb E,B,Q,M)$ is the
metric double vector bundle with
$\langle\cdot\,,\cdot\rangle_{\overline{\mathbb
    E}}=-\langle\cdot\,,\cdot\rangle_{\mathbb E}$.

A linear
splitting $\Sigma\colon Q\times_MB\to \mathbb E$ is said to be
\textbf{Lagrangian} if its image is maximal isotropic in $\mathbb E\to
B$.  The corresponding horizontal lifts
$\sigma_Q\colon\Gamma(Q)\to\Gamma^l_B(\mathbb E)$ and
$\sigma_B\colon\Gamma(B)\to\Gamma^l_Q(\mathbb E)$
  are then also said to be
  \textbf{Lagrangian}. By definition, a horizontal lift $\sigma_Q\colon
\Gamma(Q)\to \Gamma^l_B(\mathbb E)$ is Lagrangian if and only if
$\langle \sigma_Q(q_1), \sigma_Q(q_2)\rangle=0$ for all $q_1,q_2\in
\Gamma(Q)$.
Showing the existence of a Lagrangian splitting of $\mathbb E$ is
relatively easy \cite{Jotz18b}:
Note that a general linear decomposition $\Sigma$ of a metric double vector
bundle defines as follows a section $\Lambda$ of $S^2(Q)\otimes B^*$:
\begin{equation}\label{lambda}
\langle\sigma_Q(q_1),\sigma_Q(q_2)\rangle =\ell_{\Lambda(q_1,q_2)}
\end{equation}
for all $q_1,q_2\in\Gamma(Q)$.
In particular, $\Lambda(q,\cdot)\colon Q\to B^*$ is a morphism of vector
bundles for each $q\in\Gamma(Q)$.  Define a new horizontal lift
$\sigma_Q'\colon\Gamma(Q)\to \Gamma^l_B(\mathbb E)$ by
\begin{equation}\label{new_lift}
\sigma_Q'(q)=\sigma_Q(q)-\frac{1}{2}\widetilde{\Lambda(q,\cdot)^*}
\end{equation} for
all $q\in\Gamma(Q)$. 
It is easy to check that the corresponding linear decomposition $\Sigma'$ is Lagrangian.

Further, if $\Sigma^1$ and
$\Sigma^2\colon Q\times_M B\to \mathbb E$ are Lagrangian, then the
change of splitting $\phi_{12}\in\Gamma(Q^*\otimes Q^*\otimes B^*)$
defined by $\Sigma^2(q,b)=\Sigma^1(q,b)+\widetilde{\phi(q,b)}$ for all
$(q,b)\in Q\times_M B$, is a section of
$Q^*\wedge Q^* \otimes B^*$.

\begin{example}\label{metric_connections}
  
  Let $E\to M$ be a vector bundle endowed
  with a symmetric non-degenerate pairing
  $\langle\cdot\,,\cdot\rangle\colon E\times_M E\to \R$ (a
  \emph{metric vector bundle}).  Then
  $E\simeq E^*$ and the tangent prolongation $TE$ is a metric double vector
  bundle $(TE,E,TM,M)$ with pairing $TE\times_{TM}TE\to \R$ the
  tangent of the pairing $E\times_M E\to \R$. In particular, we have $\langle Te_1, Te_2\rangle_{TE}=\ell_{\dr\langle e_1,e_2\rangle}$, 
$\langle Te_1, e_2^\dagger\rangle_{TE}=p_M^* \langle e_1,e_2\rangle$ and $\langle e_1^\dagger, e_2^\dagger\rangle_{TE}=0$ for
$e_1,e_2\in\Gamma(E)$.

Recall from \cite[Example 3.11]{Jotz18b} that linear splittings of
$TE$ are equivalent to linear connections $\nabla\colon \mx(M)\times\Gamma(E)\to\Gamma(E)$.  The Lagrangian
splittings of $TE$ are exactly the linear splittings that correspond
to \textbf{metric} connections, i.e.~linear connections
$\nabla\colon \mx(M)\times\Gamma(E)\to\Gamma(E)$ that preserve the
metric:
$\langle\nabla_\cdot e_1, e_2\rangle+\langle e_1,\nabla_\cdot
e_2\rangle=\dr\langle e_1, e_2\rangle$ for $e_1,e_2\in\Gamma(E)$.
\end{example}

Let $(\mathbb E, B, Q, M)$ be a metric double vector bundle. Define $\mathcal C(\mathbb
E)\subseteq \Gamma_Q^l(\mathbb E)$ as the $C^\infty(M)$-submodule of 
linear sections with isotropic image in $\mathbb E$.
After the choice of a Lagrangian splitting $\Sigma\colon Q\times_MB\to
\mathbb E$, $\mathcal C(\mathbb E)$ can be written
$\mathcal C(\mathbb
E):=\sigma_B(\Gamma(B))+
\{\tilde\omega\mid \omega\in\Gamma(Q^*\wedge Q^*)\}$.
This shows that $\mathcal
C(\mathbb E)$ together with $\Gamma^c_Q(\mathbb E)\simeq\Gamma(Q^*)$
span $\mathbb E$ as a vector bundle over $Q$.
\medskip

We define\footnote{A metric double vector bundle $(\mathbb E, B,Q,M)$
  is dual (over $Q$) to an \emph{involutive double vector bundle}
  \cite{Jotz18b}. A morphism $\Omega\colon \mathbb F\to\mathbb E$ of
  metric double vector bundles is defined as a relation
  $\Omega\subseteq \overline{\mathbb F}\times\mathbb E$ that is the
  dual of a morphism of involutive double vector bundles
  $\omega\colon \mathbb F\duer P\to \mathbb E\duer Q$
  \cite{Jotz18b}. The characterisation given here is proved in
  \cite{Jotz18b}.} a morphism $\Omega\colon \mathbb F\to\mathbb E$ of
metric double vector bundles as a pair of maps
$\omega^\star \colon \mathcal C(\mathbb E)\to \mathcal C(\mathbb F)$,
$\omega_P^\star \colon \Gamma(Q^*)\to \Gamma(P^*)$ together with a
smooth map $\omega_0 \colon N\to M$ such that
\begin{enumerate}
\item $\omega^\star\left(\widetilde{\tau_1\wedge\tau_2}\right)=\widetilde{\omega_P^\star\tau_1\wedge\omega_P^\star\tau_2}$,
\item
  $\omega^\star(q_Q^*f\cdot\chi)=q_P^*(\omega_0^*f)\cdot\omega^\star(\chi)$
  and
\item $\omega_P^\star(f\cdot \tau)=\omega_0^*f\cdot \omega_P^\star\tau$
\end{enumerate}
for all $\tau,\tau_1,\tau_2\in\Gamma(Q^*)$, $f\in C^\infty(M)$ and
$\chi \in \mathcal C(\mathbb E)$.
A morphism $\Omega\colon
  Q_1\times_{M_1}Q_1\times_{M_1}B_1^*\to
  Q_2\times_{M_2}Q_2\times_{M_2}B_2^* $ of decomposed metric double
  vector bundles is then described by $\omega_Q\colon Q_1\to Q_2$,
  $\omega_{B^*}\colon B_1^*\to B_2^*$ and $\omega_{12}\colon Q_1\wedge
  Q_1\to B_2^*$, all morphisms of vector bundles over a smooth map
  $\omega_0\colon M_1\to M_2$. 
For $b\in\Gamma(B_1)$ the isotropic section $b^l\in\Gamma_{Q_2}^l(
  B_2\times_{M_2}Q_2\times_{M_2}Q_2^*)$,
  $b^l(q_m)=(b(m),q_m,0_m^{Q_2^*})$, is sent by $\omega^\star$ to the
  isotropic section
  $(\omega_B^\star(b))^l+\widetilde{\omega_{12}^\star(b)}\in\Gamma_{Q_1}^l(B_1\times_{M_1}Q_1\times_{M_1}Q_1^*)$.

\medskip

We write $\operatorname{MDVB}$ for the obtained category of metric
double vector bundles.  In \cite{Jotz18b} we established an
equivalence between the category of involutive double vector bundles
and the category of $[2]$-manifolds. We also proved there that there
is a (contravariant) dualisation equivalence of the categry of
involutive double vector bundles with $\operatorname{MDVB}$. This yields the following theorem.

\begin{theorem}[\cite{Jotz18b}]\label{main_crucial}
  There is a (contravariant) equivalence between the category of
  $[2]$-manifolds and the category of metric double vector
  bundles.
\end{theorem}
 This equivalence establishes in
particular an equivalence between split $[2]$-manifold
$\mathcal M=Q[-1]\oplus B^*[-2]$ and the decomposed metric double
vector bundle $(Q\times_M B\times_MQ^*,B,Q,M)$ with the obvious linear
metric over $B$.

We quickly describe the functors between the two categories.  To
construct the geometrisation functor
$\mathcal G\colon [2]\rm{-Man}\to \operatorname{MDVB}$, take a
$[2]$-manifold and considers its local trivialisations.  Changes of
local trivialisation define a set of cocycle conditions, that
correspond exactly to cocycle conditions for a double vector bundle
atlas. The local trivialisations can hence be collated to a double
vector bundle, which naturally inherits an involution. 
See \cite{Jotz18b} for more details, and remark that this
construction is as simple as the construction of a vector
bundle from a locally free and finitely generated sheaf of
$C^\infty(M)$-modules.
Conversely, the algebraisation functor $\mathcal A$ sends a metric
double vector bundle $\mathbb E$ to the $[2]$-manifold defined as
follows: the functions of degree $1$ are the sections of
$\Gamma_Q^c(\mathbb E)\simeq\Gamma(Q^*)$, and the functions of degree
$2$ are the elements of $\mathcal C(\mathbb E)$. The multiplication of
two core sections $\tau_1,\tau_2\in\Gamma(Q^*)$ is the core-linear
section $\widetilde{\tau_1\wedge\tau_2}\in\mathcal C(\mathbb E)$.

\subsubsection{Metric VB-algebroids and Poisson [2]-manifolds}
\label{subsect:VBa}\label{metricVBA}
The correspondence above of split $2$-manifolds with decomposed metric
double vector bundles induces a correspondence of split Poisson
$[2]$-manifolds with decomposed metric VB-algebroids. This
bijection extends to an equivalence of Poisson $[2]$-manifolds with
metric VB-algebroids \cite{Jotz18b}, but we only need the split
objects here.

\medskip
Consider a metric double vector bundle $(\mathbb E, B, Q,M)$ with a
linear Lie algebroid on $\mathbb E\to Q$ over a Lie algebroid
structure on $B\to M$.  Then $\mathbb E$ is a \textbf{metric
  VB-algebroid} \cite{Jotz18b} if  the bijection
\begin{equation}\label{beta}
\begin{xy}
\xymatrix{\mathbb{E}\ar[rr]^{\pi_B}\ar[dd]_{\pi_Q}&&B\ar[dd]^{q_B}&
& &\mathbb{E}\duer B\ar[rr]\ar[dd]&&B\ar[dd]\\
&Q^*\ar[dr]&   &\ar[r]^{\Beta}& &&Q^*\ar[dr]&\\
Q\ar[rr]_{q_Q}&&M&&&Q^{**}\simeq Q\ar[rr]&&M }
\end{xy}
\end{equation}
defined by the pairing is a morphism of VB-algebroids, where
$\mathbb E\duer B\to Q^{**}$ is equipped with the dual Lie algebroid
structure to the one on $\mathbb E\to Q$ \cite{Jotz18b}.

Take a decomposed metric double vector bundle
$\mathbb E=Q\times_MB\times_MQ^*$ with a linear Lie algebroid structure on
$Q\times_MB\times_MQ^*\to Q$ over a Lie algebroid $B\to M$. Let
$(\nabla^Q,\nabla^{Q^*},R)$ be the corresponding 2-representation of the Lie
algebroid $B$ on $\partial_Q\colon Q^*\to Q$.  Then
$Q\times_MB\times_MQ^*$ is a \textbf{decomposed metric VB-algebroid}
if and only if the 2-representation is dual to itself \cite{Jotz18b}:
\[\left(\nabla^Q\right)^*=\nabla^{Q^*} \quad \partial_Q^*=\partial_Q,
  \quad R(b_1,b_2)^*=-R(b_1,b_2) \quad \forall b_1,b_2\in\Gamma(B).
\]
That is, the dual decomposed VB-algebroid $\mathbb E\duer
B=Q^{**}\times_MB\times_MQ^*\to Q^{**}$ is canonically isomorphic to
$\mathbb E\to Q$ via the canonical isomorphism of $Q$ with $Q^{**}$
\cite{Jotz18b}.
Note that $R(b_1,b_2)^*=-R(b_1,b_2)$ for all $b_1,b_2\in\Gamma(B)$ is
equivalent to $R\in\Omega^2(B,Q^*\wedge Q^*)$.
\medskip

Note that a Poisson bracket of degree $-2$ on a $[2]$-manifold
$\mathcal M$ is an $\R$-bilinear map $\{\cdot\,,\cdot\}\colon
C^\infty(\mathcal M)\times C^\infty(\mathcal M)\to C^\infty(\mathcal
M)$ of the graded sheaves of functions, such that\linebreak $|\{\xi,\eta\}|=|\xi|+|\eta|-2$
for homogeneous elements $\xi,\eta\in C^\infty_{\mathcal M}(U)$. The
bracket is graded skew-symmetric;
$\{\xi,\eta\}=-(-1)^{|\xi|\,|\eta|}\{\eta,\xi\}$ and satisfies the
graded Leibniz and Jacobi identities
 \begin{equation}\label{graded_leibniz}
\{\xi_1,\xi_2\cdot\xi_3\}=\{\xi_1,\xi_2\}\cdot\xi_3+(-1)^{|\xi_1|\,|\xi_2|}\xi_2\cdot\{\xi_1,\xi_3\}
\end{equation}
and
\begin{equation}\label{graded_jacobi}
\{\xi_1,\{\xi_2,\xi_3\}\}=\{\{\xi_1,\xi_2\},\xi_3\}+(-1)^{|\xi_1|\,|\xi_2|}\{\xi_2,\{\xi_1,\xi_3\}\}
\end{equation}
for homogeneous $\xi_1,\xi_2,\xi_3\in C^\infty_{\mathcal M}(U)$.
A morphism $\mu\colon\mathcal N\to\mathcal M$ of Poisson
$[2]$-manifolds
satisfies $\mu^\star\{\xi_1,\xi_2\}=\{\mu^\star\xi_1,\mu^\star\xi_2\}$
for all $\xi_1,\xi_2\in C^\infty_{\mathcal M}(U)$, $U$ open in $M$.

Via the identification of the underlying metric double vector bundle
$\mathbb E=Q\times_MB\times_MQ^*$ 
with the $[2]$-manifold $Q[-1]\oplus B^*[-2]$, the metric VB-algebroid
structure is equivalent to a degree $-2$ Poisson structure on
$Q[-1]\oplus B^*[-2]$ (see \cite{Jotz18b}):
\begin{equation}\label{2-poisson}
\begin{split}
\{f_1,f_2\}&=0, \quad \{\tau,f\}=0,\quad
\{\tau_1,\tau_2\}=\langle \partial_Q\tau_1,\tau_2\rangle, \\
\{b,f\}&=\rho_B(b)(f),\quad \{b,\tau\}=\nabla^*_b\tau,\quad \{b_1,b_2\}=[b_1,b_2]-R(b_1,b_2).
\end{split}
\end{equation}
This identification is compatible with changes of splittings of the
$[2]$-manifolds and changes of decomposition of metric
VB-algebroids: The category of Poisson $[2]$-manifolds is equivalent
to the category of metric VB-algebroids \cite{Jotz18b}.

\begin{example}\label{self-dual-metric-conn}
Consider a metric vector bundle $E\to M$ and a metric connection 
$\nabla\colon\mx(M)\times\Gamma(E)\to\Gamma(E)$.
Since $\nabla=\nabla^*$ when $E^*$ is identified with $E$ via the
non-degenerate pairing, the $2$-representation $(\Id_E\colon E\to E,
\nabla, \nabla, R_\nabla)$ is self-dual.
The metric VB-algebroid structure on $TE\to E$ is just the standard
Lie algebroid structure on the tangent bundle of $E$. See
\cite{Jotz18b} for more details.
\end{example}

\subsubsection{VB-Courant algebroids and Lie 2-algebroids}\label{split_lie_2_rep}
A \textbf{VB-Courant algebroid} \cite{Li-Bland12} is a metric double
vector bundle $(\mathbb E\to B, \langle\cdot\,,\cdot\rangle)$ with
side $Q$ and core $Q^*$, together with a linear anchor
$\Theta\colon\mathbb E\to TB$ and a linear Courant algebroid bracket
on sections of $\mathbb E\to B$.

A \textbf{decomposed VB-Courant algebroid} can be described as
\cite{Jotz17b}
a decomposed metric VB-algebroid
$\mathbb E=Q\times_M B\times_M Q^*\to B$ together with an anchor $\rho_Q\colon Q\to M$,
 a vector bundle map $\partial_B\colon Q^*\to B$,
a dull bracket $\lb\cdot\,,\cdot\rb\colon
  \Gamma(Q)\times\Gamma(Q)\to\Gamma(Q)$,
a linear connection $\nabla\colon\Gamma(Q)\times\Gamma(B)\to\Gamma(B)$ and 
a vector valued $3$-form $\omega\in \Omega^3(Q,B^*)$,
such that \cite{Jotz17b}
\begin{enumerate}
\item[(i)] $\nabla_{\partial_B^*\beta_1}^*\beta_2+\nabla_{\partial_B^*\beta_2}^*\beta_1=0$ for
  all $\beta_1,\beta_2\in\Gamma(B^*)$,
\item[(ii)] $\lb q, \partial_B^*\beta\rb=\partial_B^*(\nabla_q^*\beta)$ for
  $q\in\Gamma(Q)$ and $\beta\in\Gamma(B^*)$,
\item [(iii)]$\Jac_{\lb \cdot,\cdot\rb}=\partial_B^*\circ \omega\in \Omega^3(Q,Q)$,
\item[(iv)]
  $R_{\nabla}(q_1,q_2)b=\partial_B\langle\ip{q_2}\ip{q_1}\omega, b\rangle$
  for $q\in\Gamma(Q)$ and $\beta\in\Gamma(B^*)$, and 
\item[(v)] $\dr_{\nabla^*}\omega=0$.
\end{enumerate}
The equation
\begin{equation}\label{rho_delta}
\rho_Q\circ \partial_B^*=0
\end{equation}
follows easily from \eqref{leibniz} and (ii), and (ii) is equivalent
to 
\begin{equation}\label{D1}
\partial_B(\Delta_q\tau)=\nabla_q(\partial_B\tau)
\end{equation}
for all $q\in\Gamma(Q)$ and $\tau\in\Gamma(Q^*)$.
Further, \eqref{Jac_Dorfman}
and (iii) yield together
\begin{equation}\label{omega_dorfman_curv}
\langle \ip{q_2}\ip{q_2}\omega,\partial_B\tau\rangle=R_\Delta(q_1,q_2)\tau
\end{equation}
for $q_1,q_2\in\Gamma(Q)$ and $\tau\in\Gamma(Q^*)$.
The linear Courant algebroid structure on $Q\times_MB\times_MQ^*\to B$
is given by the anchor $\Theta\colon \mathbb E\to TB$ defined by 
\[\Theta(q,0)=\widehat{\nabla_q}\in\mx^l(B), \quad \Theta(\tau^\dagger)=(\partial_B\tau)^\uparrow\in\mx^c(B),
\]
and the bracket defined by 
$\left\lb \tau_1^\dagger, \tau_2^\dagger\right\rb=0,$ 
$\left\lb (q,0), \tau^\dagger\right\rb=\Delta_q\tau^\dagger
$
where $\Delta$ is the Dorfman
connection that is dual to the dull bracket, and 
$\left\lb (q_1,0), (q_2, 0)\right\rb=(\lb q_1, q_2\rb, -\ip{q_2}\ip{q_1}\omega)$.
for all $q,q_1,q_2\in\Gamma(Q)$ and all
$\tau,\tau_1,\tau_2\in\Gamma(Q^*)$.

 \begin{example}\label{tangent_courant_double}\cite{Jotz17b}
We consider here a Courant algebroid $(\mathsf
E\to M,\rho,\lb\cdot\,,\cdot\rb,\langle \cdot\,,\cdot\rangle)$. We use the
pairing to identify
$\mathsf E$ with $\mathsf E^*$.
After the choice of a metric connection on $\mathsf E$
and so of a Lagrangian decomposition $I_\nabla\colon T\mathsf E\to
\mathsf E\times_MTM\times_M\mathsf E$ (see Example
\ref{metric_connections}), the VB-Courant
algebroid structure on $(T\mathsf E\to TM, \mathsf E\to M)$ is
described by $\partial_{TM}=\rho\colon \mathsf E\to TM$, 
the Dorfman connection $\Delta^{\rm bas}\colon
\Gamma(\mathsf E)\times\Gamma(\mathsf E)\to\Gamma(\mathsf E)$,
\[\Delta^{\rm bas}_ee'=\lb e,e'\rb+\nabla_{\rho(e')}e,\]
which we call the \emph{basic Dorfman connection associated to
  $\nabla$}. The dual dull bracket is given by
\begin{equation}\label{dull_corretion}
\lb e,e'\rb_{\Delta^{\rm bas}}=\lb e,e'\rb-\rho^*\langle\nabla_\cdot
e,e'\rangle
\end{equation}
for all $e,e'\in\Gamma(\mathsf E)$.  The linear
connection is 
$\nabla^{\rm bas}\colon\Gamma(\mathsf E)\times\mx(M)\to\mx(M)$,
\[\nabla^{\rm bas}_eX=[\rho(e),X]+\rho(\nabla_{X}e).\]
The \emph{basic curvature} $\omega_\Delta^{\rm
  bas}\in\Omega^3(\mathsf E,T^* M)$ is defined by
\begin{align}
  \omega_\Delta^{\rm bas}(e_1,e_2,\cdot)X=& -\nabla_X\lb e_1,e_2\rb+\lb
  \nabla_Xe_1,e_2\rb+\lb e_1,\nabla_Xe_2\rb \label{def_R_nabla_bas}\\
&+\nabla_{\nabla^{\rm
      bas}_{e_2}X}e_1-\nabla_{\nabla^{\rm
      bas}_{e_1}X}e_2-\beta\inv\langle \nabla_{\nabla^{\rm bas}_\cdot
    X}e_1,e_2\rangle\in\Gamma(\mathsf E)\nonumber
\end{align}
for all $e_1,e_2\in\Gamma(\mathsf E)$ and $X\in\mx(M)$. 
\end{example}
\medskip

A
homological vector field $\mathcal Q$ on a positively graded manifold
$\mathcal M$ is a derivation of degree $1$ of
$C^\infty(\mathcal M)$ such that
$\mathcal Q^2=\frac{1}{2}[\mathcal Q,\mathcal Q]$ vanishes
\cite{Roytenberg99}. If the graded manifold is a $[2]$-manifold, then the
pair $(\mathcal M,\mathcal Q)$ is a \emph{Lie 2-algebroid} \cite{ShZh17}.  

Consider the split $[2]$-manifold $Q[-1]\oplus B^*[-2]$ corresponding
to the underlying decomposed metric double vector bundle
$Q\times_MB\times_MQ^*$. The linear Courant algebroid structure
defines as follows a homological vector field $\mathcal Q$ on
$C^\infty(Q[-1|\oplus B^*[-2])$ (see \cite{Jotz17b}):
\begin{equation}\label{Q1}
\mathcal Q(f)=\rho_Q^*\dr f\in \Gamma(Q^*)
\end{equation}
for $f\in C^\infty(M)$, 
\begin{equation}\label{Q2}
\mathcal Q(\tau)=\dr_Q\tau+\partial_B\tau \in \Omega^2(Q)\oplus\Gamma(B)
\end{equation}
for $\tau\in\Gamma(Q^*)$ 
and 
\begin{equation}\label{Q3}
\mathcal Q(b)=\dr_\nabla b-\langle \omega,b\rangle\in \Omega^1(Q,B)\oplus\Omega^3(Q).
\end{equation}
for $b\in\Gamma(B)$.  
The tuple
$(\partial_B^*\colon B^*\to Q,
\rho_Q,\lb\cdot\,,\cdot\rb,\nabla,\omega)$,
is a \textbf{split Lie 2-algebroid} \cite{ShZh17, Jotz17b}.
Note that conversely any homological vector field on
$\mathcal M=Q[-1]\oplus B^*[-2]$ defines as in \eqref{Q1}, \eqref{Q2}
and \eqref{Q3} a
split Lie $2$-algebroid

 The
category of Lie 2-algebroid is equivalent via the correspondence
described above to the category of VB-Courant algebroids
\cite{Li-Bland12,Jotz17b}. Note that a morphism
$\mu\colon (\mathcal M_1,\mathcal Q_1)\to (\mathcal M_2,\mathcal Q_2)$
of Lie 2-algebroids is a morphism
$\mu\colon \mathcal M_1\to \mathcal M_2$ of the underlying
$[2]$-manifolds, such that
\begin{equation}\label{morphism_of_Q}
\mu^\star\circ\mathcal Q_2=\mathcal Q_1\circ\mu^\star\colon
C^\infty(\mathcal M_2)\to C^\infty(\mathcal M_1).
\end{equation}
We refer to \S 3.5 of \cite{Jotz17b} for the characterisation of a
morphism of split Lie 2-algebroids in terms of its components
$(\partial_B^*\colon B^*\to Q,
\rho_Q,\lb\cdot\,,\cdot\rb,\nabla,\omega)$.

\section{LA-Courant algebroids vs Poisson Lie 2-algebroids}\label{sec:LA-Cour}
In this section, we prove that a split Poisson Lie 2-algebroid is
equivalent to the \emph{matched pair} of a split Lie 2-algebroid with a
self-dual 2-representation.

Take a double vector bundle $(\mathbb E, B,Q,M)$
with core $Q^*$, with a VB-Lie algebroid structure on
$(\mathbb E\to Q, B\to M)$ and a VB-Courant algebroid structure on
$(\mathbb E\to B, Q\to M)$. In this section we show that the double
vector bundle is an \emph{LA-Courant algebroid} \cite{Li-Bland12} if
and only if the VB-algebroid is metric and the self-dual
2-representation defined by any Lagrangian decomposition of
$\mathbb E$ and the VB-algebroid side forms a \emph{matched pair} with
the split Lie 2-algebroid describing the Courant algebroid side. 

We begin with the following definition.
\begin{definition}\label{matched_pairs}
  Let $(B\to M, \rho_B, [\cdot\,,\cdot])$ be a Lie algebroid and
  let $(Q\to M, \rho_Q)$ be an anchored vector bundle.  Assume that $B$ acts
  on $\partial_Q\colon Q^*\to Q$ up to homotopy via
  a self-dual 2-representation $(\nabla,\nabla^*, R)$, and let
  $(\partial_B\colon Q^*\to B,
  \lb\cdot\,,\cdot\rb,\nabla, \omega)$\footnote{For the sake of simplicity, we write
    in this definition $\nabla$ for two different connections,
    unless it is not clear from the indexes which connection is
    meant.} be a split Lie 2-algebroid.  Then we say that the
  2-representation and the split Lie 2-algebroid form a matched
  pair if
  \begin{enumerate}
\item[(M1)] $\partial_Q(\Delta_q\tau)=\nabla_{\partial_B\tau}q+\lb
  q, \partial_Q\tau\rb+\partial_B^*\langle \tau, \nabla_\cdot q\rangle$,
\item[(M2)] $\partial_B(\nabla_b^*\tau)=[b,\partial_B\tau]+\nabla_{\partial_Q\tau}b$,
\item[(M3)]
  $\partial_BR(b_1,b_2)q=-\nabla_q[b_1,b_2]+[\nabla_qb_1,b_2]+[b_1,\nabla_qb_2]+\nabla_{\nabla_{b_2}q}b_1-\nabla_{\nabla_{b_1}q}b_2$,
\item[(M4)] $\partial_Q\langle\ip{q_2}\ip{q_1}\omega,b\rangle=-\nabla_b\lb q_1,q_2\rb+\lb q_1,
  \nabla_{b}q_2\rb+\lb \nabla_bq_1,
  q_2\rb+\nabla_{\nabla_{q_2}b}q_1-\nabla_{\nabla_{q_1}b}q_2+\partial_B^*\langle
  R(\cdot, b)q_1, q_2\rangle$.
\item[(M5)] $\dr_{\nabla^Q}\omega=\dr_{\nabla^B} R\in \Omega^2(B,\wedge^3Q^*)\simeq\Omega^3(Q,\wedge^2B^*)$, where 
$\omega$ is seen as an element of $\Omega^1(B,\wedge^3Q^*)$ and $R$ is
understood as an element of $\Omega^2(Q,\wedge^2B^*)$.
\end{enumerate}
\end{definition}

\begin{remark}\label{remark_simplifications}
\begin{enumerate}
\item (M5) is \begin{equation*}
\begin{split}
  &\nabla_{b_2}^*\langle\ip{q_2}\ip{q_1}\omega,b_1\rangle-\nabla_{b_1}^*\langle\ip{q_2}\ip{q_1}\omega,b_2\rangle+\langle\ip{q_2}\ip{q_1}\omega,[b_1,b_2]\rangle\\
  &+\langle\ip{q_2}\ip{\nabla_{b_1}q_1}\omega,b_2\rangle+\langle\ip{\nabla_{b_1}q_2}\ip{q_1}\omega,b_2\rangle-\langle\ip{q_2}\ip{\nabla_{b_2}q_1}\omega,b_1\rangle-\langle\ip{\nabla_{b_2}q_2}\ip{q_1}\omega,b_1\rangle\\
  &+\Delta_{q_1}R(b_1,b_2)q_2-\Delta_{q_2}R(b_1,b_2)q_1-R(b_1,b_2)\lb q_1,q_2\rb\\
  &-R(\nabla_{q_1}b_1,b_2)q_2-R(b_1,\nabla_{q_1}b_2)q_2+R(\nabla_{q_2}b_1,b_2)q_1+R(b_1,\nabla_{q_2}b_2)q_1\\
  =\,&\langle (R(b_1,\nabla_\cdot b_2)+R(\nabla_\cdot b_1, b_2))q_1,q_2\rangle -\rho_Q^*\dr\langle R(b_1,b_2)q_1, q_2\rangle
\end{split} 
\end{equation*}
for all $q_1,q_2\in \Gamma(Q)$ and $b_1,b_2\in\Gamma(B)$.
\item The equality $\rho_Q\circ \partial_Q=\rho_B\circ\partial_B$ follows from (M1)
if $Q$ has positive rank, and from (M2) if $B$ has positive rank. If both $Q$ and $B$ have rank zero, then
$\rho_Q\circ \partial_Q=\rho_B\circ\partial_B$ is trivially satisfied.
\item The equation 
\begin{equation}\label{mixed_anchors}
[\rho_Q(q), \rho_B(b)]=\rho_B(\nabla_qb)-\rho_Q(\nabla_bq)
\end{equation} follows 
from (M3) if $B$ has positive rank, and from (M4) if $Q$ has positive rank.
If both $Q$ and $B$ have rank zero, then it is trivially satisfied.
\item If $\rho_Q\circ \partial_Q=\rho_B\circ\partial_B$, then (M1) is equivalent to
\begin{equation}\label{almost_C}
\left(\Delta_{\partial_Q\tau_1}\tau_2-\nabla^*_{\partial_B\tau_2}\tau_1\right)
+\left(\Delta_{\partial_Q\tau_2}\tau_1-\nabla^*_{\partial_B\tau_1}\tau_2\right)
=\rho_Q^*\dr\langle \tau_1, \partial_Q\tau_2\rangle
\end{equation}
for all $\tau_1,\tau_2\in\Gamma(Q^*)$.
\item Assuming
  $[\rho_Q(q), \rho_B(b)]=\rho_B(\nabla_qb)-\rho_Q(\nabla_bq)$ for all
  $b\in\Gamma(B)$ and $q\in\Gamma(Q)$,  (M4) is equivalent
  to
\begin{equation}\label{(LC10)}
\langle\ip{\partial_Q\tau}\ip{q}\omega,b\rangle-R(b,\partial_B\tau)q=\Delta_q\nabla^*_b\tau-\nabla^*_b\Delta_q\tau
+\Delta_{\nabla_bq}\tau-\nabla^*_{\nabla_qb}\tau-\langle\nabla_{\nabla_\cdot  b}q, \tau\rangle
\end{equation}
for all $b\in\Gamma(B)$, $q\in\Gamma(Q)$ and $\tau\in\Gamma(Q^*)$.
\end{enumerate}
\end{remark}

\subsection{Poisson Lie 2-algebroids via matched pairs.}
We begin this subsection with the definition of a Poisson Lie
2-algebroid.
\begin{definition}\label{def_Poisson_lie_2}
  Let $\mathcal M$ be a Poisson $[2]$-manifold with algebra of
  functions $C^\infty(\mathcal M)$ and degree $-2$ Poisson bracket
  $\{\cdot\,,\cdot\}$. Assume that $\mathcal M$ has in addition a Lie
  2-algebroid structure, i.e.~that it is endowed with a homological
  vector field $\mathcal Q$. Then
  $(\mathcal M, \mathcal Q, \{\cdot\,,\cdot\})$ is a \textbf{Poisson
    Lie 2-algebroid} if the homological vector field preserves the
  Poisson structure, i.e.~ if
  \begin{equation}\label{Q_preserves_bracket}
\mathcal Q\{\xi_1,\xi_2\}=\{\mathcal Q(\xi_1), \xi_2\}+(-1)^{\deg\xi_1}\{\xi_1,\mathcal Q(\xi_2)\}
\end{equation} for all
$\xi_1,\xi_2\in\mathcal A$.

A morphism of Poisson Lie 2-algebroids is a morphism of the underlying
$[2]$-manifold that is a morphism of Poisson $[2]$-manifolds
\emph{and} a morphism of Lie 2-algebroids.
\end{definition}

The main theorem of this section shows that matched pairs as in
Definition \ref{matched_pairs} are equivalent to split Poisson Lie
2-algebroids.
\begin{theorem}\label{Poisson_Lie_2_char}
  Let $\mathcal M=Q[-1]\oplus B^*[-2]$ be a split $[2]$-manifold
  endowed with a homological vector field $\mathcal Q$ and a Poisson
  bracket $\{\cdot\,,\cdot\}$ of degree $-2$. Let
  $(\partial_B\colon Q^*\to B,\lb\cdot\,,\cdot\rb,\nabla,\omega)$ be
  the components of $\mathcal Q$, and let
  $(\partial_Q\colon Q^*\to Q, \nabla^*, \nabla, R)$ be the self-dual
  $2$-representation of $B$ that is equivalent to the Poisson bracket.

  Then $(\mathcal M, \mathcal Q, \{\cdot\,,\cdot\})$ is a Poisson Lie
  2-algebroid if and only if the self dual 2-representation and the
  split Lie 2-algebroid form a matched pair.
\end{theorem}

\begin{proof}
  The idea of this proof is very simple, but requires rather long
  computations. We will leave some of the detailed verifications to
  the reader. We check \eqref{Q_preserves_bracket} in coordinates, by
  using the formulae found in \eqref{Q1}, \eqref{Q2}, \eqref{Q3} and 
  \eqref{2-poisson} for $\mathcal Q$ and $\{\cdot\,,\cdot\}$,
  respectively.

  First we have
  $\mathcal Q(f)=\rho_Q^*\dr
  f\in\Gamma(Q^*)$ and $\{f,g\}=0$ for $f,g\in C^\infty(M)$. This
  yields
$\{\mathcal Q(f),g\}+\{f,\mathcal Q(g)\}=\left\{\rho_Q^*\dr f, g\right\}
+\left\{f, \rho_Q^*\dr g\right\}=0=\mathcal Q\{f,g\}
$
by the graded skew-symmetry and $\{\tau,f\}=0$ for
$\tau\in\Gamma(Q^*)$.  
Then we have for $\tau\in\Gamma(Q^*)$:
\begin{equation*}
\begin{split}
 \{\mathcal Q(\tau),f\}-\{\tau,\mathcal
 Q(f)\}&=\left\{\dr_Q\tau+\partial_B\tau,f\right\}-\{\tau,\rho_Q^*\dr f\}\\
  &=0+\rho_B(\partial_B\tau)(f)-\rho_Q(\partial_Q\tau)(f).
\end{split}
\end{equation*}
But we also have $\mathcal Q\{\tau,f\}=\mathcal Q(0)=0$. Hence,
$\{\mathcal Q(\tau),f\}-\{\tau,\mathcal Q(f)\}=\mathcal Q\{\tau,f\}$ is equivalent to
$\rho_B(\partial_B\tau)(f)=\rho_Q(\partial_Q\tau)(f)$.

In a similar manner, we have $\mathcal
Q\{b,f\}=\rho_Q^*\dr(\rho_B(b)(f))$ for $b\in\Gamma(B)$ and $\{\mathcal Q(b),f\}+\{b,\mathcal Q(f)\}$ is 
\begin{equation*}
\begin{split}
  &\left\{\dr_\nabla b-\langle \omega,b\rangle,
    f\right\}+\{b,\rho_Q^*\dr f\}=\left\{\dr_\nabla b,
    f\right\}+\nabla^*_{b}(\rho_Q^*\dr
  f)\\
  =&\left\{\nabla_\cdot b,
    f\right\}+\nabla^*_{b}(\rho_Q^*\dr
  f)=\rho_B(\nabla_\cdot b)(f)+\nabla^*_{b}(\rho_Q^*\dr
  f)\in\Omega^1(Q).
\end{split}
\end{equation*}
Hence, $\mathcal Q\{b,f\}=\{\mathcal Q(b),f\}+\{b,\mathcal Q(f)\}$ if and only if 
\[\rho_Q(q)\rho_B(b)(f)=\rho_B(\nabla_{q}b)(f)+\rho_B(b)\rho_Q(q)(f)-\rho_Q(\nabla_{b}q)(f)\]
for all $q\in\Gamma(Q)$. Since $f$ was arbitrary, this is \eqref{mixed_anchors}:
$[\rho_Q(q),\rho_B(b)]=\rho_B(\nabla_{q}b)-\rho_Q(\nabla_{b}q)$.
Then we have $\mathcal Q\{b,\tau\}=\mathcal
Q(\nabla_{b}^*\tau)=\partial_B(\nabla^*_{b}\tau)+\dr_Q(\nabla^*_{b}\tau)\in\Gamma(B)\oplus\Omega^2(Q)$. The Poisson bracket $\{\mathcal Q(b),\tau\}$ is
$\{\dr_\nabla b-\langle\omega,b\rangle,\tau\}$. A simple computation
shows that $\{\eta,\tau\}=(-1)^{k+1}\ip{\partial_Q\tau}\eta$ for
$\eta\in\Omega^k(Q)$. Hence,
$\{\langle\omega,b\rangle,\tau\}=(-1)^4\ip{\partial_Q\tau}\langle\omega,b\rangle=\ip{\partial_Q\tau}\langle\omega,b\rangle$.
The bracket $\{\dr_\nabla b,\tau\}$ equals $\nabla_{\partial_Q\tau}b+\psi_{b,\tau}$,
with $\psi_{b,\tau}\in\Omega^2(Q)$ the form defined by
\[\psi_{b,\tau}(q_1,q_2)=\langle\nabla^*_{\nabla_{q_1}b}\tau,q_2\rangle-\langle\nabla^*_{\nabla_{q_2}b}\tau,q_1\rangle
\]
for $q_1,q_2\in\Gamma(Q)$.
The Poisson bracket $\{b, \mathcal
Q(\tau)\}=\{b, \dr_Q\tau+\partial_B\tau\}=\{b,
\dr_Q\tau\}+[b,\partial_B\tau]-R(b,\partial_B\tau)$ simplifies to
$\nabla_b^*(\dr_Q\tau)+[b,\partial_B\tau]-R(\partial_B\tau,b)$ because
$\{b,\eta\}=\nabla_b\eta\in\Omega^k(Q)$ for all $\eta\in\Omega^k(Q)$.
By comparing the $\Gamma(B)$ and the $\Omega^2(Q)$-terms, we find that
\[\mathcal Q\{b,\tau\}=\{\mathcal Q(b),\tau\}+\{b, \mathcal Q(\tau)\}
\]
if and only if 
$\partial_B\nabla_{b}^*\tau=\nabla_{\partial_Q\tau}b+[b,\partial_B\tau],
$
which is (M2)
and 
\begin{equation}\label{Omega2Q-part}
\dr_Q(\nabla_b^*\tau)=-\ip{\partial_Q\tau}\langle\omega,b\rangle+\psi_{b,\tau}+\nabla_b^*(\dr_Q\tau)+R(\partial_B\tau,b).
\end{equation}
On $q_1,q_2\in\Gamma(Q)$, the $2$-form 
$\dr_Q(\nabla_b^*\tau)+\ip{\partial_Q\tau}\langle\omega,b\rangle-\psi_{b,\tau}-\nabla_b^*(\dr_Q\tau)-R(\partial_B\tau,b)$
is 
\begin{equation*}
\begin{split}
&\left([\rho_Q(q_1),\rho_B(b)]-\rho_B(\nabla_{q_1}b)+\rho_Q(\nabla_bq_1)\right)\langle
\tau,q_2\rangle\\
+&\,\left([\rho_B(b),\rho_Q(q_2)]+\rho_B(\nabla_{q_2}b)-\rho_Q(\nabla_bq_2)\right)\langle
\tau,q_1\rangle\\
+&\,\left\langle\tau, \partial_Q^*\langle\ip{q_2}\ip{q_1}\omega,b\rangle+\nabla_b\lb
  q_1,
  q_2\rb+\nabla_{\nabla_{q_1}b}q_2-\nabla_{\nabla_{q_2}b}q_1\right\rangle\\
&-\left\langle\tau, \lb\nabla_bq_1,q_2\rb+\lb
  q_1,\nabla_bq_2\rb+\partial_B^*\langle R(\cdot,b)q_1,q_2\rangle\right\rangle.
\end{split}
\end{equation*}
Hence if \eqref{mixed_anchors} holds, then \eqref{Omega2Q-part} is (M4).
 Next we study the equation
$\mathcal Q\{\tau_1,\tau_2\}=\{\mathcal
Q(\tau_1),\tau_2\}-\{\tau_1,\mathcal Q(\tau_2)\}$ for $\tau_1,\tau_2\in\Gamma(Q^*)$.  The
left hand side is
$\rho_Q^*\dr\langle\tau_1,\partial_Q\tau_2\rangle$. The right-hand
side is $\{\mathcal
Q(\tau_1),\tau_2\}+\{\mathcal Q(\tau_2),\tau_1\}$.
The equality 
\begin{equation*}
\begin{split}
\{\mathcal
Q(\tau_1),\tau_2\}&=\{\dr_Q\tau_1+\partial_B\tau_1,\tau_2\}=-\ip{\partial_Q\tau_2}\dr_Q\tau_1+\nabla^*_{\partial_B\tau_1}\tau_2\\
&=\rho_Q^*\dr\langle\tau_1,\partial_Q\tau_2\rangle-\Delta_{\partial_Q\tau_2}\tau_1+\nabla^*_{\partial_B\tau_1}\tau_2
\end{split}
\end{equation*}
shows hence that $\mathcal Q\{\tau_1,\tau_2\}=\{\mathcal
Q(\tau_1),\tau_2\}+\{\mathcal Q(\tau_2),\tau_1\}$ if and only if 
\eqref{almost_C}. Recall from Remark
\ref{remark_simplifications} that together with
$\rho_B\circ\partial_B=\rho_Q\circ\partial_Q$, this is equivalent to
(M1).

Finally, we choose $b_1,b_2\in\Gamma(B)$ and we study the equation
$\mathcal Q\{b_1,b_2\}=\{\mathcal Q(b_1),b_2\}+\{b_1,\mathcal
Q(b_2)\}=\{\mathcal Q(b_1),b_2\}-\{\mathcal
Q(b_2), b_1\}$. The left-hand side is 
\[\mathcal Q([b_1,b_2]-R(b_1,b_2))=\underset{\in\Omega^1(Q,B)}{\underbrace{\dr_\nabla [b_1,b_2]}}-\underset{\in\Omega^3(Q)}{\underbrace{\langle\omega,
[b_1,b_2]\rangle-\dr_Q(R(b_1,b_2))}}+\underset{\in\Omega^1(Q,B)}{\underbrace{\partial_BR(b_1,b_2)}}.
\]
Note that in the expression $\dr_Q(R(b_1,b_2))$, the object
$R(b_1,b_2)$ is understood as an element of $\Omega^2(Q)$, and in the
expression $\partial_BR(b_1,b_2)$, it is understood as a morphism
$Q\to Q^*$.
The Poisson bracket $\{\mathcal Q(b_1),b_2\}$ is 
\[\{\dr_\nabla
b_1-\langle\omega,b_1\rangle,b_2\}=\underset{\in\Omega^1(Q,B)}{\underbrace{\nabla_{\nabla_{b_2}\cdot}b_1+[\nabla_\cdot
b_1,b_2]}}-\underset{\in\Omega^3(Q)}{\underbrace{R(\nabla_\cdot b_1,b_2)+\nabla_{b_2}\langle\omega,b_1\rangle}}
\]
The projection to $\Omega^1(Q,B)$ of $\mathcal Q\{b_1,b_2\}=\{\mathcal Q(b_1),b_2\}-\{\mathcal
Q(b_2), b_1\}$ is
\[\dr_\nabla [b_1,b_2]+\partial_BR(b_1,b_2)=\nabla_{\nabla_{b_2}\cdot}b_1+[\nabla_\cdot
b_1,b_2]-\nabla_{\nabla_{b_1}\cdot}b_2-[\nabla_\cdot
b_2,b_1],
\]
which is (M3). The projection to $\Omega^3(Q)$ of $\mathcal Q\{b_1,b_2\}=\{\mathcal Q(b_1),b_2\}-\{\mathcal
Q(b_2), b_1\}$ is
\[-\langle\omega,[b_1,b_2]\rangle-\dr_Q(R(b_1,b_2))=-R(\nabla_\cdot
b_1,b_2)+\nabla_{b_2}\langle\omega,b_1\rangle
+R(\nabla_\cdot b_2,b_1)-\nabla_{b_1}\langle\omega,b_2\rangle,
\]
that is,
\[(\dr_\nabla\omega)(b_1,b_2)=\dr_Q(R(b_1,b_2))-R(\nabla_\cdot
b_1,b_2)
+R(\nabla_\cdot b_2,b_1).
\]
The right-hand side of this equation is easily calculated to be the
pairing of 
$(\dr_{\nabla^B} R)\in\Omega^3(Q,\wedge^2 B^*)$ with $(b_1,b_2)$. Hence, the projection to $\Omega^3(Q)$ of $\mathcal Q\{b_1,b_2\}=\{\mathcal Q(b_1),b_2\}-\{\mathcal
Q(b_2), b_1\}$ is (M5).
\end{proof}

\subsection{LA-Courant algebroids and equivalence of categories}
Li-Bland's definition of an LA-Courant algebroid \cite{Li-Bland12} is
quite technical and requires the consideration of triple vector
bundles \cite{Mackenzie05b}.

\subsubsection{The LA-Courant algebroid condition}

A \textbf{Dirac structure} with support in a Courant algebroid $\mathsf E\to M$ 
is a subbundle $D\to S$ over a sub-manifold $S$ of $M$, such that
$D(s)$ is maximal isotropic in $\mathsf E(s)$ for all $s\in S$ and 
\[ e_1\an{S}\in\Gamma_S(D), e_2\an{S}\in\Gamma_S(D) \quad \Rightarrow\quad \lb e_1, e_2\rb\an{S}\in\Gamma_S(D)
\]
for all $e_1,e_2\in\Gamma(\mathsf E)$.

Later we will need the following lemma in Section \ref{sec:dirac}. We
leave the proof to the reader.
\begin{lemma}\label{useful_for_dirac_w_support}
  Let $\mathsf E\to M$ be a Courant algebroid and $D\to S$ a
  subbundle with $S$ a sub-manifold of $M$. Assume that $D\to S$ is
  spanned by the restrictions to $S$ of a family $\mathcal S\subseteq
  \Gamma(\mathsf E)$ of sections of $\mathsf E$. Then $D$ is a Dirac
  structure with support $S$ if and only if
\begin{enumerate}
\item $\rho_{\mathsf E}(e)(s)\in T_sS$ for all $e\in\mathcal S$ and $s\in S$,
\item $D_s$ is Lagrangian in $\mathbb E_s$ for all $s\in S$ and
\item $\lb e_1,e_2\rb\an{S}\in\Gamma_S(D)$ for all $e_1,e_2\in\mathcal
  S$.
\end{enumerate}
\end{lemma}

Consider a Lie algebroid $(q_A\colon A\to M, \rho_A,
[\cdot\,,\cdot])$. 
In \cite{Li-Bland12} Li-Bland defines a relation
$\Pi_A\subseteq TA\times TA$ 
by
\begin{equation*}
\begin{split} 
&T_mb\rho_A(a)(m)+_A\left.\frac{d}{dt}\right\an{t=0}b(m)+t([b,a]+c)(m)\\
&\hspace*{3cm}\sim_{\Pi_A}T_ma\rho_A(b)(m)+_A\left.\frac{d}{dt}\right\an{t=0}a(m)+tc(m)
\end{split}
\end{equation*}
for all $a,b,c\in\Gamma(A)$. (Note that in \cite{Li-Bland12}, the
relation is defined in a different manner. Checking that this
alternative definition is correct is rather long. Details can be
obtained in the appendix of \cite{Jotz15}).
\medskip

Now consider  a double vector bundle $(\mathbb E, B,Q,M)$ 
endowed with a VB-Lie algebroid structure $(\mathbf
b, [\cdot\,,\cdot])$ on $\mathbb E\to Q$ and a linear 
metric $\langle\cdot\,,\cdot\rangle$ on $\mathbb E\to B$ 
(hence, $\mathbb E$ has core $Q^*$). Let $\rho_B\colon B\to TM$ be the anchor 
of the induced Lie algebroid structure on $B$.

The relation $\Pi_{\mathbb E}$ defined as above by the Lie algebroid
structure on $\mathbb E$ over $Q$ is then a relation $\Pi_{\mathbb E}$ of
the triple vector bundles \cite{Li-Bland12}
\begin{equation*}
\begin{xy}
  \xymatrix{T\mathbb E\ar[rr]\ar[dd]\ar[rd]&&TQ\ar[dd]\ar[rd]&\\
&\mathbb E\ar[dd]\ar[rr]&&Q \ar[dd]\\
TB\ar[rd]\ar[rr]&&TM \ar[rd]&\\
 &   B\ar[rr]&&M }
\end{xy}
\qquad 
\text{ and }
\qquad 
\begin{xy}
  \xymatrix{T\mathbb E\ar[rr]\ar[dd]\ar[rd]&&\mathbb E\ar[dd]\ar[rd]&\\
&TQ\ar[dd]\ar[rr]&&Q \ar[dd]\\
TB\ar[rd]\ar[rr]&&B \ar[rd]&\\
 &   TM\ar[rr]&&M }
\end{xy}
\end{equation*}
Li-Bland's definition \cite{Li-Bland12} is the following.
\begin{definition}
Let $(\mathbb E,Q,B,M)$ be a double vector bundle with
a VB-Courant algebroid structure (over $B$) and a VB-algebroid structure
$ (\mathbb E\to Q,B\to M)$.
Then $(\mathbb E, B,Q,M)$ is an LA-Courant
algebroid if $\Pi_{\mathbb E}$ is a Dirac structure with support of
the Courant algebroid $\overline{T\mathbb E}\times T\mathbb E$.
\end{definition}

We have the following theorem.

\begin{theorem}\label{LA-Courant}
  Let $(\mathbb E,Q,B,M)$ be a double vector bundle with a VB-Courant
  algebroid structure on $(\mathbb E\to B, Q\to M)$ and a VB-Lie
  algebroid structure on $(\mathbb E\to Q, B\to M)$. Then in
  particular, $\mathbb E$ is a metric double vector bundle with the
  linear metric underlying the linear Courant algebroid structure on
  $\mathbb E\to B$.  Choose a Lagrangian decomposition $\Sigma\colon
  B\times_MQ\to \mathbb E$ of $\mathbb E$.  Then $(\mathbb E,Q,B,M)$
  is an LA-Courant algebroid if and only if
\begin{enumerate}
\item the linear Lie algebroid structure on $\mathbb E\to Q$ is
  compatible in the sense of \S\ref{metricVBA} with the
  linear metric, and
\item the self-dual 2-representation and the split Lie 2-algebroid
  obtained from the Lagrangian splitting form a matched pair as in
  Definition \ref{matched_pairs}.
\end{enumerate}
\end{theorem}
The proof of this theorem is very long and technical (see the appendix
of \cite{Jotz15}), showing that Li-Bland's definition of an LA-Courant
algebroid is hard to handle.  Hence our result provides a new
definition of LA-Courant algebroids, that is much simpler to express
and probably also easier to use.

Further, we now explain how this theorem shows that LA-Courant
algebroids are equivalent to Poisson Lie 2-algebroids. This
has already been found by Li-Bland in \cite{Li-Bland12}. First,
morphisms of LA-Courant algebroids are morphisms of metric double
vector bundles that preserve the Courant algebroid structure and the
Lie algebroid structure \cite{Li-Bland12}.  Hence, the category of
LA-Courant algebroids is a full subcategory of the intersection of the
category of metric VB-algebroids and the category of VB-Courant
algebroids.

On the other hand, Definition \ref{def_Poisson_lie_2} shows that the
category of Poisson Lie 2-algebroids is a full subcategory of the
intersection of the categories of Poisson $[2]$-manifolds and of Lie
$2$-algebroids.

This, Theorem \ref{Poisson_Lie_2_char} and Theorem \ref{LA-Courant}
show that the equivalences of the categories of metric VB-algebroids
and of Poisson $[2]$-manifolds and of the categories of VB-Courant
algebroids and Lie $2$-algebroids restrict to an equivalence of the
category of LA-Courant algebroids with the category of Poisson Lie
$2$-algebroids.

\subsection{Examples of LA-Courant algebroids and Poisson Lie 2-algebroids}
Next we discuss some classes of Examples of LA-Courant algebroids, and
the corresponding Poisson Lie $2$-algebroids.
\subsubsection{The tangent double of a Courant algebroid}\label{TCourant_LACourant}
Let $\mathsf E\to M$ be a Courant algebroid and choose a metric
connection $\nabla\colon\mx(M)\times\Gamma(\mathsf E)\to\Gamma(\mathsf
E)$.  We have seen in Examples \ref{metric_connections} and \ref{self-dual-metric-conn} that the triple
$(\nabla,\nabla,R_\nabla)$ is then the self dual $TM$-representation up to
homotopy describing $(T\mathsf E\to \mathsf E, TM\to M)$ after the
choice of the Lagrangian decomposition $\Sigma^\nabla\colon\mathsf
E\times_M M\times_M \mathsf E\to
T\mathsf E$.  We have also seen in Example \ref{tangent_courant_double} that
the split Lie 2-algebroid encoding the Courant
algebroid side $(T\mathsf E\to TM, \mathsf E\to TM)$ is
$(\rho_{\mathsf E}\colon\mathsf E\to TM, \lb\cdot\,,\cdot\rb_{\Delta^{\rm bas}},\nabla^{\rm bas},
\omega^{\rm bas}_\Delta)$. 

A straightforward computation resembling the one
in \cite[Section 3.2]{GrJoMaMe18} for the tangent double of a Lie
algebroid shows that this $2$-representation and this split Lie
2-algebroid are matched, and so that $T\mathsf E$ is
an LA-Courant algebroid (see also \cite{Li-Bland12}).

The Poisson structure on the $[2]$-manifold corresponding to $T\mathsf
E$ is, via the equivalence of $[2]$-manifolds with metric double
vector bundles, just the Poisson structure that is dual to the Lie
algebroid $T\mathsf E\to \mathsf E$. Hence, it is symplectic (see \cite{Jotz18b},
in particular \S 4.5.1).

Hence, the class of LA-Courant algebroids that is equivalent to the
\emph{symplectic} Lie $2$-algebroids is just the class of tangent
prolongations of Courant algebroids.

\subsubsection{The standard Courant algebroid over a Lie algebroid}\label{PontLA_LACourant}
Let $A$ be a Lie algebroid. Then $TA\oplus T^*A$ is a double vector
bundle with sides $A$ and $TM\oplus A^*$ and with core $A\oplus
T^*M$. It has a linear Courant
algebroid structure on $TA\oplus_AT^*A\to A$ (see \cite{Jotz17b})
and a metric VB-algebroid structure
$(TA\oplus_AT^*A\to TM\oplus A^*, A\to M)$ (see \cite{Jotz18b}).

Set $\partial_A=\pr_A\colon A\oplus
T^*M\to A$, consider a skew-symmetric dull bracket
$\lb\cdot\,,\cdot\rb$ on $\Gamma(TM\oplus A^*)$, with $TM\oplus A^*$
anchored by $\pr_{TM}$, and let $\Delta\colon \Gamma(TM\oplus
E^*)\times\Gamma(A\oplus T^*M)\to\Gamma(A\oplus T^*M)$ be the dual
Dorfman connection. This Dorfman connection is equivalent to a
Lagrangian splitting of the metric double vector bundle 
$TA\oplus T^*A$ \cite{Jotz18a,Jotz18b}. It also defines as follows
a split Lie 2-algebroid
structure on the vector bundles $(TM\oplus A^*,\pr_{TM})$ and $A^*$ \cite{Jotz17b}.

Let $\nabla\colon\Gamma(TM\oplus A^*)\times\Gamma(A)\to\Gamma(A)$ be
the ordinary linear connection defined by
$\nabla=\pr_A\circ\Delta\circ\iota_A$.  The vector bundle map $l=\pr_A^*\colon
A^*\to TM\oplus A^*$ is just the canonical inclusion.
Define $\omega\in\Omega^3(TM\oplus A^*, A^*)$ by
$\omega(\nu_1,\nu_2,\nu_3)=\Jac_{\lb\cdot\,,\cdot\rb}(\nu_1,\nu_2,\nu_3)$. 

The objects $l$, $\lb\cdot\,,\cdot\rb$, $\nabla^*$, $\omega$ define a
split Lie 2-algebroid; the standard split Lie 2-algebroid defined by
the dull bracket (or equivalently by the dual Dorfman connection).

We give in \cite{Jotz18a,Jotz18b} the self-dual $2$-representation
$((\rho,\rho^*)\colon A\oplus T^*M\to TM\oplus A^*, \nabla^{\rm
  bas},\nabla^{\rm bas}, R_\Delta^{\rm bas})$
of $A$ that is defined by the VB-algebroid
$(TA\oplus T^*A\to TM\oplus A^*, A\to M)$ and any such Dorfman
connection: The connections
$\nabla^{\rm bas}\colon \Gamma(A)\times\Gamma(A\oplus T^*M)
\to\Gamma(A\oplus T^*M)$
and
$\nabla^{\rm bas}\colon \Gamma(A)\times\Gamma(TM\oplus A^*)
\to\Gamma(TM\oplus A^*)$ are
\begin{equation*}
\nabla^{\rm bas}_a(X,\alpha)=(\rho,\rho^*)(\Omega_{(X,\alpha)}a)+\ldr{a}(X,\alpha)
\quad \text{ and } \quad \nabla^{\rm
  bas}_a(b,\theta)=\Omega_{(\rho,\rho^*)(b,\theta)}a+\ldr{a}(b,\theta), 
\end{equation*}
where $\Omega\colon \Gamma(TM\oplus A^*)\times\Gamma(A)\to\Gamma(A\oplus T^*M)$ is defined by
\[\Omega_{(X,\alpha)}a=\Delta_{(X,\alpha)}(a,0)-(0,\dr\langle\alpha, a\rangle)
\] and for $a\in\Gamma(A)$, the derivations $\ldr{a}$ over $\rho(a)$
are defined by:
\[\ldr{a}\colon \Gamma(A\oplus T^*M)\to\Gamma(A\oplus T^*M),\quad \ldr{a}(b,\theta)=([a,b], \ldr{\rho(a)}\theta)\]
and 
\[\ldr{a}\colon \Gamma(TM\oplus A^*)\to\Gamma(TM\oplus A^*), \quad \ldr{a}(X,\alpha)=([\rho(a),X], \ldr{a}\alpha).\]
The basic curvature
$R_\Delta^{\rm bas}\colon \Gamma(A)\times\Gamma(A)\times\Gamma(TM\oplus A^*)\to\Gamma(A\oplus T^*M)$
is given by 
\begin{align*}
R_\Delta^{\rm bas}(a,b)(X,\xi)=&-\Omega_{(X,\xi)}[a,b] +\ldr{a}\left(\Omega_{(X,\xi)}b\right)-\ldr{b}\left(\Omega_{(X,\xi)}a\right)\\
&\qquad                                                     + \Omega_{\nabla^{\rm bas}_b(X,\xi)}a-\Omega_{\nabla^{\rm bas}_a(X,\xi)}b.
\end{align*}

A straightforward computation, that also resembles much the one in
\cite[Section 3.2]{GrJoMaMe18} for the tangent double of a Lie
algebroid, shows that the Dorfman $2$-representation and the self-dual
$2$-representation form a matched pair.  Hence, $TA\oplus_AT^*A$ is an
LA-Courant algebroid.

\subsubsection{The LA-Courant algebroid defined by a double Lie algebroid}\label{Poi_lie_2_def_by_matched_ruth}
 More generally, let
\begin{equation*}
\begin{xy}
  \xymatrix{
    D\ar[r]^{\pi_B}\ar[d]_{\pi_A}& B\ar[d]^{q_B}\\
    A\ar[r]_{q_A}&M }
\end{xy}
\end{equation*}
(with core $C$) be a double Lie algebroid. Then the pair $(D,D\duer B)$ of vector bundles over $B$ is a Lie
bialgebroid, with $D\duer B$ endowed with the trivial Lie algebroid
structure.  We get a linear Courant algebroid $D\oplus_B (D\duer B)$
over $B$ with side $A\oplus C^*$
\begin{equation*}
\begin{xy}
  \xymatrix{
   D\oplus_B (D\duer B)\ar[r]\ar[d]&B\ar[d]\\
   A\oplus C^* \ar[r]&M }
\end{xy}
\end{equation*}
and core $C\oplus A^*$.  The Courant algebroid structure
is linear, see \cite{Jotz17b}, \S 4.4.2. Recall also from there that a
linear decomposition $\Sigma\colon A\times_M B\times_MC\to D$ defines
a Lagrangian decomposition $\widetilde\Sigma$ of the metric double vector bundle
$D\oplus_B(D\duer B)$.
Further, the linear decomposition $\Sigma$ of $D$ yields a
matched pair of $2$-representations as in \S\ref{matched_pair_2_rep_sec}.

In the Lagrangian decomposition, the linear Courant algebroid
structure is equivalent to the split Lie 2-algebroid $(\partial_B\circ \pr_C\colon C\oplus A^*\to B, \Delta,\nabla,R)$
defined by
\begin{equation}\label{dorf_2_rep_1}
\begin{split}
\Delta\colon&\Gamma(A\oplus C^*)\times\Gamma(C\oplus A^*)\to\Gamma(C\oplus A^*)\\
&\Delta_{(a,\gamma)}(c,\alpha)=(\nabla_ac,\ldr{a}\alpha+\langle\nabla^*_\cdot\gamma,c\rangle),
\end{split}
\end{equation}
\begin{equation}\label{dorf_2_rep_2}
\begin{split}
  \nabla\colon&\Gamma(A\oplus C^*)\times\Gamma(B)\to\Gamma(B), \qquad \nabla_{(a,\gamma)}b=\nabla_ab
\end{split}
\end{equation}
with $A\oplus C^*$ anchored by $\rho_A$,
and 
$\omega\in \Omega^3(A\oplus C^*,B^*)$ defined by 
\begin{equation}\label{dorf_2_rep_3}
\ip{(a_2,\gamma_2)}\ip{(a_1,\gamma_1)}\omega=\left(R(a_1,a_2), \langle\gamma_2,R(a_1,\cdot)\rangle
+\langle\gamma_1,R(\cdot,a_2)\rangle\right)
\end{equation}
as a section of $\operatorname{Hom}(B, C\oplus A^*)$.

The direct sum $D\oplus_BD\duer B$ over $B$ 
has also a VB-algebroid structure $(D\oplus_BD\duer B\to A\oplus
C^*, B\to M)$ with core $C\oplus A^*$. 
The linear decomposition $\tilde\Sigma\colon
B\times_M(A\oplus C^*)\times_M (C\oplus A^*)\to D\oplus_B(D\duer B)$ defines the
$2$-representation of $B$
\begin{equation}\label{double_2_rep}
(\partial_A\oplus\partial_A^*\colon C\oplus A^*\to A\oplus C^*, \nabla^A\oplus{\nabla^C}^*,\nabla^C\oplus{\nabla^A}^*,
R\oplus(-R^*)),
\end{equation}
see \cite{Jotz18b}, \S 4.5.2.

A straightforward computation shows that the matched pair conditions for
the $2$-representations describing the sides of $D$ imply that the
$2$-representation \eqref{double_2_rep} and the split Lie 2-algebroid
\eqref{dorf_2_rep_1}--\eqref{dorf_2_rep_3}
form a matched
pair. Hence, $(D\oplus_B(D\duer B),A\oplus C^*,B,M)$ has a
natural LA-Courant algebroid structure.  In the same manner,
$(D\oplus_A(D\duer A),B\oplus C^*,A,M)$ has a natural LA-Courant
algebroid structure.  Hence, we get the following theorem.
\begin{theorem}
  Consider a matched pair of $2$-representations with the usual
  notation.  Then the split $[2]$-manifold $(A\oplus C^*)[-1]\oplus
  B^*[-2]$ endowed with the semi-direct Lie $2$-algebroid structure in
  \eqref{dorf_2_rep_1}--\eqref{dorf_2_rep_3} and the Poisson bracket defined by
  \eqref{double_2_rep}, is a split
  Poisson Lie $2$-algebroid.

By symmetry, the split $[2]$-manifold $(B\oplus C^*)[-1]\oplus
  A^*[-2]$ also inherits a split Poisson Lie $2$-algebroid structure.
\end{theorem}

In the case of the double Lie algebroid $TA$, for $A\to M$ a Lie
algebroid, the two LA-Courant algebroids obtained in this manner are
$TA\oplus_A T^*A$ described in Section \ref{PontLA_LACourant}, and the
tangent prolongation as in Section \ref{TCourant_LACourant} of the
Courant algebroid $A\oplus A^*\to M$;
$\lb (a_1,\alpha_1), (a_2,\alpha_2)\rb=([a_1,a_2], \ldr{a_1}\alpha_2
-\ip{a_2}\dr\alpha_1)$ for $a_1,a_2\in\Gamma(A)$ and
$\alpha_1,\alpha_2\in\Gamma(A^*)$.

\section{The core of an LA-Courant algebroid}\label{core_LA_CA}
We prove in this section that the core of an LA-Courant algebroid
inherits a natural structure of degenerate Courant algebroid.  We
discuss some examples and we deduce a new way of describing the
equivalence between Courant algebroids and symplectic Lie
2-algebroids.
\begin{theorem}\label{core_courant}
  Let $(\mathbb E, B, Q,M)$ be an LA-Courant algebroid and choose a
  Lagrangian splitting. Then the core $Q^*$ inherits the structure of
  a degenerate Courant algebroid over $M$, with the anchor
  $\rho_{Q^*}=\rho_Q\partial_Q$, the map
  $\mathcal D=\rho_Q^*\dr\colon C^\infty(M)\to\Gamma(Q^*)$, the
  pairing defined by
  $\langle \tau_1, \tau_2\rangle_{Q^*}=\langle
  \tau_1, \partial_Q\tau_2\rangle$
  and the bracket defined by
  \begin{equation*}
    \lb\tau_1,
    \tau_2\rb_{Q^*}=\Delta_{\partial_Q\tau_1}\tau_2-\nabla^*_{\partial_B\tau_2}\tau_1
    \end{equation*}
  for all $\tau_1,\tau_2\in\Gamma(Q^*)$.  This structure does not
  depend on the choice of the Lagrangian splitting, and the map
  $\partial_B\colon Q^*\to B$ is compatible with the brackets and the
  anchors: $\rho_B\partial_B=\rho_{Q^*}$ and
\begin{equation}\label{partial_B_morphism}
\partial_B\lb\tau_1,\tau_2\rb_{Q^*}=[\partial_B\tau_1,\partial_B\tau_2]
\end{equation}
for all $\tau_1,\tau_2\in\Gamma(Q^*)$.
\end{theorem}

\begin{proof}
  Theorem \ref{LA-Courant} states that the $2$-representation is
  self-dual and that the $2$-representation and the split Lie
  2-algebroid defined by a Lagrangian splitting form a matched pair.
  Hence, by \S\ref{metricVBA}, the pairing
  $\langle \cdot\,, \cdot\rangle_{Q^*}$ is symmetric.  The map
  $\rho_Q^*\dr\colon C^\infty(M)\to\Gamma(Q^*)$ satisfies
  $\langle \tau, \rho_Q^*\dr
  f\rangle_{Q^*}=\langle \partial_Q\tau, \rho_Q^*\dr
  f\rangle=\langle \rho_Q\partial_Q\tau, \dr
  f\rangle=(\rho_Q\circ\partial_Q)(\tau)(f)$
  for all $\tau\in\Gamma(Q^*)$ and $f\in C^\infty(M)$.  We check
  (CA1)--(CA5) in the definition of a degenerate Courant algebroid
  (see Page \pageref{CA_def}). Condition (CA5) is immediate by
  definition of the bracket. Condition (CA3) is exactly
  \eqref{almost_C}.  Note that (M2) and \eqref{D1} imply
\begin{align*}
  \partial_B\lb \tau_1,
  \tau_2\rb_{Q^*}&=\partial_B(\Delta_{\partial_Q\tau_1}\tau_2-\nabla^*_{\partial_B\tau_2}\tau_1) \nonumber\\
  &=\nabla_{\partial_Q\tau_1}\partial_B\tau_2-[\partial_B\tau_2,\partial_B\tau_1]
  -\nabla_{\partial_Q\tau_1}\partial_B\tau_2=[\partial_B\tau_1,\partial_B\tau_2].
\end{align*}
This and $\rho_Q\circ\partial_Q=\rho_B\circ\partial_B$ (see Remark
\ref{remark_simplifications}) imply the last claim of the theorem.
In the same manner (M1) and
$\nabla\circ\partial_Q=\partial_Q\circ\nabla^*$ (by Definition
of a $2$-representation) imply the equation
\begin{align}\label{partial_Q_preserves}
\partial_Q\lb \tau_1,
\tau_2\rb_{Q^*}&=\lb\partial_Q\tau_1,\partial_Q\tau_2\rb_\Delta+\partial_B^*\langle
\tau_2, \nabla_\cdot\partial_Q\tau_1\rangle.
\end{align}
The compatibility of the bracket with the anchor (CA4) follows then
immediately from \eqref{partial_Q_preserves} with \eqref{rho_delta},
or from \eqref{partial_B_morphism} with
$\rho_Q\circ\partial_Q=\rho_B\circ\partial_B$.  Next we check (CA2)
using \eqref{partial_Q_preserves} and $\nabla\circ\partial_Q=\partial_Q\circ\nabla^*$. We have
\begin{align*}
  &\rho_Q\partial_Q(\tau_1)\langle\tau_2,
  \tau_3\rangle_{Q^*}-\langle\lb \tau_1, \tau_2\rb_{Q^*},
  \tau_3\rangle_{Q^*}-\langle\tau_2, \lb \tau_1,
  \tau_3\rb_{Q^*}\rangle_{Q^*}\\
  =\,&\langle \tau_2,
  \lb \partial_Q\tau_1, \partial_Q\tau_3\rb_\Delta\rangle
  +\langle\nabla^*_{\partial_B\tau_2}
  \tau_1, \partial_Q\tau_3\rangle -
  \langle\tau_2, \partial_Q\lb \tau_1,
  \tau_3\rb_{Q^*}\rangle=0
\end{align*}
since $\langle\nabla^*_{\partial_B\tau_2}\tau_1, \partial_Q\tau_3\rangle =\langle\partial_Q\nabla^*_{\partial_B\tau_2}\tau_1, \tau_3\rangle=\langle\nabla^*_{\partial_B\tau_2}(\partial_Q\tau_1), \tau_3\rangle$.
Finally we check the Jacobi identity (CA1).  Using
\eqref{partial_B_morphism} and \eqref{partial_Q_preserves}, we have
for $\tau_1,\tau_2,\tau_3\in\Gamma(Q^*)$:
\begin{equation*}
\begin{split}
  &\lb \lb \tau_1,\tau_2\rb_{Q^*},\tau_3\rb_{Q^*}+\lb \tau_2,\lb
  \tau_1,\tau_3\rb_{Q^*}\rb_{Q^*}-\lb \tau_1,\lb \tau_2,\tau_3\rb_{Q^*}\rb_{Q^*}\\
  &=\Delta_{\lb\partial_Q\tau_1,\partial_Q\tau_2\rb+\partial_B^*\langle
    \tau_2,
    \nabla_\cdot\partial_Q\tau_1\rangle}\tau_3-\nabla^*_{\partial_B\tau_3}(
  \Delta_{\partial_Q\tau_1}\tau_2-\nabla^*_{\partial_B\tau_2}\tau_1)\\
  &\quad +\Delta_{\partial_Q\tau_2}
  (\Delta_{\partial_Q\tau_1}\tau_3-\nabla^*_{\partial_B\tau_3}\tau_1)
  -\nabla^*_{[\partial_B\tau_1,\partial_B\tau_3]}\tau_2\\
  &\quad-\Delta_{\partial_Q\tau_1}(\Delta_{\partial_Q\tau_2}\tau_3-\nabla^*_{\partial_B\tau_3}\tau_2)+\nabla^*_{[\partial_B\tau_2,\partial_B\tau_3]}\tau_1\\
  &=R_\nabla(\partial_B\tau_3,\partial_B\tau_2)\tau_1+\nabla^*_{\partial_B\tau_2}\nabla^*_{\partial_B\tau_3}\tau_1-R_\Delta(\partial_Q\tau_1,\partial_Q\tau_2)\tau_3+\Delta_{\partial_B^*\langle
    \tau_2, \nabla_\cdot\partial_Q\tau_1\rangle}\tau_3\\
  &\quad -\nabla^*_{\partial_B\tau_3}
  \Delta_{\partial_Q\tau_1}\tau_2+\Delta_{\partial_Q\tau_1}\nabla^*_{\partial_B\tau_3}\tau_2-\Delta_{\partial_Q\tau_2}\nabla^*_{\partial_B\tau_3}\tau_1-\nabla^*_{[\partial_B\tau_1,\partial_B\tau_3]}\tau_2.
\end{split}
\end{equation*} 
Using the equalities 
$R_\Delta(\partial_Q\tau_1, \partial_Q\tau_2)\tau_3=\langle\ip{\partial_Q\tau_2
}\ip{\partial_Q\tau_1}\omega,\partial_B\tau_3\rangle$ by
\eqref{omega_dorfman_curv}, $R_\nabla(\partial_B\tau_3,\partial_B\tau_2)\tau_1=R(\partial_B\tau_3,\partial_B\tau_2)\partial_Q\tau_1$
by the definition of a $2$-representation, and \eqref{(LC10)},
this is
\begin{equation*}
\begin{split}
  &-\Delta_{\nabla_{\partial_B\tau_3}\partial_Q\tau_1}\tau_2+\nabla^*_{\nabla_{\partial_Q\tau_1}\partial_B\tau_3}\tau_2+\langle\nabla_{\nabla_\cdot \partial_B\tau_3}\partial_Q\tau_1,
  \tau_2\rangle \\
  &\quad
  +\nabla^*_{\partial_B\tau_2}\nabla^*_{\partial_B\tau_3}\tau_1+\Delta_{\partial_B^*\langle
    \tau_2,
    \nabla_\cdot\partial_Q\tau_1\rangle}\tau_3-\nabla^*_{[\partial_B\tau_1,\partial_B\tau_3]}\tau_2-\Delta_{\partial_Q\tau_2}\nabla^*_{\partial_B\tau_3}\tau_1.
\end{split}
\end{equation*}
By \eqref{almost_C}, we can replace 
\begin{equation*}
\begin{split}
  &-\Delta_{\nabla_{\partial_B\tau_3}\partial_Q\tau_1}\tau_2+\nabla^*_{\partial_B\tau_2}\nabla^*_{\partial_B\tau_3}\tau_1-\Delta_{\partial_Q\tau_2}\nabla^*_{\partial_B\tau_3}\tau_1\\
  =&-\Delta_{\partial_Q(\nabla^*_{\partial_B\tau_3}\tau_1)}\tau_2+\nabla^*_{\partial_B\tau_2}(\nabla^*_{\partial_B\tau_3}\tau_1)-\Delta_{\partial_Q\tau_2}(\nabla^*_{\partial_B\tau_3}\tau_1)
\end{split}
\end{equation*}
by 
\[-\nabla^*_{\partial_B(\nabla^*_{\partial_B\tau_3}\tau_1)}\tau_2-\rho_Q^*\dr
\langle \tau_2, \partial_Q\nabla^*_{\partial_B\tau_3}\tau_1\rangle
\]
and we get
\begin{equation*}
\begin{split}
  &\nabla^*_{\nabla_{\partial_Q\tau_1}\partial_B\tau_3
    -[\partial_B\tau_1,\partial_B\tau_3]-\partial_B(\nabla_{\partial_B\tau_3}\tau_1)}\tau_2+\langle\nabla_{\nabla_\cdot \partial_B\tau_3}\partial_Q\tau_1,
  \tau_2\rangle \\
  &\quad +\Delta_{\partial_B^*\langle \tau_2,
    \nabla_\cdot\partial_Q\tau_1\rangle}\tau_3 -\rho_Q^*\dr \langle
  \tau_2, \partial_Q\nabla^*_{\partial_B\tau_3}\tau_1\rangle
\end{split}
\end{equation*}
Since
$\nabla_{\partial_Q\tau_1}\partial_B\tau_3
-[\partial_B\tau_1,\partial_B\tau_3]-\partial_B(\nabla_{\partial_B\tau_3}\tau_1)=0$ by (M2),
we finally get
\begin{equation}\label{remainder}
\begin{split}
&\langle\nabla_{\nabla_\cdot \partial_B\tau_3}\partial_Q\tau_1, \tau_2\rangle
+\Delta_{\partial_B^*\langle
\tau_2,
\nabla_\cdot\partial_Q\tau_1\rangle}\tau_3
-\rho_Q^*\dr
\langle \tau_2, \partial_Q\nabla_{\partial_B\tau_3}\tau_1\rangle.
\end{split}
\end{equation}
We write $\beta:=\langle\nabla_\cdot \partial_Q\tau_1,
\tau_2\rangle\in\Gamma(B^*)$. Since
$\rho_Q\circ\partial_Q=\rho_B\circ\partial_B$ and
$\nabla\circ\partial_Q=\partial_Q\circ\nabla$, we find
$\beta=\langle\partial_Q\nabla^*_\cdot \tau_1,
\tau_2\rangle=\langle\nabla^*_\cdot \tau_1,
\partial_Q\tau_2\rangle\in\Gamma(B^*)$.  To see that
\eqref{remainder}, which is a section of $Q^*$, vanishes, we evaluate it on an
arbitrary $q\in\Gamma(Q)$. We use \eqref{D1} and the definition of a
$2$-representation and we get
\begin{equation*}
\begin{split}
  &\langle\nabla_{\nabla_q \partial_B\tau_3}\partial_Q\tau_1,
  \tau_2\rangle +\langle\Delta_{\partial_B^*\beta}\tau_3, q\rangle
  -\rho_Q(q) \langle
  \tau_2, \partial_Q\nabla^*_{\partial_B\tau_3}\tau_1\rangle\\
  =&\langle\nabla_{\partial_B\Delta_q\tau_3}\partial_Q\tau_1,
  \tau_2\rangle +\langle\Delta_{\partial_B^*\beta}\tau_3, q\rangle
  -\rho_Q(q) \langle
  \tau_2, \partial_Q\nabla^*_{\partial_B\tau_3}\tau_1\rangle\\
  =&\langle\beta, \partial_B\Delta_q\tau_3\rangle
  +\langle\Delta_{\partial_B^*\beta}\tau_3, q\rangle -\rho_Q(q)
  \langle\beta,\partial_B\tau_3\rangle =-\langle\lb
  q, \partial_B^*\beta\rb_\Delta,\tau_3\rangle+\langle\Delta_{\partial_B^*\beta}\tau_3,
  q\rangle.
\end{split}
\end{equation*}
Since the Dorfman connection $\Delta$ is dual to the skew-symmetric
dull bracket $\lb\cdot\,,\cdot\rb_\Delta$, this is $\rho_Q(\partial_B^*\beta)\langle q,\tau_3\rangle$.
Because $\rho_Q\circ\partial_B^*=0$ by \eqref{rho_delta}, we can
conclude.

We finally prove that the degenerate Courant algebroid structure does
not depend on the choice of the Lagrangian splitting.  Clearly the
pairing and anchor are independent of the splitting, so we only need
to check that the bracket remains the same if we choose a different
Lagrangian splitting. Assume that
$\Sigma_1,\Sigma_2\colon B\times_MQ\times_MQ^*\to\mathbb E$ are two
Lagrangian decompositions.  Then there is
$\phi\in\Gamma(B^*\otimes Q^*\wedge Q^*)$ such that for all
$(b_m,q_m,\tau_m)\in B\times_MQ\times_MQ^*$,
$\Sigma_1(b_m,q_m,\tau_m)=\Sigma_2(b_m,q_m,\tau_m)+_B(0^{\mathbb
  E}_{b_m}+_Q\phi(b_m,q_m))$.
Then by Remark 2.12 of \cite{GrJoMaMe18}, we have 
$\nabla^2_b\tau=\nabla^1_b\tau+\phi(b,\partial_Q\tau)$ for all
$b\in\Gamma(B)$ and $\tau\in\Gamma(Q^*)$. By Proposition 4.7 in \cite{Jotz17b},
we have 
$\Delta^2_q\tau=\Delta^1_q\tau+\phi(\partial_B\tau,q)$
for all $q\in\Gamma(Q)$ and $\tau\in\Gamma(Q^*)$.
 Then
$\Delta^2_{\partial_Q\tau_1}\tau_2-\nabla^2_{\partial_B\tau_2}\tau_1=
\Delta^1_{\partial_Q\tau_1}\tau_2+\phi(\partial_B\tau_2,\partial_Q\tau_1)
-\nabla^1_{\partial_B\tau_2}\tau_1-\phi(\partial_B\tau_2,\partial_Q\tau_1)
=\Delta^1_{\partial_Q\tau_1}\tau_2-\nabla^1_{\partial_B\tau_2}\tau_1$.
\end{proof}

\begin{example}[Tangent Courant algebroid]
  Consider the example described in \S\ref{TCourant_LACourant}.  The
  degenerate Courant algebroid structure on the core $\mathsf E$ of
  $T\mathsf E$ is just the initial Courant algebroid structure on
  $\mathsf E$ since
  $\Delta_{e_1}e_2=\lb e_1, e_2\rb +\nabla_{\rho(e_2)}e_1$ by
  definition and so
\[\Delta_{e_1}e_2-\nabla_{\rho(e_2)}e_1
=\lb e_1, e_2\rb.
\]
\end{example}

We have hence proved that the Courant algebroid associated to a
symplectic Lie 2-algebroid can be defined directly from any of the
splittings of the Lie 2-algebroid, and so does not need to be obtained
as a derived bracket.

\begin{theorem}\label{culminate}
  Let $\mathcal M$ be a symplectic Lie 2-algebroid over a base
  manifold $M$. Then the corresponding Courant algebroid is defined as
  follows. Choose any splitting
  $\mathcal M\simeq Q[-1]\oplus T^*M[-2]$ of the underlying symplectic
  $[2]$-manifold. Then $Q\simeq Q^*$ via $\partial_Q$ and $Q^*$
  inherits a nondegenerate pairing given by
  $\langle\tau_1,\partial_Q\tau_2\rangle$ for
  $\tau_1,\tau_2\in\Gamma(Q^*)$.  The morphism
  $\partial_{TM}\colon Q^*\to TM$ of the split Lie 2-algebroid
  structure on $Q[-1]\oplus T^*M[-2]$ defines an anchor on $Q^*$. We
  have further a bracket $\lb\cdot\,,\cdot\rb_{Q^*}$ defined on
  $\Gamma(Q^*)$ by
  $\lb \tau_1,
  \tau_2\rb_{Q^*}=\Delta_{\partial_Q\tau_1}\tau_2-\{\partial_{TM}\tau_2,\tau_1\}$
  and that does not depend on the choice of the splitting. This
  anchor, pairing and bracket define a Courant algebroid structure on
  $Q^*$.  
\end{theorem}

Note that the Courant algebroid structure is transported to $Q$ by
$\Beta=\partial_Q\colon Q^*\to Q$ for our result to be consistent with
the construction in
\S\ref{TCourant_LACourant}.

\begin{example}[Core of the standard Courant algebroid over a Lie algebroid]
  Consider now the example discussed in \S\ref{PontLA_LACourant}; namely
  the standard LA-Courant algebroid
\begin{equation*}
\begin{xy}
\xymatrix{
TA\oplus_AT^*A \ar[d]\ar[r]& TM\oplus A^*\ar[d]\\
 A\ar[r]& M}
\end{xy}
\end{equation*}
over a Lie algebroid $A$.  The degenerate Courant algebroid structure
on the core $A\oplus T^*M$ of $TA\oplus T^*A$ is here given by
$\rho_{A\oplus T^*M}(a,\theta)=\rho_A(a)$, \[\langle (a_1,\theta_1),
(a_2, \theta_2)\rangle_{A\oplus T^*M}=\langle (a_1,\theta_1),
(\rho_A,\rho_A^*)(a_2, \theta_2)\rangle\] and the bracket defined
by \[\lb (a_1,\theta_1), (a_2, \theta_2)\rb_{A\oplus
  T^*M}=([a_1,a_2],\ldr{\rho_A(a_1)}\theta_2-\ip{\rho_A(a_2)}\dr\theta_1)\]
for all $a,a_1,a_2\in\Gamma(A)$ and
$\theta,\theta_1,\theta_2\in\Omega^1(M)$. To see this, use Lemma 5.16
in \cite{Jotz18a} or the next example; this degenerate Courant
algebroid plays a crucial role in the infinitesimal description of
Dirac groupoids \cite{Jotz19}, i.e.~in the definition of \emph{Dirac
  bialgebroids}.

\end{example}
\begin{example}[LA-Courant algebroid associated to a double Lie algebroid]
  More generally, the LA-Courant algebroids (and the
  corresponding Poisson Lie $2$-algebroids) considered in
  \S\ref{Poi_lie_2_def_by_matched_ruth}
and Theorem \ref{core_courant} yield the following application.

A matched pair of $2$-representations as in
\S\ref{matched_pair_2_rep_sec} defines two degenerate Courant
algebroids.  The first one is $C\oplus A^*\to M$ with the anchor
$\rho_{C\oplus A^*}\colon C\oplus A^*\to TM$ defined by
$\rho_{C\oplus
  A^*}=\rho_A\circ\pr_A\circ(\partial_A\oplus\partial_A^*)=\rho_B\circ\partial_B\circ\pr_C$.
The pairing is defined by
\[\langle (c_1,\alpha_1), (c_2,\alpha_2)\rangle_{C\oplus A^*}=\langle
\alpha_1, \partial_Ac_2\rangle+\langle \alpha_2, \partial_Ac_1\rangle
\]
for all $\alpha_1,\alpha_2\in\Gamma(A^*)$ and $c_1,c_2\in\Gamma(C)$, 
and the bracket by
\begin{equation*}
\begin{split}
  \lb (c_1,\alpha_1), (c_2,\alpha_2)\rb_{C\oplus A^*}
  &=\Delta_{(\partial_Ac_1,\partial_A^*\alpha_1)}(c_2,\alpha_2)-\nabla_{\partial_Bc_2}(c_1,\alpha_1)\\
  &=(\nabla_{\partial_Ac_1}c_2-\nabla_{\partial_Bc_2}c_1,\ldr{\partial_Ac_1}\alpha_2+\langle\nabla^*_\cdot\partial_A^*\alpha_1,c_2\rangle-\nabla^*_{\partial_Bc_2}\alpha_1)\\
  &=([c_1,c_2],\ldr{\partial_Ac_1}\alpha_2-\ip{\partial_Ac_2}\dr_A\alpha_1).
\end{split}
\end{equation*}
Note that the restriction to $\Gamma(C)$ of the Courant bracket is the
Lie algebroid bracket induced on $C$ by the matched pair, see
\S\ref{matched_pair_2_rep_sec}.

The second degenerate Courant algebroid is 
$C\oplus B^*\to M$ with the anchor $\rho_{C\oplus
  B^*}\colon C\oplus B^*\to TM$ defined by
$\rho_B\circ\pr_B\circ(\partial_B\oplus\partial_B^*)=\rho_B\circ\partial_B\circ\pr_C$,
the pairing defined by
\[\langle (c_1,\beta_1), (c_2,\beta_2)\rangle_{C\oplus B^*}=\langle
\beta_1, \partial_Bc_2\rangle+\langle \beta_2, \partial_Bc_1\rangle
\]
for all $\beta_1,\beta_2\in\Gamma(B^*)$ and $c_1,c_2\in\Gamma(C)$, 
and the bracket
\begin{equation*}
\begin{split}
  \lb (c_1,\beta_1), (c_2,\beta_2)\rb_{C\oplus B^*}
  &=(\nabla_{\partial_Bc_1}c_2-\nabla_{\partial_Ac_2}c_1,\ldr{\partial_Bc_1}\beta_2-\ip{\partial_Bc_2}\dr_B\beta_1).
\end{split}
\end{equation*}
Here again, by (m1), the restriction to $\Gamma(C)$ of the Courant bracket is the
Lie algebroid bracket induced on $C$ by the matched pair, as in
\S\ref{matched_pair_2_rep_sec}.

\end{example}

\section{VB-Dirac structures, LA-Dirac structures and pseudo Dirac
  structures}\label{sec:dirac}

In this section, we study isotropic subalgebroids of metric
VB-algebroids and Dirac structures in VB- and LA-Courant
algebroids. While we paid attention in the preceding sections to
bridge $[2]$-geometric objects to geometric structures on metric
double vector bundles, we are here more interested in classifications
of VB-Dirac structures via the simple geometric descriptions that we
found before for VB-Courant algebroids and LA-Courant algebroids.

\subsection{VB-Dirac structures}
Let $(\mathbb E, B, Q,M)$ be a VB-Courant algebroid with core $Q^*$
and anchor $\Theta\colon\mathbb E\to TB$.  Let $D$ be a double vector
subbundle structure over $B'\subseteq B$ and $U\subseteq Q$ and with
core $K$.  Choose a linear splitting $\Sigma\colon
B\times_M Q\to \mathbb E$ that is adapted\footnote{Since $D$ and
  $\mathbb E$ are both double vector bundles, there exist two
  decompositions $\mathbb I_D\colon B'\times_MU\times_MK \to D$ and
  $\mathbb I\colon  B\times_M Q\times_M Q^*\to\mathbb E$. Let
  $\iota\colon D\to \mathbb E$ be the double vector bundle inclusion,
  over $\iota_U\colon U\to Q$ and $\iota_{B'}\colon B'\to B$, and with
  core morphism $\iota_K\colon K\to Q^*$.  Then  the map
  $\mathbb I\inv\circ\iota\circ\mathbb I_D \colon
  B'\times_MU\times_MK\to B\times_M Q\times_M Q^*$ defines a morphism
  $\phi\in\Gamma({B'}^*\otimes U^*\otimes Q^*)$ by
  $(\mathbb I\inv\circ\iota\circ\mathbb I_D) (b_m,u_m,k_m)=(\iota_B(b_m),\iota_U(u_m),\iota_K(k_m)+\phi(b_m,u_m))$. Using
  local basis sections of $B$ and $Q$ adapted to $B'$ and $U$ and a
  partition of unity on $M$, extend $\phi$ to $\hat\phi\in
  \Gamma(B^*\otimes Q\otimes Q^*)$. Then define a new decomposition
  $\tilde{\mathbb I}\inv\colon \mathbb E \to B\times_M Q\times_M Q^*$ by
  $\tilde{\mathbb I}\inv(e)= \mathbb
  I\inv(e)+_B(b_m,0^Q_m,-\hat\phi(b_m,q_m))=\mathbb
  I\inv(e)+_Q(0^B_m,q_m,-\hat\phi(b_m,q_m))$ for $e\in\mathbb E$ with
  $\pi_B(e)=b_m$ and $\pi_Q(e)=q_m$. Then $(\tilde{\mathbb
    I}^{-1}\circ\iota\circ\mathbb
  I_D)(b_m,u_m,k_m)=(\iota_B(b_m),\iota_U(u_m),\iota_K(k_m))$ for all
  $(b_m,u_m,k_m)\in B'\times_M U\times_MK$.  The corresponding linear
  splitting $\tilde\Sigma\colon B\times _M Q\to \mathbb E$, $\tilde
  \Sigma(b_m,q_m)=\tilde{\mathbb I}(b_m,q_m,0^{Q^*}_m)$ sends
  $(\iota_{B'}(b_m),\iota_U(u_m))$ to $\iota(\mathbb
  I_D(b_m,u_m,0^K_m)\in \iota(D)$.} to $D$, i.e. such that
$\Sigma(B'\times_MU)\subseteq D$.  Then $D$ is spanned as a vector
bundle over $B'$ by the sections $\sigma_Q(u)\an{B'}$ for all
$u\in\Gamma(U)$ and $\tau^\dagger\an{B'}$ for all $\tau\in\Gamma(K)$.

We get immediately the following proposition.
\begin{proposition}\label{prop1}
  In the situation described above, the double subbundle $D\subseteq
  \mathbb E$ over $B'$ is isotropic if and only if $K\subseteq
  U^\circ$ and $\Lambda$ as in \eqref{lambda} sends $U\otimes U$
  to ${B'}^\circ$.
\end{proposition}

\begin{proposition}\label{prop2}
  In the situation described above, $D$ is maximal isotropic if and
  only if $U=K^\circ$ and $\Lambda$  sends
  $U\otimes U$ to ${B'}^\circ$.
\end{proposition}

Now we can prove that if $D$ is maximal isotropic, then there exists a \emph{Lagrangian} 
splitting of $\mathbb E$ that is adapted to $D$. 
\begin{corollary}
  Let $(\mathbb E, B, Q, M)$ be a metric double vector bundle and
  $D\subseteq \mathbb E$ a maximal isotropic double subbundle.  Then
  there exists a Lagrangian splitting that is adapted to $D$.
\end{corollary}

\begin{proof}
  As before, let $U\subseteq Q$ and $B'\subseteq B$ be the sides of
  $D$. Then by Proposition \ref{prop2} the core of $D$ is the vector
  bundle $U^\circ\subseteq Q^*$. Choose a linear splitting
  $\Sigma\colon Q\times_M B\to \mathbb E$ that is adapted to $D$. Then
  $D$ is spanned as a vector bundle over $B'$ by the sections
  $\sigma_Q(u)\an{B'}$ for all $u\in\Gamma(U)$ and
  $\tau^\dagger\an{B'}$ for all $\tau\in\Gamma(U^\circ)$.  As in
  \eqref{new_lift}, transform $\Sigma$ into a new Lagrangian linear
  splitting $\Sigma'$. We need to show that
  $\sigma_Q'(u)\an{B'}-\sigma_Q(u)\an{B'}$ is equivalent to a section
  of ${B'}^*\otimes U^\circ$ for all $u\in\Gamma(U)$. But
  $\sigma_Q'(u)-\sigma_Q(u)=\widetilde{\frac{1}{2}\Lambda(u,\cdot)}$
  by construction and, since $D$ is isotropic, we have
  $\Lambda(u,u')\an{B'}=0$ for all $u,u'\in\Gamma(U)$.
\end{proof}

\begin{remark}\label{invariant_part}
  Consider a Courant algebroid $\mathsf E\to M$ and its tangent double
  $T\mathsf E$.  Recall from Example \ref{metric_connections} that
  Lagrangian splittings of $T\mathsf E$ are equivalent to metric
  connections $\mx(M)\times\Gamma(\mathsf E)\to \Gamma(\mathsf E)$.
  Let $\nabla$ be such a metric connection, that is adapted to a
  maximally isotropic double subbundle $D$ over the sides $TM$ and
  $U\subseteq \mathsf E$.  Define $[\nabla]\colon
  \mx(M)\times\Gamma(U)\to\Gamma(\mathsf E/U^\perp)$ by
  $[\nabla]_Xu=\overline{\nabla_Xu}\in\Gamma(\mathsf E/U^\perp)$.  A
  second metric connection $\nabla'\colon \mx(M)\times\Gamma(\mathsf
  E)\to \Gamma(\mathsf E)$ is adapted to $D$ if and only if
  $\nabla_Xu-\nabla'_Xu\in\Gamma(U^\perp)$ for all $X\in\mx(M)$ and
  for all $u\in\Gamma(U)$. Hence, if and only if $[\nabla]=[\nabla']$.
  We call $[\nabla]$ the \textbf{invariant part of the metric
    connection adapted to $D$}.
\end{remark}

The existence of Lagrangian splittings of $\mathbb E$ adapted to
maximal isotropic double subbundles $D$ will now be used to study the
involutivity of $D$.

Note that in a very early version of this work, we studied VB-Courant
algebroids via general (not necessarily Lagrangian) linear
splittings. We found some more general objects than split Lie
2-algebroids; involving also $\Lambda\in S^2(Q,B^*)$ defined
in \eqref{lambda}.  The study of the involutivity of general (not
necessarily isotropic) double subbundles $D$ of $\mathbb E$ is
therefore also possible in this more general framework, and yields
very similar results.

\begin{proposition}\label{char_VB_dir}
  Let $(\mathbb E, B, Q, M)$ be a VB-Courant algebroid and $D\subseteq
  \mathbb E$ a maximal isotropic double subbundle.  Choose a
  Lagrangian splitting of $\mathbb E$ that is adapted to $D$ and
  consider the corresponding split Lie 2-algebroid, denoted as
  usual.  Then $D$ is a Dirac structure in $\mathbb E$ with support
  $B'$ if and only if
\begin{enumerate}
\item $\partial_B(U^\circ)\subseteq B'$,
\item $\nabla_ub\in\Gamma(B')$ for all $u\in\Gamma(U)$ and $b\in\Gamma(B')$,
\item $\lb u_1, u_2\rb\in\Gamma(U)$ for all $u_1,u_2\in\Gamma(U)$,
\item $\ip{u_2}\ip{u_1}\omega$ restricts to a section of  $\Gamma(\operatorname{Hom}(B',U^\circ))$ for all
  $u_1,u_2\in\Gamma(U)$.
\end{enumerate}
\end{proposition}
A Dirac double subbundle $D$ of a VB-Courant algebroid $\mathbb E$ as
in the proposition is called a \textbf{VB-Dirac structure}.

\begin{proof}
  This is easy to prove using Lemma \ref{useful_for_dirac_w_support}
  on sections $\sigma_Q(u)$ and $\tau^\dagger$, for $u\in\Gamma(U)$
  and $\tau\in\Gamma(U^\circ)$.  Their
  anchors and Courant brackets can be described by
\begin{equation}\label{ruth_D_to_B}
\begin{split}
\Theta(\sigma_Q(u))&=\widehat{\nabla_u}\in\mx^l(B),\qquad 
\Theta(\tau^\dagger)=(\partial_B\tau)^\uparrow\in\mx^c(B),\\
\lb \sigma_Q(u_1),\sigma_Q(u_2)\rb&=\sigma_Q(\lb u_1, u_2\rb)-\widetilde{\ip{u_2}\ip{u_1}\omega},\\
\lb \sigma_Q(u),\tau^\dagger\rb&=(\Delta_u\tau)^\dagger,\qquad \lb \tau_1^\dagger, \tau_2^\dagger\rb=0
\end{split}
\end{equation}
for all $u,u_1,u_2\in\Gamma(U)$ and
$\tau,\tau_1,\tau_2\in\Gamma(U^\circ)$.  The vector field
$\widehat{\nabla_u}$ is tangent to $B'$ on $B'$ if and only if for all
$\beta\in\Gamma((B')^\circ)$, $\widehat{\nabla_u}(\ell_\beta)=\ell_{\nabla_u^*\beta}$
vanishes on $B'$. That is, if and only if, for all
$\beta\in\Gamma((B')^\circ)$, $\nabla_u^*\beta$
is again a section of $(B')^\circ$. This yields (2).
The vector field $(\partial_B\tau)^\uparrow$ is tangent to $B'$
if and only if $\partial_B\tau\in\Gamma(B')$. This yields (1).
Next, $\lb \sigma_Q(u),\tau^\dagger\rb=(\Delta_u\tau)^\dagger$
is a section of $D$ over $B'$ if and only if $\Delta_u\tau\in\Gamma(U^\circ)$.
Since $\Delta_u\tau\in\Gamma(U^\circ)$ for all $u\in\Gamma(U)$ and
  $\tau\in\Gamma(U^\circ)$ if and only if $\lb u_1,
  u_2\rb\in\Gamma(U)$ for all $u_1,u_2\in\Gamma(U)$, this is
  (3). Further, $\sigma_Q(\lb u_1, u_2\rb)$ takes then values in $D$
  over $B'$, and so $\lb \sigma_Q(u_1),\sigma_Q(u_2)\rb$ takes values in $D$
  over $B'$ if and only if $\ip{u_2}\ip{u_1}\omega$ restricts to a morphism $B'\to
  U^\circ$. This is (4).
\end{proof}

We get the following result for  VB-Dirac structures (with support $B$) in $\mathbb E$.

\begin{corollary}\label{LA_on_U} Let $(\mathbb E, B, Q, M)$ be a VB-Courant algebroid and
  $(D,B,U,M)\subseteq \mathbb E$ a maximal isotropic double subbundle.
  Choose a Lagrangian splitting of $\mathbb E$ that is adapted to $D$
  and consider the corresponding split Lie 2-algebroid, denoted
  as usual.  If $D$ is a Dirac structure in $\mathbb E\to B$, then $U$
  inherits a Lie algebroid structure with bracket $\lb\cdot\,,\cdot\rb
  \an{\Gamma(U)\times\Gamma(U)}$ and anchor $\rho_Q\an{U}$. This Lie
  algebroid structure does not depend on the choice of Lagrangian
  splitting.
\end{corollary}

\begin{proof}
  By (3) in Proposition \ref{char_VB_dir}, $\lb\cdot\,,\cdot\rb$
  restricts to a bracket on sections of $U$.  For $u_1,u_2,u_3$,
  $\operatorname{Jac}_{\lb\cdot\,,\cdot\rb}(u_1,u_2,u_3)=\partial_B^*\omega(u_1,u_2,u_3)=0$
  since $\omega(u_1,u_2,u_3)=0\in\Gamma(B^*)$ by (4) in Proposition \ref{char_VB_dir}.
  Hence, $U$ with the bracket
  $\lb\cdot\,,\cdot\rb\an{\Gamma(U)\times\Gamma(U)}$ and the anchor
  $\rho_Q\an{U}$ is a Lie algebroid.

If $\phi\in\Gamma(Q^*\wedge Q^*\otimes B^*)$ is the tensor
defined as in the proof of Theorem \ref{core_courant}
by a change of Lagrangian splitting adapted to $D$, then, by 
Proposition 4.7 in \cite{Jotz17b},
\[\lb u, u'\rb_1=\lb u, u'\rb_2+\partial_B^*\phi(u,u')
\]
for all $u,u'$. But since both splittings $\Sigma^1,\Sigma^2\colon
B\times_M Q\to \mathbb E$ are adapted to $D$, we know that
$\sigma^1_Q(u)$ and $\sigma^2_Q(u)$ have values in $D$, and their
difference $\sigma_Q^1(u)-\sigma_Q^2(u)=\widetilde{\phi(u)}$ is a
core-linear section of $D\to B$. Hence it must takes values in
$U^\circ$, and $\phi(u,u')$ must so vanish for all
$u,u'\in\Gamma(U)$. As a consequence, 
$\lb u, u'\rb_1=\lb u, u'\rb_2$.
\end{proof}

The following two corollaries are now easy to prove.
The first one was already given in \cite{Li-Bland12}.
\begin{corollary}\label{Lie_1_sub}
Let $(\mathcal M,\mathcal Q)$ be a Lie $2$-algebroid, and 
$(\mathbb E\to B,Q\to M)$ the corresponding VB-Courant algebroid.
Then VB-Dirac structures in $\mathbb E$ are equivalent to 
wide Lie $1$-subalgebroids of $(\mathcal M,\mathcal Q)$.
\end{corollary}

\begin{proof}
  A wide Lie subalgebroid of $(\mathcal M, \mathcal Q)$ is a wide
  $[1]$-submanifold $U[-1]$ of $\mathcal M$ such that $\mathcal
  Q_U(\mu^\star\xi)=\mu^\star(\mathcal Q(\xi))$, $\xi\in
  C^\infty(\mathcal M)$, defines a Lie algebroid structure $\mathcal
  Q_U$ on $U$. Here, $\mu\colon U[-1]\to \mathcal M$ is the 
  submanifold inclusion.

  In a splitting $Q[-1]\oplus B^*[-2]$ of $\mathcal M$, the
  homological vector field $\mathcal Q$ is given by
  \eqref{Q1}--\eqref{Q3}. Choose an open subset $V$ of $M$ with a
  local frame $(u_1,\ldots,u_r,q_{r+1},\ldots,q_l)$ of $Q$ over $V$
  such that $(u_1,\ldots,u_r)$ is a local frame for $U$ over $V$.  Let
  $(\tau_1,\ldots,\tau_l)$ be the dual smooth frame for $Q^*$ over
  $V$.  Then we have
  $\mathcal Q_U(f)=\mu^\star(\rho_Q^*\dr f)=\rho_Q^*\dr f+ U^\circ$
  for all $f\in C^\infty(M)$. This translates easily to
  $\rho_U=\rho_Q\an{U}$. Then we have
  $\mathcal Q_U(\tau_i+U^\circ)=\mathcal Q_U(\mu^\star\tau_i)=
  \mu^\star(\mathcal Q(\tau_k))=-\sum_{i<j}^r\langle \lb u_i,u_j\rb,
  \tau_k\rangle\bar\tau_i\bar\tau_j$ for $k=1,\ldots,r$.  This shows
  that the bracket on $U$ must be the restriction to $\Gamma(U)$ of
  the dull bracket on $\Gamma(Q)$. Finally
  $0=\mathcal Q_U(\mu^\star b)= \mu^\star(\mathcal
  Q(b))=-\sum_{i<j<k<r}\omega(u_i,u_j,u_k)(b)\bar\tau_i\bar\tau_j\bar\tau_k$
  for all $b\in\Gamma(B)$ shows that $\omega(u_1,u_2,u_3)$ must be zero for all
  $u_1,u_2,u_3\in\Gamma(U)$. This is equivalent to (3) in Proposition
  \ref{char_VB_dir} (with $B'=B$). Note that since $B'=B$, (1) and (2)
  in Proposition \ref{char_VB_dir} are trivially satisfied. Hence we
  can conclude.
\end{proof}

The Lie algebroid structure on $U$ is the base Lie algebroid from the
VB-algebroid $D\to B$ in the following corollary. The proof is
immediate.
\begin{corollary}\label{VB_dir_is_VB_alg}
  A VB-Dirac structure $(D,B,U,M)$ in a VB-Courant algebroid inherits
  a linear Lie algebroid structure: $(D\to B, U\to M)$ is a VB-algebroid.
\end{corollary}

\subsection{LA-Dirac structures}\label{LA-Dirac}
Assume now that $(\mathbb E\to Q,B\to M)$ is a metric VB-algebroid,
and take a maximal isotropic double subbundle $D$ of $\mathbb E$ over
the sides $U\subseteq Q$ and $B'\subseteq B$. We will study conditions
on the self-dual $2$-representation defined by a Lagrangian splitting
and the linear Lie algebroid structure on $\mathbb E\to Q$, and on $Q$
and on $B'$, for $D$ to be an isotropic subalgebroid of $\mathbb E\to
Q$ over $U$. 

Note the similarity of the following result with Proposition \ref{char_VB_dir}.
\begin{proposition}\label{char_sub_alg}
  Let $(\mathbb E, B, Q, M)$ be a metric VB-algebroid and
  $(D,B',U,M)\subseteq \mathbb E$ a maximal isotropic double
  subbundle.  Choose a Lagrangian splitting of $\mathbb E$ that is
  adapted to $D$ and consider the corresponding self-dual
  $2$-representation, denoted as usual.  Then $D\to U$ is a
  subalgebroid of $\mathbb E\to Q$ if and only if
\begin{enumerate}
\item $\partial_Q(U^\circ)\subseteq U$,
\item $\nabla_bu\in\Gamma(U)$ for all $u\in\Gamma(U)$ and $b\in\Gamma(B')$,
\item $[b_1, b_2]\in\Gamma(B')$ for all $b_1,b_2\in\Gamma(B')$,
\item $R(b_1,b_2)$ restricts to a section of
  $\Gamma(\operatorname{Hom}(U,U^\circ))$ for all
  $b_1,b_2\in\Gamma(B')$.
\end{enumerate}
\end{proposition}
\begin{proof}
  This proof is very similar to the proof of Proposition
  \ref{char_VB_dir}, and left to the reader.
\end{proof}

Now let  $(\mathbb E, Q, B, M)$ be an LA-Courant algebroid.
A VB-Dirac structure $(D,U,B',M)$ in $\mathbb E$ is an \textbf{LA-Dirac structure} 
if $(D\to U, B'\to M)$ is also a subalgebroid of $(\mathbb E\to Q,B\to M)$.
We deduce from Propositions \ref{char_VB_dir} and \ref{char_sub_alg}
a characterisation of LA-Dirac structures.

\begin{proposition}\label{char_LA_D}
  Let $(\mathbb E, B, Q, M)$ be an LA-Courant algebroid and
  $(D,B',U,M)$ a maximal isotropic double
  subbundle of $\mathbb E$.  Choose a Lagrangian splitting of $\mathbb E$ that is
  adapted to $D$ and consider the corresponding matched self-dual
  $2$-representation and split Lie 2-algebroid.  Then $D\to U$ is
  an LA-Dirac structure in $\mathbb E$ if and only if
\begin{enumerate}
\item $\partial_B(U^\circ)\subseteq B'$ and $\partial_Q(U^\circ)\subseteq U$,
\item $\nabla_ub\in\Gamma(B')$ for all $u\in\Gamma(U)$ and $b\in\Gamma(B')$,
\item $\nabla_bu\in\Gamma(U)$ for all $u\in\Gamma(U)$ and $b\in\Gamma(B')$,
\item $\lb u_1, u_2\rb\in\Gamma(U)$ for all $u_1,u_2\in\Gamma(U)$,
\item $[b_1, b_2]\in\Gamma(B')$ for all $b_1,b_2\in\Gamma(B')$,
\item $\ip{u_2}\ip{u_1}\omega$ restricts to a section of  $\Gamma(\operatorname{Hom}(B',U^\circ))$ for all
  $u_1,u_2\in\Gamma(U)$,
\item $R(b_1,b_2)$ restricts to a section of
  $\Gamma(\operatorname{Hom}(U,U^\circ))$ for all
  $b_1,b_2\in\Gamma(B')$.
\end{enumerate}
\end{proposition}

Hence, we also have the following result.
\begin{corollary}
  VB-subalgebroids $(D\to U, B\to M)$ of a metric VB-algebroid
  $(\mathbb E\to Q,B\to M)$ are equivalent to wide coisotropic
  $[1]$-submanifolds of the corresponding Poisson $[2]$-manifold.

  LA-Dirac structures $(D\to U, B\to M)$ in an LA-Courant algebroid
  $(\mathbb E\to Q,B\to M)$ are equivalent to wide coisotropic Lie
  subalgebroids of the corresponding Poisson Lie 2-algebroid.
\end{corollary}

\begin{proof}
  Let $U[-1]$ be a $[1]$-submanifold of a Poisson $[2]$-manifold
  $(\mathcal M, \{\cdot\,,\cdot\})$.  Then $U[-1]$ is coisotropic if
  and only if $\mu^\star(\xi)=\mu^\star(\eta)=0$ imply
  $\mu^\star(\{\xi,\eta\})=0$ for all
  $\xi,\eta\in C^\infty(\mathcal M)$, where
  $\mu\colon Q[-1]\to\mathcal M$ is the inclusion.  In a local
  splitting, we find easily that this implies
  $\partial_Q(U^\circ)\subseteq U$, $\nabla_b^*\tau\in\Gamma(U^\circ)$
  for all $b\in\Gamma(B)$ and $\tau\in\Gamma(U^\circ)$, and the
  restriction to $U$ of $R(b_1,b_2)$ has image in $U^\circ$. By
  Proposition \ref{char_sub_alg}, we can conclude.
  The second claim follows with Corollary \ref{Lie_1_sub}.
\end{proof}

As a corollary of Theorem \ref{LA-Courant}, Proposition \ref{char_VB_dir} and
Proposition \ref{char_sub_alg}, we get the following theorem.  
\begin{theorem}
  Let $(\mathbb E, B, Q, M)$ be an LA-Courant algebroid and
  $(D,U,B,M)\subseteq \mathbb E$ a (wide) LA-Dirac structure in $\mathbb E$.

  Then $D$ is a double Lie algebroid with the VB-algebroid structure
  in Corollary \ref{VB_dir_is_VB_alg} and the VB-algebroid structure
  $(D\to U, B\to M)$.
\end{theorem}

\begin{proof}
  Let us study the two linear Lie algebroid structures on $D$. Choose
  as before a linear splitting $\Sigma\colon B\times_MQ\to \mathbb E$
  that restricts to a linear splitting
  $\Sigma_D\colon U\times_MB\to D$ of $D$. The LA-Courant algebroid
  structure of $\mathbb E$ is then encoded as in Sections \ref{split_lie_2_rep} and
  \ref{metricVBA}, respectively, by a split Lie 2-algebroid
  $(\partial_B\colon Q^*\to B, \rho_Q\colon Q\to
  TM,\lb\cdot\,,\cdot\rb,\nabla,\omega)$ and by a self-dual
  $2$-representation $(\nabla,\nabla^*,R)$ of the Lie algebroid $B$ on
  $\partial_Q=\partial_Q^*\colon Q^*\to Q$. By Theorem
  \ref{LA-Courant}, the Dorfman $2$-representation and the
  $2$-representation form a matched pair as in Definition
  \ref{matched_pairs}.

  By Proposition \ref{char_VB_dir} and Corollary \ref{LA_on_U}, the
  restriction to $\Gamma(U)$ of the dull bracket on $\Gamma(Q)$ that
  is dual to $\Delta$ defines a Lie algebroid structure on $U$,
  $\omega\an{U\otimes U\otimes Q}$ can be seen as an element of
  $\Omega^2(U,\operatorname{Hom}(B,U^\circ))$ and since
  $\Delta_u\tau\in\Gamma(U^\circ)$ for all $u\in\Gamma(U)$ and
  $\tau\in\Gamma(U^\circ)$, the Dorfman connection $\Delta$ restricts
  to a map
  $\Delta^D\colon\Gamma(U)\times\Gamma(U^\circ)\to\Gamma(U^\circ)$.
  Since $\Delta^D_u(f\tau)=f\Delta^D_u\tau+\rho_Q(u)(f)\tau$ and
  $\Delta^D_{fu}\tau=f\Delta^D_u\tau+\langle u,\tau\rangle\rho_Q^*\dr
  f=f\Delta_u\tau$ for $f\in C^\infty(M)$, we find that this
  restriction is in fact an ordinary connection.  Since
  $\omega(u_1,u_2,u_3)$ vanishes for all $u_1,u_2,u_3\in\Gamma(U)$, it
  is then easy to see that the restrictions to sections of $U$ and
  $U^\circ$ of \eqref{D1}, \eqref{omega_dorfman_curv} and of (iv) and
  (v) in the definition of a split Lie 2-algebroid define an ordinary
  $2$-representation.  By \eqref{ruth_D_to_B}, this $2$-representation
  $(\partial_B\colon U^\circ \to B, \nabla,\Delta^D,\omega\an{U\otimes
    U\otimes Q})$ of the Lie algebroid $U$ on
  $\partial_B\colon U^\circ \to B$ encodes the VB-algebroid structure
  that $D\to B$ inherits from the Courant algebroid $\mathbb E\to B$.

  In a similar manner, we find using Proposition \ref{char_sub_alg}
  that the self-dual $2$-representation $(\partial_Q\colon Q^*\to Q,
  \nabla, \nabla^*, R\in\Omega^2(B,Q^*\wedge Q^*))$ restricts to a
  $2$-representation $(\partial_Q\colon U^\circ \to U, \nabla^U\colon
  \Gamma(B)\times\Gamma(U)\to\Gamma(U), \nabla^{U^\circ}\colon
  \Gamma(B)\times\Gamma(U^\circ)\to\Gamma(U^\circ),
  R\in\Omega^2(B,\operatorname{Hom}(U,U^\circ)))$ of $B$.

  A study of the restrictions to sections of $U$ and $U^\circ$ of the
  equations in Definition \ref{matched_pairs} shows then that (M1)
  restricts to (m2) in \S\ref{matched_pair_2_rep_sec} since
  $\partial_B^*\langle \tau, \nabla^U\cdot u\rangle=0$ for all
  $u\in\Gamma(U)$ and $\tau\in\Gamma(U^\circ)$.  The equations (M2),
  and (M3) immediately yield (m3) and (m6), respectively.  (M4) restricts
  to (m5) since $\langle R(\cdot, b)u_1, u_2\rangle=0$ for all
  $u_1,u_2\in\Gamma(U)$ and $b\in\Gamma(B)$. (M5) restricts to (m7)
  since the right-hand side of (M5) in (1) of Remark
  \ref{remark_simplifications} vanishes.  Finally, \eqref{almost_C}
  restricts to (m1) and \eqref{(LC10)} restricts to (m4) since $\langle
  \nabla_{\nabla_\cdot b}u, \tau\rangle =0$ for all $b\in \Gamma(B)$,
  $u\in\Gamma(U)$ and $\tau\in\Gamma(U^\circ)$.  Thus, the two
  $2$-representations describing the sides of $D$ given the splitting
  $\Sigma_D$ form a matched pair, which implies that $D$ is a double
  Lie algebroid (see \cite{GrJoMaMe18} or
  \S\ref{matched_pair_2_rep_sec} for a quick summary of this paper).
\end{proof}

Note finally that with a different approach as the one adopted in this
paper, we could deduce the main result in \cite{GrJoMaMe18} from our
Theorem \ref{LA-Courant}. Once one has `directly' proved that for each
double Lie algebroid $(D, A, B,M)$ with core $C$, the direct sum over
$B$ of $D$ and $D\duer B$ defines an LA-Courant algebroid
$(D\oplus_B(D\duer B), A\oplus C^*, B,M)$ as in
\S\ref{Poi_lie_2_def_by_matched_ruth}, then one can use the last
theorem to deduce the equations in \S\ref{matched_pair_2_rep_sec} from
the ones in Definition \ref{matched_pairs} and in Remark
\ref{remark_simplifications}: by construction, the double vector
subbundle $D$ of $D\oplus_B(D\duer B)$ is a VB-Dirac structure in
$D\oplus_B(D\duer B)\to B$ and a linear Lie subalgebroid in
$D\oplus_B(D\duer B)\to A\oplus C^*$.  Instead, we have chosen to use
the main theorem in \cite{GrJoMaMe18} to prove that
$(D\oplus_B(D\duer B), A\oplus C^*, B,M)$ is an LA-Courant algebroid,
see \S\ref{Poi_lie_2_def_by_matched_ruth}. By the complexity of
Li-Bland's definition of an LA-Courant algebroid, this is the most
simple approach.

\subsection{Pseudo-Dirac structures}
We explain here the notion of pseudo-Dirac structures that was
introduced in \cite{Li-Bland12,Li-Bland14} and we compare it with our approach to
VB- and LA-Dirac structures in the tangent of a Courant algebroid. Consider a
VB-Courant algebroid $\mathbb E$
with core $Q^*$, and a double vector subbundle
in $\mathbb E$ with core $K$, as in the following diagrams.
\begin{equation*}
\begin{xy}
\xymatrix{\mathbb E\ar[r]\ar[d]& Q \ar[d]\\
B\ar[r]&M
}
\end{xy}\qquad \qquad \begin{xy}
\xymatrix{D\ar[r]\ar[d]& U\ar[d]\\
B\ar[r]&M
}
\end{xy}
\end{equation*}
Consider the restriction $\mathbb E\an{U}$ of $\mathbb E$ to $U$; 
i.e. $\mathbb E\an{U}=\pi_Q\inv(U)$. 
This is a double vector bundle with sides $B$ and $U$ and
with core $Q^*$. The \textbf{total quotient} of $\mathbb E\an{U}$ by $D$ 
is the map $\mathsf q$ from 
\begin{equation*}
\begin{xy}
\xymatrix{\mathbb E\an{U}\ar[r]\ar[d]& Q \ar[d]\\
B\ar[r]&M
}
\end{xy}
\quad \text{
to
}
\quad
\begin{xy}
\xymatrix{Q^*/K\ar[r]\ar[d]& 0^M \ar[d]\\
0^M\ar[r]&M,
}
\end{xy}
\end{equation*}
defined by
\[\mathsf q(e)=\bar\tau \Leftrightarrow e-\tau^\dagger\in D.
\]
After the choice of a linear splitting of $\mathbb E$ that is adapted
to $D$, we know that each element of $\mathbb E\an{U}$ can be written
$\sigma_Q(u)(b_m)+\tau^\dagger(b_m)$ for some $u\in\Gamma(U)$,
$\tau\in\Gamma(Q^*)$ and $b_m\in B$. The image of
$\sigma_Q (u)(b_m)+\tau^\dagger(b_m)$ under $\mathsf q$ is then simply
$\bar\tau(m)$. Conversely it is easy to see that $D$ can be recovered
from $\mathsf q$.  Recall that if
$e_1=\sigma_Q(u_1)(b_m)+\tau_1^\dagger(b_m)$ and
$e_2=\sigma_Q(u_2)(b_m)+\tau_2^\dagger(b_m)\in \mathbb E$, then
\begin{equation*}
\begin{split}
  \langle e_1, e_2\rangle&=\langle \sigma_Q(u_1)(b_m)+\tau_1^\dagger(b_m), \sigma_Q(u_2)(b_m)+\tau_2^\dagger(b_m)\rangle\\
  &=\ell_{\Lambda(u_1,u_2)}(b_m)+\langle u_1(m),
  \tau_2(m)\rangle+\langle u_2(m), \tau_1(m)\rangle.
\end{split}
\end{equation*}
In particular,
$\langle e_1, e_2\rangle=\langle \pi_Q(e_1), \mathsf
q(e_2)\rangle+\langle \pi_Q(e_2), \mathsf q(e_1)\rangle$ for all
$e_1, e_2\in \mathbb E\an{U}$ if and only if $\Lambda\an{U\otimes U}$
vanishes and $K=U^\circ$, i.e.~if and only if $D$ is maximal isotropic
(Proposition \ref{prop2}).

\medskip
Now we recall Li-Bland's definition of a pseudo-Dirac structure \cite{Li-Bland14}.
\begin{definition}\label{davids_pseudo}
  Let $\mathsf E\to M$ be a Courant algebroid. A pseudo-Dirac
  structure is a pair $(U,\nabla^p)$ consisting of a subbundle
  $U\subseteq \mathsf E$ together with a map $\nabla^p\colon
  \Gamma(U)\to\Omega^1(M,U^*)$ satisfying
\begin{enumerate}
\item $\nabla^p(fu)=f\nabla^p u+\dr f\otimes\langle u,\cdot\rangle$,
\item $\dr\langle u_1,u_2\rangle=\langle\nabla^p u_1, u_2\rangle+\langle u_1, \nabla^p u_2\rangle$,
\item $\lb u_1,u_2\rb_p:=\lb u_1, u_2\rb_{\mathsf E}-\rho^*\langle
  \nabla^p u_1, u_2\rangle$ defines a bracket
  $\Gamma(U)\times\Gamma(U)\to\Gamma(U)$,
\item and 
\begin{multline}\label{alternative_curvature}
(\langle \lb u_1, u_2\rb_p,\nabla^p u_3\rangle+\ip{\rho(u_1)}\dr\langle\nabla^p u_2, u_3\rangle)+{\rm c.p.}\\
+\dr\left(\left\langle \nabla^p_{\rho(u_1)}u_2-\nabla^p_{\rho(u_2)}u_1, u_3
\right\rangle-\left\langle\lb u_1,u_2\rb_p, u_3
\right\rangle\right)=0
\end{multline}
\end{enumerate}
for all $u_1,u_2,u_3\in\Gamma(U)$ and $f\in C^\infty(M)$.
\end{definition}
Consider the tangent double $(T\mathsf E, TM, \mathsf E, M)$ where
$\mathsf E$ is a Courant algebroid over $M$. Choose a linear (wide)
Dirac structure $D$ in $T\mathsf E$, over the side
$U\subseteq \mathsf E$ and a metric connection
$\nabla\colon\mx(M)\times\Gamma(\mathsf E)\to \Gamma(\mathsf E)$ that
is adapted to $D$. Li-Bland defines the pseudo-Dirac structure
associated to $D$ \cite{Li-Bland14} as the map
$\nabla^p\colon\Gamma(U)\to\Omega^1(M,U^*)$ that is defined by
$\nabla^pu=\mathsf q\circ Tu$ for all $u\in\Gamma(U)$. By definition
of $\sigma^\nabla_{\mathsf E}$, we have
$Tu=\sigma^\nabla_{\mathsf E}(u)+\widetilde{\nabla_\cdot u} $ and we
find that
$\nabla^pu(v_m)=\overline{\nabla_{v_m}u}=[\nabla]_{v_m}u$. The
pseudo-Dirac structure is nothing else than the invariant part of the
metric connection that is adapted to $D$ (Remark
\ref{invariant_part}).  Condition (2) in Definition
\ref{davids_pseudo} is then
\begin{equation}\label{pseudo}
\dr\langle u_1,u_2\rangle
=\langle Tu_1, Tu_2\rangle_{T\mathsf E}
=\langle u_1, \nabla^pu_2\rangle+\langle u_2 , \nabla^pu_1\rangle
\end{equation}
for all $u_1,u_2\in\Gamma(U)$ and Condition (1) is
\begin{equation}\label{pseudo2}
  \nabla^p(\varphi\cdot u)=\overline{\nabla_\cdot (\varphi\cdot u)}
  =\varphi\cdot \overline{\nabla_\cdot u}+\dr\varphi\otimes \overline{u}
=\varphi\cdot\nabla^pu+\dr\varphi\otimes \overline{u}.
\end{equation}
The bracket $\lb\cdot\,,\cdot\rb_p$ is then 
\[\lb u_1, u_2\rb_p=\lb u_1, u_2\rb_{\mathsf
  E}-\rho^*\langle\nabla^pu_1,u_2\rangle=\lb u_1, u_2\rb_{\mathsf
  E}-\rho^*\langle\nabla_\cdot u_1,u_2\rangle=\lb u_1, u_2\rb_\nabla,
\]
the bracket defined in \eqref{dull_corretion}.  Finally, a straightforward
computation shows that the left-hand side of
\eqref{alternative_curvature} equals $R_\Delta^{\rm
  bas}(u_1,u_2)^*u_3\in\Gamma(B^*)$, which is zero by Proposition
\ref{char_VB_dir}.  Li-Bland proves that the bracket
$\lb\cdot\,,\cdot\rb_p$ defines a Lie algebroid structure on $U$. More
explicitly, he finds that the left-hand side $\Psi(u_1,u_2,u_3 )$ of
\eqref{alternative_curvature} defines a tensor
$\Psi\in\Omega^3(U,T^*M)$ that is related as follows to the Jacobiator of
$\lb\cdot\,,\cdot\rb_p$:
$\operatorname{Jac}_{\lb\cdot\,,\cdot\rb_p}=(\Beta\inv\circ\rho_{\mathsf
  E}^*)\Psi$. He proves so that (wide) linear Dirac structures in
$T\mathsf E$ are in bijection with pseudo-Dirac structures on $\mathsf
E$. Hence, our result in Proposition \ref{char_VB_dir} is a
generalisation of Li-Bland's result to linear Dirac structures in
general VB-Courant algebroids.

Further, our Theorem \ref{char_sub_alg} can be formulated as follows
in Li-Bland's setting. 
\begin{theorem}
  In the correspondence of linear Dirac structures with pseudo-Dirac
  connections in \cite{Li-Bland14}, LA-Dirac structures correspond
  to pseudo-Dirac connections $(U,\nabla^p)$ such that
\begin{enumerate}
\item $U\subseteq \mathsf E$ is an isotropic (or `quadratic') subbundle,
  i.e.~$U^\perp\subseteq U$,
\item $\nabla^p$ sends $U^\perp$ to zero and so, by Condition (2) in
  Definition \ref{davids_pseudo}, has image in $U/U^\perp\subseteq
  \mathsf E/U^\perp\simeq U^*$,
\item the induced ordinary connection $\overline{\nabla^p}\colon
  \Gamma(U/U^\perp)\to\Omega^1(M,U/U^\perp)$ is flat.
\end{enumerate} 
\end{theorem}
We propose to call these pseudo-Dirac connections \textbf{quadratic
  pseudo-Dirac connections}.  Note that $\overline{\nabla^p}$ equals
$\bar\nabla\colon\mx(M)\times\Gamma(U/U^\perp)\to\Gamma(U/U^\perp)$
$\bar\nabla_X\bar u=\overline{\nabla_Xu}$, $u\in\Gamma(U)$ and
$X\in\mx(M)$, for any metric connection
$\nabla\colon\mx(M)\times\Gamma(U)\to\Gamma(U)$ such that
$[\nabla]=\nabla^p$. Such a connection must preserve $U$ by Condition
(2) in Proposition \ref{char_sub_alg}, and so also $U^\perp$ since it
is metric. The condition $R_\nabla(X_1,X_2)u\in\Gamma(U^\perp)$ for
all $X_1,X_2\in\mx(M)$ and $u\in\Gamma(U)$ in Proposition
\ref{char_sub_alg} is then equivalent to $R_{\bar\nabla}=0$.

\subsection{The Manin pair associated to an LA-Dirac structure}
Consider as before an LA-Courant algebroid $\mathbb E$ with sides $B$ and $Q$ and with core $Q^*$, and an LA-Dirac structure $D$
\begin{equation*}
\begin{xy}
  \xymatrix{\mathbb E\ar[r]\ar[d]&Q\ar[d]\\
    B\ar[r]&M }
\end{xy}
\qquad \qquad 
\begin{xy}
\xymatrix{D\ar[r]\ar[d]&U\ar[d]\\
B\ar[r]&M
}
\end{xy}
\end{equation*}
in $\mathbb E$ with core $U^\circ$.  Since $\partial_Q$ restricts to a
map from $U^\circ $ to $U$, we can define the vector bundle
\[\mathbb B=\frac{U\oplus
  Q^*}{\operatorname{graph}(-\partial_Q\an{U^\circ})}\to M.
\]
This vector bundle is anchored by the map
\[\rho_{\mathbb B}\colon\mathbb B\to TM, \qquad \rho_{\mathbb B}(u\oplus
\tau)=\rho_Q(u+\partial_Q\tau)=\rho_Q(u)+\rho_B(\partial_B\tau).
\]
Note that this map is well-defined because 
\[ \rho_{\mathbb
  B}(-\partial_Q\tau\oplus\tau)=\rho_Q(-\partial_Q\tau+\partial_Q\tau)=0
\]
for all $\tau\in U^\circ$.  We will show that there is a symmetric
non-degenerate pairing $\langle \cdot\,,\cdot\rangle_{\mathbb B}$ on
$\mathbb B\times_M\mathbb B$ and a bracket
$\lb\cdot\,,\cdot\rb_{\mathbb B}$ on $\Gamma(\mathbb B)$ such that
\[ (\mathbb B\to M, \rho_{\mathbb B}, \langle
\cdot\,,\cdot\rangle_{\mathbb B}, \lb\cdot\,,\cdot\rb_{\mathbb B})
\]
is a Courant-algebroid.
We define the pairing on $\mathbb B$ by
\begin{align*}
\langle u_1\oplus\tau_1, u_2\oplus\tau_2\rangle_{\mathbb B}
=\langle u_1,\tau_1\rangle +\langle u_2,\tau_2\rangle +\langle
\tau_1, \partial_Q\tau_2\rangle.
\end{align*}
It is easy to check that this pairing is well-defined and
non-degenerate and that the
induced 
map $\mathcal D_{\mathbb B}: C^\infty(M)\to\Gamma(\mathbb B)$
given by 
\[\langle \mathcal D_{\mathbb B}f, u\oplus\tau\rangle_{\mathbb
  F}=\rho_{\mathbb B}(u\oplus\tau)(f)\]
can alternatively be defined by 
$ \mathcal D_{\mathbb B}f=0\oplus \rho_Q^*\dr f$.

Choose as before a Lagrangian splitting of $\mathbb E$ that is adapted
to $D$, and recall that the linear Courant algebroid structure and the
linear Lie algebroid structure on $\mathbb E$ are then encoded by a
split Lie 2-algebroid and by a self-dual $2$-representation,
respectively, both denoted as usual.  We define the bracket on
$\Gamma(\mathbb B)$ by
\begin{equation}\label{C_bracket}
\begin{split}
  &\lb u_1\oplus\tau_1, u_2\oplus\tau_2\rb_{\mathbb B}\\
  =&(\lb u_1,
  u_2\rb_U+\nabla_{\partial_B\tau_1}u_2-\nabla_{\partial_B\tau_2}u_1 )
  \oplus (\lb \tau_1,
  \tau_2\rb_{Q^*}+\Delta_{u_1}\tau_2-\Delta_{u_2}\tau_1+\rho_Q^*\dr\langle
  \tau_1, u_2\rangle ).
\end{split}
\end{equation}
A quick computation as the one at the end of the proof of Theorem \ref{core_courant}
show that this bracket does not depend on the
choice of Lagrangian splitting.

\begin{theorem}\label{thm_CA_sd}
  Let $(D,U,B,M)$ be an LA-Dirac structure in a LA-Courant algebroid
  $(\mathbb E,Q,B,M)$.  Then the vector bundle \[ \mathbb
  B=\frac{U\oplus
    Q^*}{\operatorname{graph}(-\partial_Q\an{U^\circ})}\to M,\] with
  the anchor $\rho_{\mathbb B}$, the pairing $\langle
  \cdot\,,\cdot\rangle_{\mathbb B}$ and the bracket
  $\lb\cdot\,,\cdot\rb_{\mathbb B}$, is a Courant algebroid.
Further, $U$ is a Dirac structure in $\mathbb B$, via the inclusion $U\hookrightarrow\mathbb B$, 
$u\mapsto u\oplus 0$.
\end{theorem}
The proof of Theorem \ref{thm_CA_sd} can be found in Appendix
\ref{proof_of_manin_ap}.

\begin{corollary}\label{disapointing_cor}
  Let $(D,U,B,M)$ be an LA-Dirac structure in an LA-Courant algebroid
  $(\mathbb E, B, Q, M)$ (with core $Q^*$). The Manin pair $(\mathbb B, U)$ 
defined in Theorem \ref{thm_CA_sd} and the degenerate Courant algebroid $Q^*$ satisfy the following conditions:
\begin{enumerate}
%\item $\iota(U)^\circ$ is isotropic in $Q^*$, FOLLOWS FROM 3
\item There is a morphism $\psi\colon
  Q^*\to \mathbb B$ of degenerate Courant algebroids and an embedding $\iota\colon U\to Q$ over the identity on $M$
\item $\iota$ is compatible with the anchors: $\rho_Q\circ\iota=\rho_{\mathbb B}\an{U}$,
\item $\psi(Q^*)+U=\mathbb B$ and
\item $\langle \psi(\tau), u\rangle_{\mathbb B}=\langle  \iota(u), \tau\rangle$ for all
  $\tau\in Q^*$ and $u\in U$.
\end{enumerate}
\end{corollary}

\begin{proof}
 Take an LA-Dirac structure $(D,U,B,M)$ in an LA-Courant
  algebroid\linebreak $(\mathbb E,Q,B,M)$.  The morphism $\psi\colon Q^*\to \mathbb B$ defined by
  $\psi(\tau)\mapsto 0\oplus \tau$ is obviously a morphism of
  degenerate Courant algebroids. Conditions (1)--(4) are then immediate.
\end{proof}

Conversely take a Manin pair $(\mathbb B,U)$ over $M$ satisfying with $Q^*$ the
conditions in Corollary \ref{disapointing_cor} and identify $U$ with a
subbundle of $Q$.  If $\tau\in U^\circ\subseteq Q^*$, then
$\psi(\tau)$ satisfies
\[\langle u, \psi(\tau)\rangle_{\mathbb B}=\langle u, \tau\rangle=0
\]
for all $u\in U$. Since $U$ is a Dirac structure, we find that $\psi$
restricts to a map $U^\circ\to U$. Conversely, we find easily that
$\psi(\tau)\in U$ if and only if $\tau\in U^\circ$.  Next choose
$\tau_1\in U^\circ $ and $\tau_2\in Q^*$.  Then since $\psi(\tau_1)\in U$,
\[\langle
\psi(\tau_1), \tau_2\rangle=\langle
\psi(\tau_1),\psi(\tau_2)\rangle_{\mathbb B}= \langle
\tau_1,\tau_2\rangle_{Q^*}=\langle \partial_Q\tau_1, \tau_2\rangle,
\] 
which shows that $\psi\an{U^\circ}=\partial_Q\an{U^\circ}$. In
particular, $\partial_Q$ sends $U^\circ$ to $U$, and $U^\circ$ is
isotropic in $Q^*$.  Consider the vector bundle map $U\oplus Q^*\to
\mathbb B$, $(u,\tau)\mapsto u+\psi(\tau)$. By assumption, this map is
surjective. Its kernel is the set of pairs $(u,\tau)$ with
$u=-\psi(\tau)$, i.e.~the graph of $-\partial_Q\an{U^\circ}\colon
U^\circ\to U$. It follows that
\begin{equation}\label{id_C}
\mathbb B\simeq \frac{U\oplus Q^*}{\operatorname{graph}(-\partial_Q\an{U^\circ}\colon
U^\circ\to U)}.
\end{equation}
Hence, we can use the notation $u\oplus \tau$ for
$\overline{u+\psi(\tau)}\in \mathbb B$.

In the case of an LA-Courant algebroid $(TA\oplus_AT^*A,TM\oplus
A^*,A,M)$ as in \S\ref{PontLA_LACourant}, for a Lie algebroid $A$, we
could show in \cite{Jotz19} that Manin pairs as in Corollary
\ref{disapointing_cor} are \emph{in bijection} with LA-Dirac
structures on $A$. That is, given a Manin pair $(\mathbb B,U)$ with an
inclusion $U\hookrightarrow TM\oplus A^*$ and a degenerate Courant
algebroid morphism $A\oplus T^*M\to \mathbb B$ satisfying (1)--(4), then via
\eqref{id_C}, there exists a Lagrangian splitting
of $TA\oplus_AT^*A$ such that the Courant bracket on $\mathbb B$ is given
by \eqref{C_bracket}.

 \appendix
\section{Proof of Theorem \ref{thm_CA_sd}}\label{proof_of_manin_ap}
Note that in the following computations, we will make use of the
identity $\partial_Q=\partial_Q^*$ without always mentioning it.  We
begin by proving the following two lemmas.
\begin{lemma}\label{lemma_looks_like_basic}
  Consider an LA-Courant algebroid $(\mathbb E, Q, B,M)$.  The bracket
  $\lb\cdot\,,\cdot\rb_{Q^*}$ on sections of the core $Q^*$ satisfies
  the following equation:
\begin{equation}\label{bracket_on_Q*_basic}
\begin{split}
  R(\partial_B\tau_1,\partial_B\tau_2)q
  =&-\Delta_q\lb\tau_1,\tau_2\rb_{Q^*}+\lb \Delta_q\tau_1,\tau_2\rb_{Q^*}+\lb\tau_1,\Delta_q\tau_2\rb_{Q^*}\\
  &+\Delta_{\nabla_{\partial_B\tau_2}q}\tau_1-\Delta_{\nabla_{\partial_B\tau_1}q}\tau_2
-\rho_Q^*\dr\langle\tau_1,\nabla_{\partial_B\tau_2}q\rangle
\end{split}
\end{equation}
for all $q\in\Gamma(Q)$ and $\tau_1,\tau_2\in\Gamma(Q^*)$.
\end{lemma}

\begin{proof}
The proof is just a computation using (M1) and \eqref{(LC10)}.
We have 
\begin{equation*}
\begin{split}
  &\Delta_q\lb\tau_1,\tau_2\rb_{Q^*}-\lb
  \Delta_q\tau_1,\tau_2\rb_{Q^*}-\lb\tau_1,\Delta_q\tau_2\rb_{Q^*}
+\Delta_{\nabla_{\partial_B\tau_1}q}\tau_2\\
&-\Delta_{\nabla_{\partial_B\tau_2}q}\tau_1
+\rho_Q^*\dr\langle\tau_1,\nabla_{\partial_B\tau_2}q\rangle\\
  =&\Delta_q\Delta_{\partial_Q\tau_1}\tau_2-\Delta_q\nabla^*_{\partial_B\tau_2}\tau_1
-\Delta_{\partial_Q(\Delta_q\tau_1)}\tau_2+\nabla^*_{\partial_B\tau_2}\Delta_q\tau_1-\Delta_{\partial_Q\tau_1}\Delta_q\tau_2\\
  & +\nabla^*_{\partial_B(\Delta_q\tau_2)}\tau_1
  +\Delta_{\nabla_{\partial_B\tau_1}q}\tau_2-\Delta_{\nabla_{\partial_B\tau_2}q}\tau_1
+\rho_Q^*\dr\langle\tau_1,\nabla_{\partial_B\tau_2}q\rangle
\end{split}
\end{equation*}
Replacing $\Delta_q\Delta_{\partial_Q\tau_1}\tau_2-\Delta_{\partial_Q\tau_1}\Delta_q\tau_2$ 
by $R_\Delta(q,\partial_Q\tau_1)\tau_2+\Delta_{\lb q,\partial_Q\tau_1\rb}\tau_2$
and reordering the terms yields
\begin{equation*}
\begin{split}
  &R_\Delta(q,\partial_Q\tau_1)\tau_2+\Delta_{\lb
    q,\partial_Q\tau_1\rb-\partial_Q(\Delta_q\tau_1)+\nabla_{\partial_B\tau_1}q}\tau_2
  -\Delta_q\nabla^*_{\partial_B\tau_2}\tau_1+\nabla^*_{\partial_B\tau_2}\Delta_q\tau_1\\
  & +\nabla^*_{\partial_B(\Delta_q\tau_2)}\tau_1-\Delta_{\nabla_{\partial_B\tau_2}q}\tau_1
+\rho_Q^*\dr\langle\tau_1,\nabla_{\partial_B\tau_2}q\rangle.
\end{split}
\end{equation*}
Since
$R_\Delta(q,\partial_Q\tau_1)\tau_2=\langle \ip{\partial_Q\tau_1}\ip{q}\omega,\partial_B\tau_2\rangle$
by \eqref{omega_dorfman_curv}, we can now use \eqref{(LC10)} and
$\nabla_q\circ\partial_B=\partial_B\circ\Delta_q$ to replace \[
R_\Delta(q,\partial_Q\tau_1)\tau_2-\Delta_q\nabla^*_{\partial_B\tau_2}\tau_1
+\nabla^*_{\partial_B\tau_2}\Delta_q\tau_1-\Delta_{\nabla_{\partial_B\tau_2}q}\tau_1+\nabla^*_{\nabla_q\partial_B\tau_2}\tau_1\]
by
$-\langle\nabla_{\nabla_\cdot\partial_B\tau_2 }q, \tau_1\rangle
+R(\partial_B\tau_2,\partial_B\tau_1)q$.
We use (M1) to replace $\Delta_{\lb
  q, \partial_Q\tau_1\rb-\partial_Q(\Delta_q\tau_1)+\nabla_{\partial_B\tau_1}q}\tau_2$
by $-\Delta_{\partial_B^*\langle \tau_1, \nabla_\cdot
  q\rangle}\tau_2$.  These two steps yield that the right hand side
of our equation is
\begin{equation*}
\begin{split}
  &-\langle\nabla_{\nabla_\cdot\partial_B\tau_2 }q, \tau_1\rangle
+R(\partial_B\tau_2,\partial_B\tau_1)q-\Delta_{\partial_B^*\langle \tau_1, \nabla_\cdot
    q\rangle}\tau_2
  +\rho_Q^*\dr\langle\tau_1,\nabla_{\partial_B\tau_2}q\rangle.
\end{split}
\end{equation*}
To conclude, let us show  that 
\[-\langle\nabla_{\nabla_\cdot\partial_B\tau_2 }q, \tau_1\rangle
-\Delta_{\partial_B^*\langle
  \tau_1, \nabla_\cdot q\rangle}\tau_2
+\rho_Q^*\dr\langle\tau_1,\nabla_{\partial_B\tau_2}q\rangle\in\Gamma(Q^*)
\]
vanishes. On $q'\in\Gamma(Q)$, this is 
\begin{equation*}
\begin{split}
  &-\langle \nabla_{\nabla_{q'}(\partial_B\tau_2)}q,
  \tau_1\rangle+\langle \lb \partial_B^*\langle\nabla_\cdot q,
  \tau_1\rangle, q'\rb,\tau_2\rangle
  +\rho_Q(q')\langle\tau_1, \nabla_{\partial_B\tau_2}q\rangle\\
  =&-\langle \nabla_{\nabla_{q'}(\partial_B\tau_2)}q,
  \tau_1\rangle
  +\langle\Delta_{q'}\tau_2, \partial_B^*(\langle\nabla_\cdot q, \tau_1\rangle)\rangle\\
  =&-\langle \nabla_{\nabla_{q'}(\partial_B\tau_2)}q,
  \tau_1\rangle +\langle\nabla_{\partial_B(\Delta_{q'}\tau_2)} q,
  \tau_1\rangle=0.
\end{split}
\end{equation*}
We have used \eqref{rho_delta}  and
\eqref{dualdd} in the first
line, as well as for the first equality. To conclude, we have used
$\partial_B\circ\Delta_{q'}=\nabla_{q'}\circ\partial_B$ by \eqref{D1}.
\end{proof}

\begin{lemma}\label{bracket_and_pullback}
  The bracket on $Q^*$ satisfies
\begin{equation}\label{bracket_pullback}
\lb \rho_Q^*\dr f,\tau\rb_{Q^*}=0
\end{equation}
for all $f\in C^\infty(M)$ and $\tau\in\Gamma(Q^*)$.
\end{lemma}

\begin{proof}
By \eqref{almost_C}, we have $\lb \rho_Q^*\dr
f,\tau\rb_{Q^*}=\nabla_{\partial_B\rho_Q^*\dr
  f}\tau-\Delta_{\partial_Q\tau}(\rho_Q^*\dr
f)+\rho_Q^*\dr((\rho_Q\partial_Q\tau)f)$. But
$\partial_B\rho_Q^*=0$ by \eqref{rho_delta} and
$\Delta_{\partial_Q\tau}(\rho_Q^*\dr
f)=\rho_Q^*\dr(\rho_Q(\partial_Q\tau)(f))$ by \eqref{dorfman_on_exact}.
\end{proof}

\medskip

Now we check that the bracket $\lb\cdot\,,\cdot\rb_{\mathbb B}$ in
Theorem \ref{thm_CA_sd} is well-defined. We have for all
$\upsilon\in\Gamma(U^\circ)$, $\tau\in\Gamma(Q^*)$ and $u\in\Gamma(U)$:
\begin{equation*}
\begin{split}
  \lb u\oplus \tau, (-\partial_Q\upsilon)\oplus\upsilon\rb&=
  \left(-\lb u,\partial_Q\upsilon\rb_U+\nabla_{\partial_B\tau}(-\partial_Q\upsilon)-\nabla_{\partial_B\upsilon}u\right)\\
  &\quad \oplus
  \left(\lb\tau,\upsilon\rb_{Q^*}+\Delta_u\upsilon-\Delta_{-\partial_Q\upsilon}\tau+\rho_Q^*\dr\langle\tau,-\partial_Q\upsilon\rangle\right).
\end{split}
\end{equation*}
By (M1), the properties of $2$-representations and \eqref{almost_C}, this is 
\begin{equation*}
\begin{split}
  &\left(-\partial_Q(\Delta_u\upsilon)+\partial_B^*\langle\upsilon,\nabla_\cdot
    u\rangle
    +\cancel{\nabla_{\partial_B\upsilon}u}-\partial_Q\nabla_{\partial_B\tau}^*\upsilon-\cancel{\nabla_{\partial_B\upsilon}u}\right)\\
  &\quad \oplus
  \left(-\cancel{\Delta_{\partial_Q\upsilon}\tau}+\nabla_{\partial_B\tau}\upsilon
    +\cancel{\rho_Q^*\dr\langle\tau,\partial_Q\upsilon\rangle}
    +\Delta_u\upsilon+\cancel{\Delta_{\partial_Q\upsilon}\tau}-\cancel{\rho_Q^*\dr\langle\tau,\partial_Q\upsilon\rangle}\right)\\
  =&\left(-\partial_Q(\Delta_u\upsilon)+\partial_B^*\langle\upsilon,\nabla_\cdot
    u\rangle-\partial_Q\nabla_{\partial_B\tau}^*\upsilon\right) \oplus
  \left(\nabla^*_{\partial_B\tau}\upsilon+\Delta_u\upsilon\right).
\end{split}
\end{equation*}
Since $\upsilon\in\Gamma(U^\circ)$ and $\nabla_b$ preserves $\Gamma(U)$
for all $b\in\Gamma(B)$, the section $\langle\upsilon,\nabla_\cdot
u\rangle$ of $B^*$ vanishes and we get
\begin{equation*}
\begin{split}
\lb u\oplus \tau, (-\partial_Q\upsilon)\oplus\upsilon\rb
=&\left(-\partial_Q(\Delta_u\upsilon+\nabla^*_{\partial_B\tau}\upsilon)\right) \oplus \left(\nabla^*_{\partial_B\tau}\upsilon+\Delta_u\upsilon\right).
\end{split}
\end{equation*}
Because $\Delta_u$ preserves as well $\Gamma(U^\circ)$, the sum
$\nabla_{\partial_B\tau}^*\upsilon+\Delta_u\upsilon$ is a section of
$U^\circ$, and so $\lb u\oplus \tau, (-\partial_Q\upsilon)\oplus\upsilon\rb$
is zero in $\mathbb B$.  \medskip 

We now check the Courant algebroid
axioms (CA1), (CA2) and (CA3).  The last one, (CA3), is immediate:
\begin{equation*}
\begin{split}
  &\lb u_1\oplus\tau_1, u_2\oplus\tau_2\rb_{\mathbb{B}}+\lb
  u_2\oplus\tau_2, u_1\oplus\tau_1\rb_{\mathbb{B}}\\
  &=0\oplus
  \left(\rho_Q^*\dr\langle\tau_1,\partial_Q\tau_2\rangle+\rho_Q^*\dr\langle\tau_1,u_2\rangle
    +\rho_Q^*\dr\langle\tau_2,u_1\rangle\right)
  =0\oplus\rho_Q^*\dr\langle u_1\oplus\tau_1, u_2\oplus\tau_2\rangle_{\mathbb B}\\
  &=\mathcal D_{\mathbb B}\langle u_1\oplus\tau_1,
  u_2\oplus\tau_2\rangle_{\mathbb B}.
\end{split}
\end{equation*}
Next we prove (CA2). We have, using \eqref{almost_C} to replace $\lb \tau_1,
\tau_2\rb_{Q^*}$ by
$-\Delta_{\partial_Q\tau_2}\tau_1+\nabla_{\partial_B\tau_1}\tau_2+\rho_Q^*\dr\langle\tau_1,\partial_Q\tau_2\rangle$:
\begin{equation*}\label{metric}
\begin{split}
  &\langle \lb u_1\oplus\tau_1, u_2\oplus\tau_2\rb_{\mathbb B}, u_3\oplus\tau_3\rangle_{\mathbb B}\\
  &=\langle \lb u_1,
  u_2\rb_U+\nabla_{\partial_B\tau_1}u_2-\nabla_{\partial_B\tau_2}u_1, \tau_3\rangle\\
  &+\langle
  -\Delta_{\partial_Q\tau_2}\tau_1+\nabla^*_{\partial_B\tau_1}\tau_2+\rho_Q^*\dr\langle\tau_1,\partial_Q\tau_2\rangle+\Delta_{u_1}\tau_2-\Delta_{u_2}\tau_1+\rho_Q^*\dr\langle
  \tau_1, u_2\rangle, u_3+\partial_Q\tau_3\rangle.
\end{split}
\end{equation*}
We sum $\langle \lb u_1\oplus\tau_1, u_2\oplus\tau_2\rb_{\mathbb B},
u_3\oplus\tau_3\rangle_{\mathbb B}$ with $\langle u_2\oplus\tau_2,
\lb u_1\oplus\tau_1, u_3\oplus\tau_3\rb_{\mathbb B} \rangle_{\mathbb B}$, and 
replace only in the first summand the term
$\langle\Delta_{u_1}\tau_2,\partial_Q\tau_3\rangle$ by
$\rho_Q(u_1)\langle\tau_2,\partial_Q\tau_3\rangle-\langle\tau_2,\lb
u_1,\partial_Q\tau_3\rb\rangle$.  This yields
\begin{equation*}
\begin{split}
  &\langle \lb u_1\oplus\tau_1, u_2\oplus\tau_2\rb_{\mathbb B}, u_3\oplus\tau_3\rangle_{\mathbb B}+\langle u_2\oplus\tau_2, \lb u_1\oplus\tau_1, u_3\oplus\tau_3\rb_{\mathbb B} \rangle_{\mathbb B}\\
  &=\rho_Q(u_1)\langle u_2,\tau_3\rangle-\cancel{\langle u_2, \Delta_{u_1}\tau_3\rangle}+\langle\nabla_{\partial_B\tau_1}u_2-\nabla_{\partial_B\tau_2}u_1,\tau_3\rangle\\
  &-\langle\Delta_{\partial_Q\tau_2}\tau_1,u_3\rangle+\langle\nabla^*_{\partial_B\tau_1}\tau_2,
  u_3\rangle+\rho_Q(u_3)\langle\tau_1,\partial_Q\tau_2\rangle
  +\cancel{\langle\Delta_{u_1}\tau_2, u_3\rangle}-\langle\Delta_{u_2}\tau_1, u_3\rangle\\
  &+\rho_Q(u_3)\langle\tau_1,u_2\rangle-\langle\Delta_{\partial_Q\tau_2}\tau_1,\partial_Q\tau_3\rangle+\langle\nabla^*_{\partial_B\tau_1}\tau_2,\partial_Q\tau_3\rangle+\rho_Q(\partial_Q\tau_3)\langle\tau_1, \partial_Q\tau_2\rangle\\
  & +\rho_Q(u_1)\langle\tau_2,\partial_Q\tau_3\rangle-\langle\tau_2,\lb u_1,\partial_Q\tau_3\rb\rangle-\langle\Delta_{u_2}\tau_1,\partial_Q\tau_3\rangle+\rho_Q(\partial_Q\tau_3)\langle\tau_1,u_2\rangle\\
  &+\rho_Q(u_1)\langle u_3,\tau_2\rangle-\cancel{\langle u_3, \Delta_{u_1}\tau_2\rangle}+\langle\nabla_{\partial_B\tau_1}u_3-\nabla_{\partial_B\tau_3}u_1,\tau_2\rangle\\
  &-\langle\Delta_{\partial_Q\tau_3}\tau_1,u_2\rangle+\langle\nabla^*_{\partial_B\tau_1}\tau_3,
  u_2\rangle+\rho_Q(u_2)\langle\tau_1,\partial_Q\tau_3\rangle
  +\cancel{\langle\Delta_{u_1}\tau_3, u_2\rangle}-\langle\Delta_{u_3}\tau_1, u_2\rangle\\
  &+\rho_Q(u_2)\langle\tau_1,u_3\rangle-\langle\Delta_{\partial_Q\tau_3}\tau_1,\partial_Q\tau_2\rangle+\langle\nabla^*_{\partial_B\tau_1}\tau_3,\partial_Q\tau_2\rangle+\rho_Q(\partial_Q\tau_2)\langle\tau_1, \partial_Q\tau_3\rangle\\
  &+\langle\Delta_{u_1}\tau_3,\partial_Q\tau_2\rangle-\langle\Delta_{u_3}\tau_1,\partial_Q\tau_2\rangle+\rho_Q(\partial_Q\tau_2)\langle\tau_1,u_3\rangle.
\end{split}
\end{equation*}
We reorder the remaining
terms and replace eight times sums like
$\rho_Q(\partial_Q\tau_2)\langle\tau_1,u_3\rangle-\langle\Delta_{\partial_Q\tau_2}\tau_1,
u_3\rangle$ by $\langle\lb \partial_Q\tau_2, u_3\rb,
\tau_1\rangle$ \eqref{dualdd}, and three times sums like
$\langle\nabla_{\partial_B\tau_1}u_2,\tau_3\rangle+\langle
u_2,\nabla^*_{\partial_B\tau_1}\tau_3\rangle$ by
$\rho_B(\partial_B\tau_1)\langle
u_2,\tau_3\rangle$.  This leads to
\begin{equation*}
\begin{split}
  &\langle \lb u_1\oplus\tau_1, u_2\oplus\tau_2\rb, u_3\oplus\tau_3\rangle_{\mathbb B}+\langle u_2\oplus\tau_2, \lb u_1\oplus\tau_1, u_3\oplus\tau_3\rb \rangle_{\mathbb B}\\
  &=\rho_Q(u_1)\langle u_2\oplus\tau_2,u_3\oplus\tau_3\rangle_{\mathbb B}\\
&+\cancel{\langle \lb \partial_Q\tau_2, \partial_Q\tau_3\rb, \tau_1\rangle}
+\cancel{\langle \lb \partial_Q\tau_2, u_3\rb, \tau_1\rangle}
+\cancel{\langle \lb u_3, \partial_Q\tau_2\rb, \tau_1\rangle}
+\cancel{\langle \lb u_2, u_3\rb, \tau_1\rangle}\\
&+\cancel{\langle\lb u_3,u_2\rb, \tau_1\rangle}
+\cancel{\langle\lb \partial_Q\tau_3, u_2\rb, \tau_1\rangle}
+\cancel{\langle\lb u_2, \partial_Q\tau_3\rb,\tau_1\rangle}
+\cancel{\langle\lb\partial_Q\tau_3, \partial_Q\tau_2\rb,\tau_1\rangle}\\
&+\rho_B(\partial_B\tau_1)\langle u_2,\tau_3\rangle
+\rho_B(\partial_B\tau_1)\langle\tau_3,\partial_Q\tau_2\rangle
+\rho_B(\partial_B\tau_1)\langle u_3,\tau_2\rangle\\
&-\langle\nabla_{\partial_B\tau_2}u_1,\tau_3\rangle
  -\langle\tau_2,\lb u_1,\partial_Q\tau_3\rb\rangle-\langle\nabla_{\partial_B\tau_3}u_1,\tau_2\rangle
  +\langle\Delta_{u_1}\tau_3,\partial_Q\tau_2\rangle.
\end{split}
\end{equation*}
The four last terms cancel each other by (M1) and $\partial_Q=\partial_Q^*$.  This
yields
\begin{equation*}
\begin{split}
  &\langle \lb u_1\oplus\tau_1, u_2\oplus\tau_2\rb, u_3\oplus\tau_3\rangle_{\mathbb B}+\langle u_2\oplus\tau_2, \lb u_1\oplus\tau_1, u_3\oplus\tau_3\rb \rangle_{\mathbb B}\\
  &=(\rho_Q(u_1)+\rho_B\partial_B\tau_1)\langle
  u_2\oplus\tau_2,u_3\oplus\tau_3\rangle_{\mathbb B}=\rho_{\mathbb
    B}(u_1\oplus\tau_1)\langle
  u_2\oplus\tau_2,u_3\oplus\tau_3\rangle_{\mathbb B}.
\end{split}
\end{equation*}

\medskip Finally we check the Jacobi identity in Leibniz form
(CA1). We will check that
\[\operatorname{Jac}_{\lb\cdot\,,\cdot\rb}(u_1\oplus\tau_1,
u_2\oplus\tau_2, u_3\oplus\tau_3)=(-\partial_Q\upsilon)\oplus \upsilon
\]
with $\upsilon=(R(\partial_B\tau_1,\partial_B\tau_2)u_3-\langle\ip{u_2}\ip{u_1}\omega,\partial_B\tau_3\rangle)+\text{cyclic
  permutations}$.
Since by Proposition \ref{char_sub_alg} $\ip{u_2}\ip{u_1}\omega$ has image in
$U^\circ$ for all $u_1,u_2\in\Gamma(U)$ and $R(b_1,b_2)$ restricts to
a morphism $U\to U^\circ$ for all $b_1,b_2\in\Gamma(B)$ by Proposition
\ref{char_VB_dir}, $\upsilon$ is a section of $U^\circ$ and so
$\operatorname{Jac}_{\lb\cdot\,,\cdot\rb}(u_1\oplus\tau_1,
u_2\oplus\tau_2, u_3\oplus\tau_3)$ will be zero in $\mathbb B$.

Using \eqref{partial_B_morphism},
  \eqref{dorfman_on_exact}, \eqref{bracket_pullback}
and \eqref{rho_delta}, we find
\begin{equation*}
\begin{split}
  &\lb\lb u_1\oplus\tau_1, u_2\oplus\tau_2\rb, u_3\oplus\tau_3\rb\\
  =&\Bigl( \lb \lb u_1,
  u_2\rb_U, u_3\rb_U+\lb \nabla_{\partial_B\tau_1}u_2-\nabla_{\partial_B\tau_2}u_1, u_3\rb_U\\
  &\qquad+ \nabla_{[\partial_B
    \tau_1,\partial_B\tau_2]+\partial_B(\Delta_{u_1}\tau_2-\Delta_{u_2}\tau_1)}u_3 -\nabla_{\partial_B\tau_3}(\lb u_1,
  u_2\rb_U+\nabla_{\partial_B\tau_1}u_2-\nabla_{\partial_B\tau_2}u_1)
  \Bigr)\\
  &\oplus\Bigl(\lb\lb \tau_1,
  \tau_2\rb_{Q^*}, \tau_3\rb_{Q^*}+\lb\Delta_{u_1}\tau_2-\Delta_{u_2}\tau_1, \tau_3 \rb_{Q^*}\\
  &\qquad+\Delta_{\lb u_1,
    u_2\rb_U+\nabla_{\partial_B\tau_1}u_2-\nabla_{\partial_B\tau_2}u_1}\tau_3-\Delta_{u_3}(\lb
  \tau_1,
  \tau_2\rb_{Q^*}+\Delta_{u_1}\tau_2-\Delta_{u_2}\tau_1)\\
  &\qquad+\rho_Q^*\dr\left\langle\lb \tau_1,
  \tau_2\rb_{Q^*}+\Delta_{u_1}\tau_2-\Delta_{u_2}\tau_1, u_3 \right\rangle \Bigr).
\end{split}
\end{equation*}
In the same manner, we compute 
\begin{equation*}
\begin{split}
  &\lb u_1\oplus\tau_1,\lb u_2\oplus\tau_2,
  u_3\oplus\tau_3\rb\rb\\
  &=\bigl\lb u_1\oplus\tau_1, (\lb u_2,
  u_3\rb_U+\nabla_{\partial_B\tau_2}u_3-\nabla_{\partial_B\tau_3}u_2
  ) \oplus (\lb \tau_2,
  \tau_3\rb_{Q^*}+\Delta_{u_2}\tau_3-\Delta_{u_3}\tau_2+\rho_Q^*\dr\langle
  \tau_2, u_3\rangle
  )\bigr\rb\\
  &=\Bigl(\lb u_1, \lb u_2, u_3\rb_U\rb_U+\lb u_1,
  \nabla_{\partial_B\tau_2}u_3-\nabla_{\partial_B\tau_3}u_2\rb_U\\
&\qquad+\nabla_{\partial_B\tau_1}(\lb u_2,
  u_3\rb_U+\nabla_{\partial_B\tau_2}u_3-\nabla_{\partial_B\tau_3}u_2)-\nabla_{[ \partial_B\tau_2,
  \partial_B\tau_3]+\partial_B(\Delta_{u_2}\tau_3-\Delta_{u_3}\tau_2)
 }u_1\Bigr)\\
&\oplus\Bigl(\lb \tau_1, \lb \tau_2,
  \tau_3\rb_{Q^*}\rb_{Q^*}+\lb\tau_1, \Delta_{u_2}\tau_3-\Delta_{u_3}\tau_2\rb_{Q^*}+\rho_Q^*\dr(\rho_Q(\partial_Q\tau_1)\langle
  \tau_2, u_3\rangle)\\
&\qquad +\Delta_{u_1}(\lb \tau_2,
  \tau_3\rb_{Q^*}+\Delta_{u_2}\tau_3-\Delta_{u_3}\tau_2)+\rho_Q^*\dr(\rho_Q(u_1)\langle
  \tau_2, u_3\rangle)-\Delta_{ \lb u_2,
  u_3\rb_U+\nabla_{\partial_B\tau_2}u_3-\nabla_{\partial_B\tau_3}u_2
  }\tau_1\\
&\qquad +\rho_Q^*\dr \langle\tau_1, \lb u_2,
  u_3\rb_U+\nabla_{\partial_B\tau_2}u_3-\nabla_{\partial_B\tau_3}u_2
\rangle 
\Bigr)
\end{split}
\end{equation*}
Since $\operatorname{Jac}_{\lb\cdot\,,\cdot\rb_U}(u_1,u_2,u_3)=0$ and
$\partial_B\circ\Delta_q=\nabla_q\circ\partial_B$ for all $q\in\Gamma(Q)$, the
$U$-term of
$\operatorname{Jac}_{\lb\cdot\,,\cdot\rb}(u_1\oplus\tau_1,u_2\oplus\tau_2,u_3\oplus\tau_3)$
equals
\begin{equation*}
\begin{split}
&\lb \nabla_{\partial_B\tau_1}u_2-\nabla_{\partial_B\tau_2}u_1, u_3\rb_U\\
  &+ \nabla_{[\partial_B
    \tau_1,\partial_B\tau_2]+\nabla_{u_1}\partial_B\tau_2-\nabla_{u_2}\partial_B\tau_1}u_3 -\nabla_{\partial_B\tau_3}(\lb u_1,
  u_2\rb_U+\nabla_{\partial_B\tau_1}u_2-\nabla_{\partial_B\tau_2}u_1)\\
&+\lb u_2,
  \nabla_{\partial_B\tau_1}u_3-\nabla_{\partial_B\tau_3}u_1\rb_U\\
&+\nabla_{\partial_B\tau_2}(\lb u_1,
  u_3\rb_U+\nabla_{\partial_B\tau_1}u_3-\nabla_{\partial_B\tau_3}u_1)-\nabla_{[ \partial_B\tau_1,
  \partial_B\tau_3]+\nabla_{u_1}\partial_B\tau_3-\nabla_{u_3}\partial_B\tau_1
 }u_2\\
& -\lb u_1,
  \nabla_{\partial_B\tau_2}u_3-\nabla_{\partial_B\tau_3}u_2\rb_U\\
&-\nabla_{\partial_B\tau_1}(\lb u_2,
  u_3\rb_U+\nabla_{\partial_B\tau_2}u_3-\nabla_{\partial_B\tau_3}u_2)+\nabla_{[ \partial_B\tau_2,
  \partial_B\tau_3]+\nabla_{u_2}\partial_B\tau_3-\nabla_{u_3}\partial_B\tau_2
 }u_1.
\end{split}
\end{equation*}
Note that since for any $b_1,b_2\in\Gamma(B)$,  $R(b_1,b_2)$ restricts to a section
of $\operatorname{Hom}(U,U^\circ)$ (see \S\ref{LA-Dirac}), the
last summand on the right hand side of (M4) vanishes on sections of $U$.
By sorting out the terms and using (M4) on sections of $U$, we get
\begin{equation*}
\begin{split}
  &-R_\nabla(\partial_B\tau_3, \partial_B\tau_1)u_2-R_\nabla(\partial_B\tau_1,\partial_B\tau_2)u_3
  -R_\nabla(\partial_B\tau_2,\partial_B\tau_3)u_1\\
  &+\partial_Q\langle\ip{u_2}\ip{u_1}\omega,\partial_B\tau_3\rangle
   +\partial_Q\langle\ip{u_3}\ip{u_2}\omega,\partial_B\tau_1\rangle
   +\partial_Q\langle\ip{u_1}\ip{u_3}\omega,\partial_B\tau_2\rangle.
\end{split}
\end{equation*}
Since $R_\nabla=\partial_Q\circ R$, this is $-\partial_Q\upsilon$. 
We conclude by computing the $Q^*$-part of
$\operatorname{Jac}_{\lb\cdot\,,\cdot\rb}(u_1\oplus\tau_1,u_2\oplus\tau_2,u_3\oplus\tau_3)$.
Again, because
$\operatorname{Jac}_{\lb\cdot\,,\cdot\rb_{Q^*}}(\tau_1,\tau_2,\tau_3)=0$,
we get
\begin{equation*}
\begin{split}
  &\lb\Delta_{u_1}\tau_2-\Delta_{u_2}\tau_1, \tau_3
  \rb_{Q^*}+\Delta_{\lb u_1,
    u_2\rb_U+\nabla_{\partial_B\tau_1}u_2-\nabla_{\partial_B\tau_2}u_1}\tau_3-\Delta_{u_3}(\lb
  \tau_1,
  \tau_2\rb_{Q^*}+\Delta_{u_1}\tau_2-\Delta_{u_2}\tau_1)\\
  &+\rho_Q^*\dr\left\langle\lb \tau_1,
    \tau_2\rb_{Q^*}+\cancel{\Delta_{u_1}\tau_2}-\cancel{\Delta_{u_2}\tau_1}, u_3
  \right\rangle+\lb\tau_2,
  \Delta_{u_1}\tau_3-\Delta_{u_3}\tau_1\rb_{Q^*}+\rho_Q^*\dr(\rho_Q(\partial_Q\tau_2)\langle
  \tau_1, u_3\rangle)\\
  &+\Delta_{u_2}(\lb \tau_1,
  \tau_3\rb_{Q^*}+\Delta_{u_1}\tau_3-\Delta_{u_3}\tau_1)+\cancel{\rho_Q^*\dr(\rho_Q(u_2)\langle
  \tau_1, u_3\rangle)}-\Delta_{ \lb u_1,
    u_3\rb_U+\nabla_{\partial_B\tau_1}u_3-\nabla_{\partial_B\tau_3}u_1
  }\tau_2\\
  &+\rho_Q^*\dr \langle\tau_2, \cancel{\lb u_1,
  u_3\rb_U}+\nabla_{\partial_B\tau_1}u_3-\nabla_{\partial_B\tau_3}u_1
  \rangle -\lb\tau_1,
  \Delta_{u_2}\tau_3-\Delta_{u_3}\tau_2\rb_{Q^*}-\rho_Q^*\dr(\rho_Q(\partial_Q\tau_1)\langle
  \tau_2, u_3\rangle)\\
  &-\Delta_{u_1}(\lb \tau_2,
  \tau_3\rb_{Q^*}+\Delta_{u_2}\tau_3-\Delta_{u_3}\tau_2)-\cancel{\rho_Q^*\dr(\rho_Q(u_1)\langle
  \tau_2, u_3\rangle)}+\Delta_{ \lb u_2,
    u_3\rb_U+\nabla_{\partial_B\tau_2}u_3-\nabla_{\partial_B\tau_3}u_2
  }\tau_1\\
  &-\rho_Q^*\dr \langle\tau_1, \cancel{\lb u_2,
  u_3\rb_U}+\nabla_{\partial_B\tau_2}u_3-\nabla_{\partial_B\tau_3}u_2
  \rangle
 \end{split}
\end{equation*}
The six cancelling terms cancel by \eqref{dualdd}.  Reordering the terms, we get using Lemma
\ref{lemma_looks_like_basic}:
\begin{equation*}
\begin{split}
&-R_\Delta(u_3,
u_1)\tau_2-R_\Delta(u_2,u_3)\tau_1-R_\Delta(u_1,u_2)\tau_3\\
&+R(\partial_B\tau_2,\partial_B\tau_3)u_1+R(\partial_B\tau_1,\partial_B\tau_2)u_3
-R(\partial_B\tau_1,\partial_B\tau_3)u_2\\
&-\rho_Q^*\dr\langle
\tau_2,\partial_Q\Delta_{u_3}\tau_1\rangle
+\rho_Q^*\dr\left\langle\lb \tau_1,
  \tau_2\rb_{Q^*}, u_3\right\rangle+\rho_Q^*\dr(\rho_Q(\partial_Q\tau_2)\langle
  \tau_1, u_3\rangle)\\
&+\rho_Q^*\dr \langle\tau_2, \nabla_{\partial_B\tau_1}u_3
\rangle -\rho_Q^*\dr(\rho_Q(\partial_Q\tau_1)\langle
  \tau_2, u_3\rangle).
 \end{split}
\end{equation*}
For the second use of Lemma \ref{lemma_looks_like_basic}, we had to
replace $-\lb \tau_2,\Delta_{u_3}\tau_1\rb_{Q^*}$ by $\lb
\Delta_{u_3}\tau_1,\tau_2\rb_{Q^*}-\rho_Q^*\dr\langle
\tau_2, \partial_Q\Delta_{u_3}\tau_1\rangle$. This is why we get the
first term on the third line. Using \eqref{omega_dorfman_curv}, we get $\upsilon+\rho_Q^*\dr f$
with $f\in C^\infty(M)$ defined by  
\begin{equation*}
\begin{split}
  f=&-\langle
  \tau_2,\partial_Q\Delta_{u_3}\tau_1\rangle+\left\langle\Delta_{\partial_Q\tau_1}\tau_2-\nabla^*_{\partial_B\tau_2}\tau_1,
    u_3\right\rangle+\rho_Q(\partial_Q\tau_2)\langle
  \tau_1, u_3\rangle\\
  &+\langle\tau_2, \nabla_{\partial_B\tau_1}u_3 \rangle
  -\rho_Q(\partial_Q\tau_1)\langle
  \tau_2, u_3\rangle\\
  =&\langle \tau_2,
  -\partial_Q\Delta_{u_3}\tau_1-\lb\partial_Q\tau_1,u_3\rb+\nabla_{\partial_B\tau_1}u_3\rangle+\rho_Q(\partial_Q\tau_2)\langle
  \tau_1, u_3\rangle-\langle\nabla^*_{\partial_B\tau_2}\tau_1,
  u_3\rangle.
 \end{split}
\end{equation*}
By (M1), the first pairing equals 
$-\langle\tau_1,\nabla_{\partial_B\tau_2}u_3\rangle$.
Hence, we find $f=0$ using
$\rho_B\circ\partial_B=\rho_Q\circ\partial_Q$. We have proved 
$\operatorname{Jac}_{\lb\cdot\,,\cdot\rb}(u_1\oplus\tau_1,u_2\oplus\tau_2,u_3\oplus\tau_3)=(-\partial_Q\upsilon)\oplus\upsilon$.

% \bibliographystyle{plain}
% \bibliography{biblio}

\begin{thebibliography}{10}

\bibitem{delCarpio-Marek15}
F.~del Carpio-Marek.
\newblock {\em Geometric structures on degree 2 manifolds}.
\newblock PhD thesis, IMPA, available at
  \url{www.impa.br/wp-content/uploads/2017/05/Fernando\_Del\_Carpio.pdf}, Rio
  de Janeiro, 2015.

\bibitem{DrJoOr15}
T.~Drummond, M.~Jotz, and C.~Ortiz.
\newblock {VB}-algebroid morphisms and representations up to homotopy.
\newblock {\em Differential Geometry and its Applications}, 40:332--357, 2015.

\bibitem{GrJoMaMe18}
A.~Gracia-Saz, M.~Jotz~Lean, K.~C.~H. Mackenzie, and R.~A. Mehta.
\newblock Double {L}ie algebroids and representations up to homotopy.
\newblock {\em J. Homotopy Relat. Struct.}, 13(2):287--319, 2018.

\bibitem{GrMe10a}
A.~Gracia-Saz and R.~A. Mehta.
\newblock Lie algebroid structures on double vector bundles and representation
  theory of {L}ie algebroids.
\newblock {\em Adv. Math.}, 223(4):1236--1275, 2010.

\bibitem{Jotz15}
M.~Jotz~Lean.
\newblock N-manifolds of degree 2 and metric double vector bundles.
\newblock {\em arXiv:1504.00880}, 2015.

\bibitem{Jotz18a}
M.~Jotz~Lean.
\newblock Dorfman connections and {C}ourant algebroids.
\newblock {\em J. Math. Pures Appl. (9)}, 116:1--39, 2018.

\bibitem{Jotz18b}
M.~Jotz~Lean.
\newblock The geometrization of $\N$-manifolds of degree 2.
\newblock {\em Journal of Geometry and Physics}, 133:113 -- 140, 2018.

\bibitem{Jotz17b}
M.~Jotz~Lean.
\newblock Lie 2-algebroids and matched pairs of 2-representations -- a
  geometric approach.
\newblock {\em To appear in Pacific Journal of Mathematics; arXiv:1217.07035},
  2018.

\bibitem{Jotz19}
M.~Jotz~Lean.
\newblock {D}irac groupoids and {D}irac bialgebroids.
\newblock {\em arXiv:1403.2934, to appear in Journal of Symplectic Geometry},
  2019.

\bibitem{Li-Bland12}
D.~Li-Bland.
\newblock Phd thesis: {LA}-{C}ourant {A}lgebroids and their {A}pplications.
\newblock {\em arXiv:1204.2796}, 2012.

\bibitem{Li-Bland14}
D.~Li-Bland.
\newblock Pseudo-{D}irac structures.
\newblock {\em Indag. Math. (N.S.)}, 25(5):1054--1101, 2014.

\bibitem{LiWeXu97}
Z.-J. Liu, A.~Weinstein, and P.~Xu.
\newblock Manin triples for {L}ie bialgebroids.
\newblock {\em J. Differential Geom.}, 45(3):547--574, 1997.

\bibitem{Mackenzie05b}
K.~C.~H. Mackenzie.
\newblock Duality and triple structures.
\newblock In {\em The breadth of symplectic and {P}oisson geometry}, volume 232
  of {\em Progr. Math.}, pages 455--481. Birkh\"auser Boston, Boston, MA, 2005.

\bibitem{Mackenzie05}
K.~C.~H. Mackenzie.
\newblock {\em General {T}heory of {L}ie {G}roupoids and {L}ie {A}lgebroids},
  volume 213 of {\em London Mathematical Society Lecture Note Series}.
\newblock Cambridge University Press, Cambridge, 2005.

\bibitem{Mackenzie11}
K.~C.~H. Mackenzie.
\newblock Ehresmann doubles and {D}rinfel'd doubles for {L}ie algebroids and
  {L}ie bialgebroids.
\newblock {\em J. Reine Angew. Math.}, 658:193--245, 2011.

\bibitem{Pradines77}
J.~Pradines.
\newblock {\em Fibr\'es vectoriels doubles et calcul des jets non holonomes},
  volume~29 of {\em Esquisses Math\'ematiques [Mathematical Sketches]}.
\newblock Universit\'e d'Amiens U.E.R. de Math\'ematiques, Amiens, 1977.

\bibitem{Roytenberg99}
D.~Roytenberg.
\newblock {\em Courant algebroids, derived brackets and even symplectic
  supermanifolds}.
\newblock ProQuest LLC, Ann Arbor, MI, 1999.
\newblock Thesis (Ph.D.)--University of California, Berkeley.

\bibitem{Roytenberg02}
D.~Roytenberg.
\newblock On the structure of graded symplectic supermanifolds and {C}ourant
  algebroids.
\newblock In {\em Quantization, {P}oisson brackets and beyond ({M}anchester,
  2001)}, volume 315 of {\em Contemp. Math.}, pages 169--185. Amer. Math. Soc.,
  Providence, RI, 2002.

\bibitem{ShZh17}
Y.~Sheng and C.~Zhu.
\newblock Higher extensions of {L}ie algebroids.
\newblock {\em Commun. Contemp. Math.}, 19(3):1650034, 41, 2017.

\bibitem{Uchino02}
K.~Uchino.
\newblock {Remarks on the definition of a Courant algebroid.}
\newblock {\em Lett. Math. Phys.}, 60(2):171--175, 2002.

\end{thebibliography}

\def\cprime{$'$} \def\polhk#1{\setbox0=\hbox{#1}{\ooalign{\hidewidth
  \lower1.5ex\hbox{`}\hidewidth\crcr\unhbox0}}} \def\cprime{$'$}
  \def\cprime{$'$}

\end{document}